\documentclass[a4paper, 12pt]{amsart}
\usepackage{amsfonts, amsthm, amssymb, amsmath, stmaryrd}
\usepackage{mathrsfs,array}
\usepackage{xy}
\usepackage{hyperref}
\usepackage{verbatim}
\usepackage{amscd}
\usepackage{tikz}
\usepackage{tikz-cd}
\usetikzlibrary{matrix,arrows,decorations.pathmorphing}
\input xy
\xyoption{all}

\newtheorem{lem}{Lemma}[section]
\newtheorem{definition}[lem]{Definition}

\newtheorem{cor}[lem]{Corollary}
\newtheorem{thm}[lem]{Theorem}
\newtheorem{prop}[lem]{Proposition}

\newtheorem{conj}[lem]{Conjecture}

\theoremstyle{remark}
\newtheorem{rem}[lem]{Remark}
\newtheorem{example}[lem]{Example}

\DeclareMathOperator{\Id}{Id}
\DeclareMathOperator{\Hom}{Hom}
\DeclareMathOperator{\End}{End}

\DeclareMathOperator{\Ker}{ker}

\DeclareMathOperator{\coker}{coker}
\DeclareMathOperator{\modd}{mod}

\DeclareMathOperator{\Spec}{Spec}
\DeclareMathOperator{\Spm}{Spm}
\DeclareMathOperator{\Spf}{Spf}
\DeclareMathOperator{\Fil}{Fil}

\DeclareMathOperator{\Gal}{Gal}

\DeclareMathOperator{\Frob}{Frob}

\DeclareMathOperator{\Ver}{Ver}
\DeclareMathOperator{\Sym}{Sym}

\def\crys{\mathrm{crys}}
\def\rig{\mathrm{rig}}
\def\dR{\mathrm{dR}}

\def\St{\mathrm{St}}
\def\uSym{\underline{\Sym}}
\def\Ind{\mathrm{Ind}}
\def\Mfc{\mathrm{Mod^{fl}_{comp}(O_E)}}
\def\Fix{\mathrm{Fix}}

\newcommand{\N}{\mathbb{N}}
\newcommand{\Z}{\mathbb{Z}}
\newcommand{\F}{\mathbb{F}}
\newcommand{\Q}{\mathbb{Q}}

\newcommand{\A}{\mathbb{A}}

\newcommand{\C}{\mathbb{C}}

\newcommand{\tr}{\mathrm{tr}}
\newcommand{\GL}{\mathrm{GL}}
\newcommand{\SL}{\mathrm{SL}}
\newcommand{\PGL}{\mathrm{PGL}}

\newcommand{\g}{\begin{pmatrix} a & b\\ c&d\end{pmatrix}}
\newcommand{\w}{\begin{pmatrix} 0 & -1\\ p&0\end{pmatrix}}
\newcommand{\wo}{\begin{pmatrix} 1 & 0\\ 0&p^{-1}\end{pmatrix}}
\newcommand{\we}{\begin{pmatrix} 0 & 1\\ -1&0\end{pmatrix}}
\newcommand{\ws}{\begin{pmatrix} 1 & \frac{1}{p}\\ 0 & 1\end{pmatrix}}
\newcommand{\cA}{\mathcal{A}}
\newcommand{\cI}{\mathcal{I}}
\newcommand{\cO}{\mathcal{O}}
\newcommand{\cL}{\mathcal{L}}
\newcommand{\defeq}{\stackrel{\mathrm{def}}{=}}
\newcommand{\Fr}{\widetilde{Fr}}
\newcommand{\e}{\tilde{e}}
\newcommand{\uu}{\tilde{u}}
\newcommand{\tomega}{\tilde{\omega}}
\newcommand{\PP}{\mathbb{P}}
\newcommand{\indkg}{\Ind_{\GL_2(\Z_p)\Q_p^\times}^{\GL_2(\Q_p)}}
\newcommand{\ebf}{\mathbf{e}}
\newcommand{\wtso}{\widetilde{\Sigma_{1,O_F}}}
\newcommand{\whso}{\widehat{\Sigma_{1,O_F}}}
\newcommand{\whsow}{\widehat{\Sigma_{1,O_F}^{(0)}}}
\newcommand{\wtx}{\widetilde{\Sigma_{1,O_F,\xi}}}
\newcommand{\wtxs}{\widetilde{\Sigma_{1,O_F,s,\xi}}}
\newcommand{\wts}{\widetilde{\Sigma_{1,O_F,s}}}
\newcommand{\wtxp}{\widetilde{\Sigma_{1,O_F,s',\xi}}}
\newcommand{\wtp}{\widetilde{\Sigma_{1,O_F,s'}}}

\newcommand{\wtxc}{\widetilde{\Sigma_{1,O_F,s'_0,\xi}}}
\newcommand{\wtc}{\widetilde{\Sigma_{1,O_F,s'_0}}}

\newcommand{\wtwx}{\widetilde{\Sigma_{1,O_F,s_0,\xi}}}

\newcommand{\wtxe}{\widetilde{\Sigma_{1,O_F,e,\xi}}}

\newcommand{\wtxsp}{\widetilde{\Sigma_{1,O_F,[s,s'],\xi}}}
\newcommand{\wtsp}{\widetilde{\Sigma_{1,O_F,[s,s']}}}
\newcommand{\wtsow}{\widetilde{\Sigma_{1,O_F}^{(0)}}}
\newcommand{\X}{\mathcal{X}}

\newcommand{\RNum}[1]{\uppercase\expandafter{\romannumeral #1\relax}}

\newcommand{\overbar}[1]{\mkern 1.5mu\overline{\mkern-1.5mu#1\mkern-1.5mu}\mkern 1.5mu}

\begin{document}

\title[First covering of Drinfel'd upper half plane]{First covering of Drinfel'd upper half plane and Banach representations of $\GL_2(\Q_p)$}
\author{Lue Pan}
\address{Department of Mathematics, Princeton University, Fine Hall, Washington Road, Princeton, NJ 08540}
\email{lpan@math.princeton.edu}
\begin{abstract} For an odd prime $p$, we construct some admissible Banach representations of $\GL_2(\Q_p)$ that conjecturally should correspond to some $2$-dimensional tamely ramified, potentially Barsotti-Tate representations of $\Gal(\overbar{\Q_p}/\Q_p)$ via the $p$-adic local Langlands correspondence. To achieve this, we generalize Breuil's work in the semi-stable case and work on the first covering of Drinfel'd upper half plane. Our main tool is an explicit semi-stable model of the first covering.
\end{abstract}

\maketitle

\tableofcontents

\section{Introduction}
In the paper \cite{Bre}, Breuil constructed some Banach representations of $\GL_2(\Q_p)$, which conjecturally should be non-zero and admissible and correspond to $2$-dimensional semi-stable, non crystalline representations of $G_{\Q_p}$ under the $p$-adic local Langlands correspondence. Here $G_{\Q_p}=\Gal(\overbar{\Q_p}/\Q_p)$, where $\overbar{\Q_p}$ is some fixed algebraic closure of $\Q_p$. Later on in \cite{Colm2}, Colmez found the relationship between these Banach representations and $(\phi,\Gamma)$-modules and proved their admissibility. Breuil and M\'ezard also proved the admissibility in some cases by explicitly computing the mod $p$ reductions of these Banach representations in \cite{BM}. The aim of this paper is to generalize Breuil's work to some $2$-dimensional tamely ramified, potentially Barsotti-Tate representations of $G_{\Q_p}$.

First we recall some of Breuil's construction in \cite{Bre}. Let $E$ be a finite extension of $\Q_p$ and $k$ an integer greater than $2$. Up to a twist by some character, all $2$-dimensional semi-stable, non crystalline $E$-representations of $G_{\Q_p}$ with Hodge-Tate weights $(0,k-1)$ are classified by the `$\cL$-invariant' (Exemple 1.3.5. of \cite{Bre}).  We use $V(k,\cL)$ to denote this Galois representation. Here $\cL$ is an element in $E$ and basically tells you the position of the Hodge filtration on the Weil-Deligne representation associated to $V(k,\cL)$. Notice that this Weil-Deligne representation does not depend on $\cL$. So via the classical Local Langlands correspondence, all $V(k,\cL)$ correspond to the same smooth representation of $\GL_2(\Q_p)$, which is a twist of $\St$, the usual Steinberg representation.

The idea of Breuil is that for each $\cL$, there should exist a $\GL_2(\Q_p)$-invariant norm on $\uSym^{k-2}E^2\otimes \St$, where $\uSym^{k-2}E^2$ is a twist of the algebraic representation $\Sym^{k-2}E^2$. Different $\cL$ should give different non commensurable unit balls of $\uSym^{k-2}E^2\otimes \St$. If we take the completion, we get a Banach representation $B(k,\cL)$ of $\GL_2(\Q_p)$ for each $\cL$. Moreover, we hope this representation is admissible in the sense of \cite{ST} and the correspondence between  $V(k,\cL)$ and $B(k,\cL)$ is compatible with the mod $p$ correspondence defined by Breuil in \cite{Bre3}.

So how to construct these $B(k,\cL)$? For simplicity, I assume $E=\Q_p$ and $k$ is even. The strategy of Breuil is to realize the unit ball $O(k,\cL)^U$ of the dual representation of $B(k,\cL)$ in $O(k)=\Gamma(\Omega,\cO(k))$, where $\cO(k)$ is a coherent sheaf on the Drinfel'd upper half plane $\Omega$ over $\Q_p$. Concretely, $\cO(2)$ is the sheaf of rigid differential forms and $\cO(2n)=\cO(2)^{\otimes n}$. Here $\Omega$ is considered as a rigid analytic space and $\GL_2(\Q_p)$ acts on everything. We note that the de Rham cohomology of $\Omega$ is nothing but $\St^\vee$, the algebraic dual representation of $\St$ (Theorem 1 of \cite{SS}). The construction of $O(k,\cL)^U$, as far as I understand, has the following two important properties:
\begin{enumerate}
\item $O(k,\cL)^U$ is `globally bounded' and hence compact. In other words, it is contained in $\Gamma(\widehat{\Omega},\omega^{\otimes k/2})$, where $\widehat{\Omega}$ is a semi-stable model of $\Omega$ and $\omega$ is an integral structure of $\cO(2)$. This guarantees that the dual of $O(k,\cL)^U$ is indeed a Banach representation (after inverting $p$).
\item If $f\in O(k)$ comes from a modular form of weight $k$ (see section 5 of \cite{Bre} for the precise meaning), then $f\in O(k,\cL_0)^U$ if and only the $\cL$-invariant of $f$ is $\cL_0$.
\end{enumerate}

Now consider the case where the Galois representation is tamely ramified. We will see later that the situation is very similar. Fix $E$ a finite extension of $\Q_p$ large enough and let $O_E$ be its ring of integers. This time we need to work on the first covering of Drinfel'd upper half plane. According to Drinfel'd, there is a universal $p$-divisible group $X$ over $\widehat{\Omega}\hat{\otimes}\widehat{\Z_p^{nr}}$ and $O_D$ acts on it, where $O_D$ is the ring of integers inside the quaternion algebra $D$ over $\Q_p$. Fix an uniformizer $\Pi\in O_D$ and define $\X_n$ as the generic fibre of $X[\Pi^n]$. The first covering $\Sigma_1=\X_1- \X_0$, also carries the action of $\GL_2(\Q_p)$ and $O_D^\times$. It was shown by Drinfel'd in \cite{Dr} that the action of $O_D^\times$ can be extended to $D^\times$.  This is a left action and we will keep this convention in this paper unless explicitly inverting it. One remark is that the action of $\Q_p^\times$ inside $D^\times$ and $\GL_2(\Q_p)$ become the same once we invert the action of $D^\times$.

First we note that the ($E$-coefficient) de Rham cohomology $H_{\dR}^1(\Sigma_1,E)\stackrel{def}{=}H^1_{\dR}(\Sigma_1)\otimes_{\Q_p} E$ of $\Sigma_1$ has the following decomposition. Let $\psi:\Q_p^\times\to O_E^\times$ be a unitary character of level $0$ in the sense that $1+p\Z_p$ is contained in the kernel of $\psi$. We will view it as a character of $\Q_p^\times\subset D^\times$. In the following theorem, we invert the action of $D^\times$ so that it acts on the cohomology on the left. We denote the $\psi$-isotypic component of $H^1_{\dR}(\Sigma_1,E)$ by $H^1_{\dR}(\Sigma_1,E)^\psi$.

\begin{thm} \label{gs}
As a representation of $D^\times \times \GL_2(\Q_p)$,
\[H^1_{\dR}(\Sigma_1,E)^\psi\simeq\bigoplus_{\pi\in \mathcal{A}^0(D^\times)(\psi^\vee)_0}(\pi\otimes JL(\pi))^{\vee}\otimes_E D_{\pi},\]
where  ${}^\vee$ denotes the algebraic dual representation, $\mathcal{A}^0(D^\times)(\psi^\vee)_0$ is the space of admissible irreducible representations of $D^\times$ of level $0$ over $E$ that are not characters and with central character $\psi^\vee$ (see Chapter 13 of \cite{BH}), $JL(\pi)$ is the representation of $\GL_2(\Q_p)$ associated to $\pi$ by the Jacquet-Langlands correspondence, $D_{\pi}$ is a two-dimensional vector space over $E$.
\end{thm}

\begin{rem}
In fact, we can define more structures on $D_{\pi}$. Roughly speaking, we may find a finite extension $F$ of $\Q_p$ such that
\[F\otimes_{\Q_p}D_{\pi}\simeq F\otimes_{F_0} D_{\crys,\pi},\]
where $F_0$ is the maximal unramified extension of $\Q_p$ inside $F$ and $D_{\crys,\pi}$ is a $(\varphi,N,F/\Q_p,E)$-modules (see the notations here in section \ref{Gal}). Then up to some unramified character, the Weil-Deligne representation associated to $D_{\crys,\pi}$ corresponds to $JL(\pi)$ under the classical local Langlands correspondence. See theorem \ref{mthi} below.
\end{rem}

Explicitly, any $\pi\in \mathcal{A}^0(D^\times)(\psi^\vee)_0$ is an induced representation:
\[\pi\simeq\Ind_{O_D^\times\Q_p^\times}^{D^\times}\Xi,\]
where $\Xi:O_D^\times\Q_p^\times\to O_E^\times$ is a character which extends $\psi^\vee$ and trivial on $1+\Pi O_D$. It is clear that $\pi$ has an integral structure $\pi_0$ over $O_E$.

As we noted before, we need to construct a $\GL_2(\Q_p)\times D^\times$-equivariant formal model $\widehat{\Sigma_1^{nr}}$ of $\Sigma_1$. This will be done by using Raynaud's theory of $\F$-vector space schemes. As Breuil did in the case of Drinfel'd upper half plane, we can define a $\GL_2(\Q_p)\times D^\times$-equivariant integral model $\omega^1$ of $\Omega^1_{\Sigma_1}$ on this formal model, where $\Omega^1_{\Sigma_1}$ is the sheaf of differential forms (see remark \ref{o1F0}). Consider the composition of the following maps:
\[H^0(\widehat{\Sigma_1^{nr}},\omega^1)\to H^0(\Sigma_1,\Omega^1_{\Sigma_1}) \to H^1_{\dR}(\Sigma_1).\]

We will show that this map is injective (proposition \ref{inj}), so that $H^0(\widehat{\Sigma_1^{nr}},\omega^1)$ can be viewed as a subspace in the de Rham cohomology. Rewrite theorem \ref{gs} as:
\[H^1_{\dR}(\Sigma_1,E)(\pi^\vee)=H^1_{\dR}(\Sigma_1,E)^\psi(\pi^{\vee})\simeq JL(\pi)^\vee\otimes D_{\pi},\]
where $(*)(\pi^\vee)=\Hom_{E[D^\times]}(\pi^\vee,*)$. For any line $\cL$ inside $D_{\pi}$ (the $\cL$-invariant in our case), we may view $JL(\pi)^\vee \otimes \cL$ as a subspace inside  $H^1_{\dR}(\Sigma_1,E)(\pi^\vee)$ by the above isomorphism. We can now define the (dual) of our Banach space representations:

\begin{definition} \label{dmi}
\[M(\pi,\cL)\stackrel{def}{=}(H^0(\widehat{\Sigma_1^{nr}},\omega^1)\otimes_{\Z_p} O_E)(\pi_0^\vee) \cap (JL(\pi)^\vee \otimes \cL).\]
\end{definition}

Recall that $\pi_0$ is some integral structure of $\pi$. Notice that $M(\pi,\cL)$ is contained in $(H^0(\widehat{\Sigma_1^{nr}},\omega^1)\otimes_{\Z_p} O_E)(\pi^\vee)$, which is a natural subspace of $(H^0(\Sigma_1,\Omega^1_{\Sigma_1})\otimes_{\Q_p}E)(\pi^\vee)$. This last space has a natural Fr\'echet space structure over $E$. The induced topology on $M(\pi,\cL)$ makes it into a compact topological space, and thus allows us to define:

\begin{definition} \label{dbi}
$B(\pi,\cL)=\Hom^{cont}_{O_E}(M(\pi,\cL),E)$.
\end{definition}
This is a unitary representation of $\GL_2(\Q_p)$.

\begin{rem}
The argument of Lemme 4.1.1. in \cite{Bre} shows that $H^0(\widehat{\Sigma_1^{nr}},\omega^1)$ and  $H^0(\widehat{\Sigma_1^{nr}},\omega'^1)$ are commensurable, where $\omega'^1$ is any other $\GL_2(\Q_p)\times D^\times$-equivariant integral model of $\Omega^1_{\Sigma_1}$. Hence $B(\pi,\cL)$ is independent of the choice of $\omega^1$.
\end{rem}

Now we can state the main result of this paper. Assume $p$ is an odd prime.

\begin{thm} \label{mthm3}
\begin{enumerate}
\item $B(\pi,\cL)$ is non-zero and admissible as a representation of $\GL_2(\Q_p)$. In fact, its mod $p$ reduction can be computed explicitly.
\item $B(\pi,\cL)$ is a unitary completion of $JL(\pi)$.
\end{enumerate}
\end{thm}

The computation will give us an interesting $\GL_2(\Q_p)$-equivariant short exact sequence (corollary \ref{by1}, \ref{by2}):
\begin{cor} \label{intcor}
\[0\to \widehat{JL(\pi)} \to H^0(\widehat{\Sigma_1^{nr}},\omega^1)_E^d(\pi) \to B(\pi,\cL) \to0,\]
where $\widehat{JL(\pi)}$ is the universal unitary completion of $JL(\pi)$ (see \cite{Eme}), and 
\[H^0(\widehat{\Sigma_1^{nr}},\omega^1)_E^d=\Hom^{cont}_{\Z_p}(H^0(\widehat{\Sigma_1^{nr}},\omega^1),E).\]
\end{cor}

Note that the kernel and the middle term are independent of $\cL$ while the map between them depends on $\cL$.

\begin{rem}
Unfortunately, we have to assume $p\geq 3$ in the proof of theorem \ref{mthm3} (for example in the proof of lemma \ref{tchl1} ). However theorem \ref{gs} is also true for $p=2$.
\end{rem}

Now we explain the strategy of proving theorem \ref{gs}. By twisting with some unramified unitary characters, it suffices to deal with the case  where the central character $\psi$ satisfies $\psi(p)=1$. This suggests us to descend  $\Sigma_1$ from $\widehat{\Q_p^{nr}}$ to $\Q_{p^2}$, the unramified quadratic extension of $\Q_p$, by taking the '$p$-invariant' of $\Sigma_1$ (see section \ref{secact}).  We use $\Sigma_1^p$ to denote this rigid analytic space. One warning here: even though $\Sigma_1^p$ has a structure map to $\Q_{p^2}$, I will view it as a rigid space over $\Q_p$. A semi-stable model of $\Sigma_1^p$ is very helpful (see theorem \ref{mst}): 

\begin{thm} 
$\Sigma_1^p\times_{\Q_p}F$ has an explicit $D^\times \times \GL_2(\Q_p)$-equivariant semi-stable model $\widehat{\Sigma_{1,O_F}^{(0)}}$ over $O_F$, where $F\simeq\Q_{p^2}[(-p)^{\frac{1}{p^2-1}}]$.
\end{thm}

Similar results have been obtained before by Teitelbaum in \cite{Tei2}. 

Denote the generic fibre of this semi-stable model by $\Sigma_{1,F}^{(0)}=\Sigma_1^p\times_{\Q_p} F$. With the help of the semi-stable model, we can compute its de Rham cohomology. Let $\chi(E)$ be the character group of $O_D^\times/(1+\Pi O_D)$ with values in $E^\times$. Recall that $O_D^\times$ acts on $\Sigma_{1,F}^{(0)}$. We have the following result (see section \ref{F_0s}, especially corollary \ref{mnc} and remark \ref{smdual}).

\begin{thm} \label{mthi}
For any $\chi\in\chi(E)$ such that $\chi\neq\chi^p$, we have a $\Gal(F/\Q_p)\times O_D^\times \times \GL_2(\Q_p)$-equivariant isomorphism:
\begin{eqnarray*}
F\otimes_{F_0}D_{\crys,\chi}\otimes_E (c-\indkg\rho_{\chi^{-1}})^\vee\stackrel{\sim}{\to} (H^1_{\dR}(\Sigma_{1,F}^{(0)})\otimes_{\Q_p}E)^{\chi},
\end{eqnarray*}
where $F_0\simeq\Q_{p^2}$, $c-\Ind$ is the induction with compact support, $^\vee$ means the algebraic dual, $\rho_{\chi^{-1}}$ is a cuspidal representation of $\GL_2(\F_p)$ over $E$ defined via Deligne-Lusztig theory, and $D_{\crys,\chi}$ is a free $F_0\otimes E$-module of rank $2$ with an action of $\Gal(F/\Q_p)$. In addition, we can define a Frobenius operator $\varphi$ acting on it. It is explicitly described in proposition \ref{dcrys}.
\end{thm}

Take $\pi=\Ind_{O_D^\times\Q_p^\times}^{D^\times}\chi$, where $\chi$ is viewed as a character of $O_D^\times\Q_p^\times$ trivial on $p$. Then $JL(\pi)\simeq c-\indkg\rho_{\chi^{-1}}$. Theorem \ref{gs} follows from the above theorem by taking the $\Gal(F/\Q_p)$-invariants. There is some inverse involved since we invert the action of $D^\times$ in theorem \ref{gs}. 

It is clear from the theorem that $D_{\crys,\chi}$ is a $(\varphi,N,F/\Q_p,E)$-module. A line $\cL$ inside $D_{\pi}$, or equivalently, a $\Gal(F/\Q_p)$-invariant `line' inside $F\otimes_{F_0} D_{\crys,\chi}$, essentially gives a filtration and makes $D_{\crys,\chi}$ into a filtered $(\varphi,N,F/\Q_p,E)$-module. See section \ref{Gal} for more details.

After fixing some basis $D_{crys,\chi}$ (see proposition \ref{dcrys}), any line $\cL$ can be identified with an element $b$ inside $E$ or $\infty$. Assume $b\in O_E$ for the moment. We will write 
\[M(\chi,[1,b])=M(\Ind_{O_D^\times\Q_p^\times}^{D^\times}\chi,\cL)\]
and similarly $B(\chi,[1,b])=B(\Ind_{O_D^\times\Q_p^\times}^{D^\times}\chi,\cL)$ in the paper.

Some notations here. Fix a $\Z_p$-linear embedding of $W(\F_{p^2})$, the Witt vector of $\F_{p^2}$ into $O_D$. Then any $\chi\in\chi(E)$ can be viewed as a character of $\F_{p^2}^\times$ by composing this embedding with the Teichm\"uller character. Also fix an embedding $\tau$ of $W(\F_{p^2})$ into $E$. Similarly the Teichm\"uller character gives us a character $\chi_{\tau}:\F_{p^2}^\times\to E^\times$.

\begin{definition}
We define $m$ as the unique integer in $\{0,\cdots,p^2-2\}$ such that 
\[\chi=\chi_{\tau}^{-m}:\F_{p^2}^\times\to O_E^\times.\]
We will write $m=i+(p+1)j$, where $i\in\{0,\cdots,p\},j\in\{0,\cdots,p-2\}$ and $[-mp]$ as the unique integer in $\{0,\cdots,p^2-2\}$ congruent to $-mp$ modulo $p^2-1$.
\end{definition}

Let $\sigma_i(j)$ be the following representation of $\GL_2(\Q_p)$
\[\sigma_i(j)=\indkg (\Sym^{i}\F_p^2)\otimes O_E/p\otimes \det{}^j,\]
where $\Sym^i\F_p^2$ is the $i$-th symmetric power of the natural representation of $\GL_2(\F_p)$ on the canonical basis of $\F_p^2$, viewed as a representation of $\GL_2(\Z_p)\Q_p^\times$ trivial on $p^\Z$.

Using our explicit semi-stable model, we can compute the mod $p$ reduction of $M(\chi,[1,b])$ (corollary \ref{mc1}, remark \ref{mc3}, corollary \ref{mc2}).

\begin{thm}
\begin{enumerate}
\item Assume $p^2-1-m\geq [-mp],i\in\{2,\cdots,p-1\}$. As a representation of $\GL_2(\Q_p)$,\begin{eqnarray*}
0\to\left\{ \begin{array}{ll}X\in \sigma_{i-2}(j+1),\\ c(\chi,b)X=T(X)\end{array}\right\}
\to M(\chi,[1,b])/p\to 
\left\{\begin{array}{ll}X\in \sigma_{p-1-i}(i+j),\\ X=c(\chi,b)T(X)\end{array}\right\}
\to 0,
\end{eqnarray*}
\item Assume $p^2-1-m\leq [-mp],i\in\{2,\cdots,p-1\}$.
\begin{eqnarray*}
0\to\left\{ \begin{array}{ll}X\in \sigma_{p-1-i}(i+j),\\ X=c(\chi,b)T(X)\end{array}\right\}
\to M(\chi,[1,b])/p\to 
\left\{\begin{array}{ll}X\in \sigma_{i-2}(j+1),\\ c(\chi,b)X=T(X)\end{array}\right\}
\to 0.
\end{eqnarray*}
\item Assume $i=p$,
\[M(\chi,[1,b])/p\simeq \{X\in \sigma_{p-2}(j+1),-c(\chi,b)X+T(X)-c(\chi,b)T^2(X)=0\},\]
\item Assume $i=1$,
\[M(\chi,[1,b])/p\simeq \{X\in\sigma_{p-2}(j+1),  X+c(\chi,b)T(X)+T^2(X)=0\},\]
\end{enumerate}
where $T$ is the usual Hecke operator defined in \cite{Bre2} and $c(\chi,b)=(-1)^{j+1}\tau(w_1^{-i})b\in O_E/p$. Here $\tau(w_1)$ satisfies $\tau(w_1)^{p+1}=-1$ and is independent of $\chi,\cL$. Thus in any case, $B(\chi,[1,b])$ is non-zero and admissible.
\end{thm}

\begin{rem}
In a recent paper \cite{DLB}, Gabriel Dospinescu and Arthur-C\'esar Le Bras independently use a very similar method to construct some locally analytic representations of $\GL_2(\Q_p)$ and verify the compatibilites with the $p$-adic local Langlands correspondence, thus generalize Breuil's work \cite{Bre} in this direction. Their method works for all the coverings of Drinfel'd upper half plane and relies on the previous work of Colmez on the relationship between Banach space representations and $(\phi,\Gamma)$-modules. Combining with some known results of $p$-adic local Langlands correspondence, they can also prove theorem \ref{gs} and theorem \ref{mthm3}.  However, it seems that corollary \ref{intcor} does not follow directly from their work.
\end{rem}

We give a brief outline of this paper. The goal of the first $8$ sections (\ref{bfd}-\ref{act}) is to explicitly write down a semi-stable model of $\Sigma_1$. Our strategy is to apply Raynaud's theory of $\F$-vector spaces schemes (\cite{Ray}) to $X_1$. We will collect some basic facts about Drinfel'd upper half plane in section \ref{bfd} and review Raynaud's theory in section \ref{rayppp}. To compute the data in Raynaud's theory, we need the existence of some `polarization' of $X_1$ (proposition \ref{forp}), which comes from a formal polarization of $X$ (section \ref{fp}). Using this, a formal model is obtained in section \ref{soxfd}. By comparing the invariant differential forms of $X_1$ computed in two different ways, we can write down the local equation of this formal model in section \ref{leox}. From this, it's not too hard to work out a semi-stable model in section \ref{ssm} and make clear how $\GL_2(\Z_p),O_D^\times, \Gal(F/\Q_p)$ acts on it in section \ref{act}.

In section \ref{adoc}, \ref{dR}, \ref{F_0s}, we compute the de Rham cohomology of $\Sigma_{1,F}^{(0)}$. Using our semi-stable models, this can be expressed by the crystalline cohomology of the irreducible components of the special fibre, which is well-understood via Deligne-Lusztig theory. The main result is corollary \ref{mnc} which describes the structure of the de Rham cohomology.

In section \ref{Gal}, we classify all possible filtrations on $D_{\crys,\chi}$ with Hodge-Tate weights $(0,1)$. We use this result to define $M(\chi,[1,b])$ in section \ref{cb}.

Section \ref{cmodp}, \ref{M1}, \ref{M2} contain the computation of $M(\chi,[1,b])/p$. In section \ref{cmodp}, we compute $H^0(\widehat{\Sigma_1^{nr}},\omega^1)/p$ (not exactly this space, see the precise statement there). Roughly speaking, the method is by carefully studying the shape of differential forms on each irreducible component of the special fibre. The main result is proposition \ref{dmp} which says that this space is an extension of two inductions. Section \ref{M1} and section \ref{M2} treat different cases of computations of $M(\chi,[1,b])/p$ according to the value of $i$, but their strategies are the same: first we interpret $M(\chi,[1,b])$ as the kernel of a map $\theta_b$ from $(H^0(\widehat{\Sigma_1^{nr}},\omega^1)\otimes O_E/p)^{\chi}$ to a space $J_2$, then we compute the mod $p$ reduction $\bar{\theta}_b$ of this map explicitly and show that $\bar{\theta}_b$ is in fact surjective. Hence $\theta_b$ has to be surjective as well since both $H^0(\widehat{\Sigma_1^{nr}},\omega^1)$ and $J_2$ are $p$-adically complete. Therefore $M(\chi,[1,b])$ is just the kernel of $\bar{\theta}_b$.

\subsection*{Notations} 
Throughout this paper, fix an odd prime number $p$.

Let $\Q_p^{nr}$ be the maximal unramified extension of $\Q_p$ and $\widehat{\Q_p^{nr}}$ be its $p$-adic completion. We will write $\Z_{p^2}=W(\F_{p^2})$, the ring of Witt vectors of $\F_{p^2}$ and fix an embedding of it into $\Q_p^{nr}$. Denote the fractional field of $\Z_{p^2}$ as $\Q_{p^2}$. We use $F_0$ to denote the unique unramified quadratic extension of $\Q_p$. Hence the fixed embedding of $\Q_{p^2}$ into $\Q_{p}^{nr}$ gives us an isomorphism between $F_0$ and $\Q_{p^2}$. Later on, $F_0$ will appear as some intermediate field extension when we try to compute a semi-stable model. Let $O_{F_0}$ be the ring of integers inside $F_0$. Frequently we will identify $F_0$ with $\Q_{p^2}$ by this fixed isomorphism.

We denote $D$ as the quaternion algebra of $\Q_p$ and fix a uniformizer $\Pi\in D$ such that $\Pi^2=p$. We will also fix a $\Z_p$-linear embedding of $\Z_{p^2}$ into $O_D$, hence an isomorphism:
\[O_D/\Pi O_D\simeq \F_{p^2}.\]

Let $E$ be a finite extension of $\Q_p$ such that $\Hom_{\Q_p}(F_0,E)\neq 0$. We use $\tau,\bar{\tau}$ to denote the embeddings of $F_0$ into $E$ and $O_E$ to denote its ring of integers. For any $O_{F_0}$-module $A$, we denote $A\otimes_{O_{F_0},\tau}O_E$ by $A_{\tau}$ and $A\otimes_{O_{F_0},\bar{\tau}}O_E$ by $A_{\bar{\tau}}$.

For $K=E,F_0$, we use $\chi(K)$ to denote the character group of $O_D^\times/(1+\Pi O_D)=(O_D/\Pi)^\times$ with values in $K^\times$, 

For any integer $n$, we will use $[n]$ to denote the unique integer in $\{0,1,\cdots,p^2-2\}$ congruent to $n$ modulo $p^2-1$.

For any ring $A$ and integer $n$, we use $\mu_n(A)$ to denote $\{a\in A, a^n=1\}$.

For any abelian group $M$, we denote the $p$-adic completion of $M$ by $M~\widehat{}$.

We use $\Sym^i\F_p^2$ to denote the $i$-th symmetric power of the natural representation of $\GL_2(\F_p)$ on the canonical basis of $\F_p^2$ for $i$ non-negative. Explicitly, we can identify $\Sym^i\F_p^2$ with $\oplus_{r=0}^i\F_p x^r y^{i-r}$ where the action is given by:
\[ \g x^i y^{i-r}=(ax+cy)^r(bx+dy)^{i-r}.\]
Sometimes we will also view it as a representation of $\GL_2(\Z_p)$ by abuse of notations.

Also, we define an induced representation of $\GL_2(\Q_p)$:
\[\sigma_i=\indkg (\Sym^{i}\F_p^2)\otimes O_E/p,\]
where the induction has no restriction on the support and we view $\Sym^i\F_p^2$ as a representation of $\GL_2(\Z_p)\Q_p^\times$ trivial on $p^{\Z}$. We define $\sigma_{-1}$ as $0$ and
\[\sigma_i(j)=\indkg (\Sym^{i}\F_p^2)\otimes O_E/p \otimes \det{}^j,\]
where $\det{}$ is the determinant map. 

We recall the definition of Hecke operator $T$ here. See 3.2 of \cite{Bre2} for more details. Let $\sigma=\Sym^{r}\F_p^2\otimes \det{}^m,~0\leq r\leq p-2$ be an irreducible representation of $\GL_2(\F_p)$ over $\F_p$. I would like to view it as a representation of $\GL_2(\Z_p)\Q_p^{\times}$ with $p$ acting trivially. We use $V_{\sigma}$ to denote the underlying representation space. Hence,
\[\indkg\sigma=\{f:\GL_2(\Q_p) \to V_{\sigma}, f(hg)=\sigma(h)(f(g)),h\in\GL_2(\Z_p)\Q_p^{\times}\}.\]
We note that we put no restriction on the support. Following \cite{Bre2}, denote by 
\[[g,v]:\GL_2(\Q_p)\to V_{\sigma} \]
the following element of $\indkg\sigma$:
\begin{eqnarray*}
[g,v](g')&=\sigma(g'g)v ~&\mbox{if }g'\in \GL_2(\Z_p)\Q_p^{\times}g^{-1} \\ {}
[g,v](g')&=0 ~&\mbox{if }g'\notin \GL_2(\Z_p)\Q_p^{\times}g^{-1}.
\end{eqnarray*}
We have $g([g',v])=[gg',v]$ and $[gh,v]=[g,\sigma(h)v]$ if $h\in\GL_2(\Z_p)\Q_p^{\times}$. It is clear that every element in $\indkg\sigma$ can be written uniquely as an infinite sum of $[g_i,v_i]$ such that no two $g_i$ are in the same coset $\GL_2(\Q_p)/\GL_2(\Z_p)\Q_p^\times$. Identify $V_{\sigma}$ with $\oplus_{k=0}^{r}\F_px^ky^{r-k}$. We define $\varphi_{r}:\GL_2(\Q_p)\to \End_{\F_p}(V_{\sigma},V_{\sigma})$  as follows:
\begin{enumerate} \label{hecke}
\item {}$\varphi_{r}(g)=0$ if  $g\notin\GL_2(\Z_p)\Q_p^{\times}\wo\GL_2(\Z_p)$.
\item {}$\varphi_{r}(\wo)(x^ky^{r-k})=0$, if $k\neq 0$.
\item {}$\varphi_{r}(\wo)(y^r)=y^r$.
\item {}$\varphi_r(h_1gh_2)=\sigma(h_1)\circ\varphi_r(g)\circ\sigma(h_2)$, $h_1,h_2\in\GL_2(\Z_p)\Q_p^\times$.
\end{enumerate}
The Hecke operator $T_{\varphi_{r}}$ (or $T$ for simplicity) is defined as:
\[T([g,v])=\sum_{g'\GL_2(\Z_p)\Q_p^{\times}\in \GL_2(\Q_p)/(\GL_2(\Z_p)\Q_p^{\times})}[gg',\varphi_r(g'^{-1})(v)].\]

\subsection*{Acknowledgement} I heartily thank my advisor Richard Taylor for suggesting this problem and his constant encouragement and numerous discussions. I also wish to thank Gabriel Dospinescu, Yongquan Hu, Ruochuan Liu for helpful discussions and answering my questions. Finally, I would like to thank the anonymous referees for their careful work. The author was partially supported by NSF grant DMS-1252158.

\section{Some facts about Drinfel'd upper half plane}\label{bfd}
Let $\Omega$ be the $p$-adic upper half plane (or Drinfel'd upper half plane) over $\Q_p$. It is a rigid analytic space over $\Q_p$ and its $\C_p$-points are $\C_p-\Q_p$, where $\C_p$ is the completion of an algebraic closure of $\Q_p$. There is a right action of $\GL_2(\Q_p)$ on $\Omega$. On the set of $\C_p$-points, it is given by $z\mapsto z|_g=\frac{az+c}{bz+d}$, for $g=\g\in\GL_2(\Q_p)$.

$\Omega$ has a $\GL_2(\Q_p)$-invariant formal model $\widehat{\Omega}$ over $\Z_p$, which is described in details in \cite{Bou-Ca}. One warning here: in this paper, the action of $\GL_2(\Q_p)$ on $\Omega$ is a right action rather than a left action used in Drinfel'd's original paper \cite{Dr} and \cite{Bou-Ca}. Our action is the inverse of their action. I apologize here if this causes any confusion.

Let me recall some facts we need to use later. There exists an open covering $\{ \widehat{\Omega}_e \}_e$ on $\widehat{\Omega}$ indexed by the set of edges of Bruhat-Tits tree $I$ of $\PGL_2(\Q_p)$. Two different $\widehat{\Omega}_e$ and $\widehat{\Omega}_{e'}$ have non empty intersection if and only $e$ and $e'$ share a vertex $s$. When this happens, $\widehat{\Omega}_e \cap \widehat{\Omega}_{e'}$ only depends on the vertex $s$. We call it $\widehat{\Omega}_s$.  For two adjacent vertices $s,s'$, we denote the unique edge connecting them by $[s,s']$. Explicitly, (~$\widehat{}~$ is for $p$-adic completion) 
\begin{eqnarray} \label{isom2}
\widehat{\Omega}_{s'}\simeq \Spf O_{\eta} \defeq \Spf\Z_p[\eta,\frac{1}{\eta-\eta^p}]~\widehat{} \\ \label{isom3}
\widehat{\Omega}_{s}\simeq\Spf O_{\zeta}\defeq\Spf\Z_p[\zeta,\frac{1}{\zeta-\zeta^p}]~\widehat{}\\ \label{isom1}
\widehat{\Omega}_{[s,s']}\simeq \Spf O_{\zeta,\eta} \defeq \Spf\frac{\Z_p[\zeta,\eta]}{\zeta\eta-p}[\frac{1}{1-\zeta^{p-1}},\frac{1}{1-\eta^{p-1}}]~\widehat{}  
\end{eqnarray}
The inclusion from $\widehat{\Omega}_s$ (resp. $\widehat{\Omega}_{s'}$) to $\widehat{\Omega}_{[s,s']}$ under these isomorphisms is just $\zeta$ (resp. $\eta$) goes to $\frac{p}{\eta}$ (resp. $\frac{p}{\zeta}$).

The set of vertices of the tree is in bijection with $\GL_2(\Z_p)\Q_p^\times\setminus \GL_2(\Q_p)$. Clearly there is a right action of $\GL_2(\Q_p)$ on this set and it extends to an action on the set of edges. In fact, this action can be identified with the action on the open covering $\{ \widehat{\Omega}_e \}_e$. When $s$ is the vertex that corresponds to the coset $\GL_2(\Z_p)\Q_p^\times$,  which I call the central vertex $s'_0$, we can choose the isomorphism \eqref{isom2} such that the action of $\GL_2(\Z_p)$ on it is given by $\eta\mapsto\frac{a\eta+c}{b\eta+d}$, for $g=\g\in\GL_2(\Z_p)$.

From the explicit description of $\widehat{\Omega}_{[s,s']}$ and $\widehat{\Omega}_s$ above, it is clear the  special fibre of $\widehat{\Omega}$ is a tree of rational curves over $\F_p$ intersecting at all $\F_p$-rational points. The set of irreducible components (singular points) is nothing but the set of vertices (edges) of the tree. The dual graph of the special fibre of $\widehat{\Omega}$ is just the Bruhat-Tits tree.

In Drinfel'd's paper \cite{Dr}, it was shown that there exists a universal family of formal groups $X$ of height $4$ over $\widehat{\Omega}\hat{\otimes}\widehat{\Z_p^{nr}}$, where $\widehat{\Z_p^{nr}}$ is the $p$-adic completion of the ring of integers inside the maximal unramified extension $\Q_p^{nr}$ of $\Q_p$. We denote by $D$ the `unique' quaternion algebra over $\Q_p$, and $O_D$ the ring integers inside $D$. Then from Drinfel'd's consctruction, we know that $O_D$ acts on the universal formal group on the left.

Fix a uniformizer $\Pi$ inside $O_D$ such that $\Pi^2=p$. Define $X_n=X[\Pi^n]$. They are finite formal group schemes over $X_0=\widehat{\Omega}\hat{\otimes}\widehat{\Z_p^{nr}}$. Let $\X_n$ be the rigid space associated to $X_n$, or equivalently, $\X_n$ is the generic fibre of $X_n$. These $\X_n$ are \'{e}tale coverings of $\X_0=\Omega\hat{\otimes}\widehat{\Q_p^{nr}}$. Then $O_D/(\Pi^n)$ acts on it and we have equivariant inclusions $\X_{n-1}\hookrightarrow\X_n$. Now define $\Sigma_n=\X_n-\X_{n-1}$. It can be shown that $\Sigma_n$ is a finite \'{e}tale Galois covering over $\X_0$ with Galois group $(O_D/(\Pi^n))^\times$. 

It is important that all the spaces ($X_n,\X_n,\Sigma_n$) we construct here have a natural $\GL_2(\Q_p)$ action and all the maps here are $\GL_2(\Q_p)$-equivariant. On $X_0=\Omega\hat{\otimes}\widehat{\Z_p^{nr}}$, $\GL_2(\Q_p)$ acts on $\widehat{\Omega}$ as we described before and acts on $\widehat{\Z_p^{nr}}$ via $\Fr^{v_p(det(g))}$, where $\widetilde Fr$ is the (lift of) arithmetic Frobenius and $v_p$ is the usual $p$-adic valuation on $\Q_p$. One can show that the action of $\Z_p^\times$ in $\GL_2(\Q_p)$ on $\Sigma_n$ is inverse to the action of $\Z_p^\times$ in $O_D^\times$. 

Now we want to describe the action of $\Pi$ on the tangent space $T$ of $X$. It is easy to see from the construction that $T$ is a rank $2$ vector bundle on $\widehat{\Omega}\hat{\otimes}\widehat{\Z_p^{nr}}$. Moreover, $T$ splits canonically into a direct sum of two line bundles $T_0,T_1$ by considering the action of $\Z_{p^2}$ inside $O_D$ (recall that we fix such an embedding in the previous section). Each eigenspace of this action is a line bundle because $X$ is `special' in the sense of Drinfel'd. $\Pi$ interchanges $T_0,T_1$ and under the isomorphism \eqref{isom2}\eqref{isom3}\eqref{isom1}, we can write it down explicitly. But before doing that, I must introduce the notion of odd and even vertices.

\begin{definition}
A class $[g]$ in $\GL_2(\Z_p)\Q_p^\times\setminus \GL_2(\Q_p)$ is called odd (resp. even) if $v_p(det(g))$ is odd (resp. even).
\end{definition}

Notice that this is well defined. And we will call a vertex in the tree (or an irreducible component of the special fibre of $\widehat{\Omega}$) is even or odd if the corresponding class is so.

Back to the discussion of tangent space. I should mention that all $T_0,T_1,\Pi$ descend naturally to $\widehat{\Omega}$, and I still call them $T_0,T_1,\Pi$ by abuse of notations. Suppose $s$ is an odd vertex and $s'$ is adjacent to $s$, which must be even. On $\widehat{\Omega}_{[s,s']}$, both $T_0,T_1$ are trivial. If we choose appropriate bases $e_0,e_1$ of them, then under the isomorphisms \eqref{isom2} \eqref{isom3} \eqref{isom1}, $\Pi$ becomes:
\begin{eqnarray} \label{strlie1}
\Pi_0:T_0\to T_1, ~e_0\mapsto \zeta e_1 \\ \label{strlie2}
\Pi_1:T_1\to T_0, ~e_1\mapsto \eta e_0
\end{eqnarray}
Identify $\Pi_0,\Pi_1$ with global sections of $T_0^*\otimes T_1$ and $T_1^*\otimes T_0$, where $T^*_i$ denotes the dual of $T_i,~i=0,1$ (the cotangent space). Then the explicit description of $\Pi$ tells us on an odd component of the special fibre, $\Pi_0$ has a simple zero at each intersection point with other irreducible components. Since each irreducible component is a rational curve over $\F_p$ and intersects other components at $\F_p$-rational points, $\Pi_0$ corresponds to the divisor that is the sum of all points of $\mathbb{P}^1(\F_p)$. On the other hand, $\Pi_1$ is zero on an odd component because  $\eta=\frac{p}{\zeta}=0$ (we are working over the special fibre, so already modulo $p$). On an even component, a similar argument shows that $\Pi_0$ is zero and $\Pi_1$ is the sum of all points of $\mathbb{P}^1(\F_p)$ as a divisor.

Restrict everything to the central vertex $s'_0$, we have an isomorphism $\widehat{\Omega}_{s'_0}\simeq\Spf\widehat{\Z_p}[\eta,\frac{1}{\eta-\eta^p}]~\widehat{}$, and $\GL_2(\Z_p)$ acts on it via $\eta\mapsto\frac{a\eta+c}{b\eta+d},~g=\g\in\GL_2(\Z_p)$. The action of $\GL_2(\Z_p)$ on $T_0^*$ is easier to describe than the action on $T_0$. Using the same basis as in the last paragraph and denoting the dual basis element of $e_0$ by $e_0^*$, we have:
\begin{eqnarray} \label{actd}
g:T_0^*\to T_0^*,~f(\eta)e_0^*\mapsto\frac{1}{b\eta+d}f(\frac{a\eta+c}{b\eta+d})e_0^*
\end{eqnarray}
for $g=\g\in\GL_2(\Z_p)$.

Most details here can be found in \cite{Bou-Ca}, especially the first chapter about Deligne's functor (and notice the action of $\GL_2(\Q_p)$ here is the inverse of the action there). 

\section{Raynaud's theory of \texorpdfstring{$\F$}{}-vector space schemes} \label{rayppp}

We want to write down the equation defining $X_1$. Recall that there exists an action $O_D/(\Pi)$ on $X_1$. But $\F \defeq O_D/(\Pi)$ is a finite field which is isomorphic to $\F_{p^2}$. So $X_1$ is a `$\F$-vector space scheme' in the sense of Raynaud. Let's recall Raynaud's theory of $\F$-vector space schemes in our situation. Reference is the first section of Raynaud's paper \cite{Ray}.

\begin{definition}
Let $S$ be a scheme and $\F$ a finite field. A $\F-$vector space scheme is a group scheme $G$ over $S$ with an embedding of $\F$ into the endomorphism ring of $G$ (over $S$).
\end{definition}

Although the definition here is different from Raynaud's orginal definition, it's clear that they are equivalent. Now let $G$ be a $\F-$vector space scheme, we also use $G$ to denote the group scheme in the definition by abuse of notations and the action of $\lambda\in\F$ by $[\lambda]$. Following Raynaud, we assume $G$ is finite, flat and of finite presentation over $S$.

Let $\cA$ be the bialgebra of $G$ and $\cI$ be the augmentation ideal. Then $\F^\times$ acts on $\cA$ and $\cI$. In Raynaud's paper, he defined a ring `$D$'. Since we already use $D$ for quaternion algebra, I will use $D_R$ for that `$D$'. In our case, we can think $D_R$ as $\Z_{p^2}$, the quadratic extension of $\Z_p$ in $\Z_p^{nr}$. Although this ring is much bigger than Raynaud's $D_R$, both of them give the same result here. Under the hypothesis ($*$) in Raynaud's paper and fixing a map $S\to\Spec D_R$, we have a canonical decomposition of $\cI$:
\[
\cI=\oplus_{\chi\in M}\cI_\chi
\]
where $M$ is the set of characters of $\F^\times$ with value in $D_R^\times$, and $\cI_\chi$ is defined as the `$\chi$-isotypic component'. More precisely, for every open set $V$ on $S$, $H^0(V,\cI_\chi)$ is the set of elements $a\in H^0(V,\cI)$, such that $[\lambda]a=\chi(\lambda)a$ for any $\lambda\in\F^\times$.

\begin{definition} \label{fch}
Let $\chi_1,\chi_2$ be the characters of $\F^\times=(O_D/\Pi)^\times$ with values in $D_R^\times=\Z_{p^2}^\times$ such that the composition:
\[
\F_{p^2}^\times\simeq (O_D/\Pi)^\times \stackrel{\chi_i}{\longrightarrow} \Z_{p^2}^\times
\]
is the Teichm\"uller character if $i=1$ and its Galois twist if $i=2$. Here, the first isomorphism is the one  we fixed in the beginning. They are the fundamental characters defined in Raynaud's paper.
\end{definition}
It is clear that $\chi_1^p=\chi_2,\chi_2^p=\chi_1$. Every character $\chi$ in $M$ can be written uniquely as:
\[\chi=\chi_1^{n_1}\chi_2^{n_2},~~~~0\leq n_1,n_2\leq p-1,~(n_1,n_2)\neq (0,0)\]

Now, it is easy see to that given two characters $\chi,\chi'$ in $M$, we have two $\cO_S$-linear map:
\[ \left\{
\begin{array}{cc}
c_{\chi,\chi'}:\cI_{\chi\chi'}\to\cI_{\chi}\otimes\cI_{\chi'} \\
d_{\chi,\chi'}:\cI_{\chi}\otimes\cI_{\chi'}\to\cI_{\chi\chi'}
\end{array}  \right.
\]
coming from the comultiplication and multiplication structure of $\cA$. A slight generalization of this (or equivalently iterate this $p-1$ times) will give us:
\[ \left\{
\begin{array}{cc}
c_i:\cI_{\chi_{i+1}}\to\cI_{\chi_i}^{\otimes p} \\
d_i:\cI_{\chi_i}^{\otimes p}\to\cI_{\chi_{i+1}}
\end{array}  \right.\]
for $i=1,2$, and we identify $\chi_3$ as $\chi_1$.

Under the hypothesis ($**$) in Raynaud's paper which says that each $\cI_\chi$ is an invertible sheaf on $S$, we have the following classification theorem of $\F-$vector space scheme.

\begin{thm}[Raynaud] \label{Rthm}
Let $S$ be a $D_R$-scheme. The map:
\[G\mapsto(\cI_{\chi_i},c_i:\cI_{\chi_{i+1}}\to\cI^{\otimes p}_{\chi_i},d_i:\cI^{\otimes p}_{\chi_i}\to\cI_{\chi_{i+1}})_{i=1,2}\]
defines a bijection between the isomorphism classes of $\F-$vector space schemes over $S$ satisfying ($**$) and the isomorphism classes of $(\cL_1,\cL_2,c_1,c_2,d_1,d_2)$, where:
\begin{enumerate}
\item $\cL_i$ is an invertible sheaf on $S$ for any $i=1,2$.
\item $c_i,d_i$ are $\cO_S$-linear map:
\[ \left\{
\begin{array}{cc}
c_i:\cL_{\chi_{i+1}}\to\cL_{\chi_i}^{\otimes p} \\
d_i:\cL_{\chi_i}^{\otimes p}\to\cL_{\chi_{i+1}}
\end{array}  \right.\]
such that $d_i\circ c_i=w~\Id_{\cL_{i+1}}$. Here $w$ is a constant in $D_R$ that only depends on $\F$ and can be expressed using Gauss sums. More precisely, if we identify $D_R$ with $\Z_{p^2}$, $w\in\Z_p\subset\Z_{p^2}$ with $p$-adic valuation $1$. And if we write $w=pu$, then $u\equiv -1 (\modd~ p)$.
\end{enumerate}
\end{thm}

The inverse map in the theorem is as follows: we define $\cA=\bigoplus_{0\leq a_i\leq p-1}(\cL_1^{\otimes a_1}\otimes\cL_2^{\otimes a_2})$ and equip it with the multiplication and comultiplication structure using $d_i,c_i$. $\cA$ is now a bialgebra and thus defines a group scheme over $S$. The action of $\F^\times$ is defined as the character $\chi_i$ on $\cL_i$ and more generally as the character $\chi_1^{a_1}\chi_2^{a_2}$ on $\cL_1^{\otimes a_1}\otimes \cL_2^{\otimes a_2}$. We now define $0$ in $\F$ acts trivially on $\cA$. The condition of $c_i,d_i$ guarantees we indeed get a $\F$-vector space scheme. As a corollary, we have a description of the invariant differential forms of $G$:

\begin{cor} \label{inv}
\[
\omega_{G/S}\simeq\cI/\cI^2=(\cL_1/d_{2}(\cL_{2}^{\otimes p}))\oplus (\cL_2/d_1(\cL_1^{\otimes p}))
\]
\end{cor}

\begin{rem} \label{leq}
When $S$ is an affine scheme, say $\Spec(A)$, and $\cL_i$ is free over $S$ for all $i$, we have an explicit description of $\cA$. Suppose $x_i$ is a basis of $\cL_i$. Under such basis, $d_i$ becomes an element $v_i$ inside $A$, namely $d_i(x_i^{\otimes p})=v_ix_{i+1}$. Then the bialgebra $\cA$ is isomorphic to $A[x_1,x_2]/(x_1^p-v_1x_2,x_2^p-v_2x_1)$ as an $A$-algebra.
\end{rem}

\begin{rem} The Cartier dual of a $\F$-vector space scheme is also a $\F$-vector space scheme by the dual action of $\F$. On the level of bialgebra, the action of $\F$ is given by its transpose. If $G$ corresponds to $(\cL_i,c_i,d_i)$, the Cartier dual $G^*$ corresponds to $(\cL_i^*,d_i^*,c_i^*)$, where $\cL_i^*=\Hom_{\cO_S}(\cL_i,\cO_S)$ and $d_i^*$ (resp. $c_i^*$) is the transpose of $d_i$ (resp. $c_i$).
\end{rem}

\section{Some results about the formal polarization} \label{fp}

We want to apply Raynaud's theory  to our situation. Although our base scheme is a formal scheme, the argument of Raynaud can be extended naturally to this situation. As we remarked in the beginning of the previous section, $X_1=X[\Pi]$ is a $\F$-vector space scheme over $X_0=\widehat{\Omega}\hat{\otimes}\widehat{\Q_p^{nr}}$, where $\F=O_D/(\Pi)$. Using the fact that its generic fibre $\X_1$ is \'{e}tale over $\Omega\hat{\otimes}\widehat{\Z_p^{nr}}$ and applying Proposistion 1.2.2. in Raynaud's paper, we know that $X_1$ satisfies hypothesis ($**$). So the classification theorem tells us there exists $2$ invertible sheaves $\cL_1,\cL_2$, and maps:

\begin{eqnarray} \label{str1}
c_1:\cL_2\mapsto\cL_1^{\otimes p},~~~ c_2:\cL_1\mapsto\cL_2^{\otimes p}\\ \label{str2}
d_1:\cL_1^{\otimes p}\mapsto\cL_2,~~~d_2:\cL_2^{\otimes p}\mapsto\cL_1
\end{eqnarray}  
such that $d_1\circ c_1=w~\Id_{\cL_{2}},d_2\circ c_2=w~\Id_{\cL_{1}}$.

In order to determine $c_i,d_i$, we need the existence of `formal $*$-polarization' of the universal formal group $X$, which is a lemma in the proof of Proposition 4.3. of \cite{Dr}, and proved in details in \cite{Bou-Ca} (Chapitre \RNum{3} Lemma 4.2.). I would like to recall it here without proof.

\begin{lem} \label{lem1}
Suppose $t\in D$ such that $t^2\in pO_D^\times$. There exists a symmetric isomophism $\lambda:X\to X^*$, where $X^*$ is the Cartier dual of $X$, such that the following diagram commutes:
\[\begin{CD}
X@>\lambda>>X^*\\
@Vt^{-1}\bar{d}tVV     @VVd^*V\\
X@>\lambda>>X^*
\end{CD}
\]
for any $d\in O_D$, where $\bar{d}$ is the canonical involution of $d$ in $D$, $d^*$ is the dual morphism of the endomorphism $d$. Here symmetric means $\lambda=\lambda^*$ under the canonical identification between $X$ and $X^{**}$.
\end{lem}

\begin{rem}This isomorphism is not unique, but is unique up to $\Z_p^\times$ action. From now on, we will fix one such isomorphism $\lambda$ that is defined in \cite{Dr} and \cite{Bou-Ca}. So we also fix such a $t$.
\end{rem}

How does this ismorphism behave under the action $\GL_2(\Q_p)$? Recall that $X_0=\widehat{\Omega}\hat{\otimes}\widehat{\Z_p^{nr}}$.

\begin{lem} \label{lem2}
Suppose $g\in \GL_2(\Q_p)$ and $\det(g)\in p^\Z$,  then $g$ `commutes' with $\lambda$. More precisely, let $g:X_0\to X_0$ be the automorphism of $X_0$ induced by $g$ by abuse of notations. Then there exists a natural isomorphism $\mu_g:X\to g^*X$ over $X_0$, where $g^*X$ is the pull back of $X$ under $g:X_0\to X_0$ by the equivariance of $\GL_2(\Q_p)$ action. Denote $\mu_g^*$ the dual morphism of $\mu_g$. We have the following commutative diagram:
\[ \label{diag1}
\begin{CD} 
X^* @<\mu_g^*<< (g^*X)^*\simeq g^*X^* @>>>X^* \\
@A\lambda AA          @Ag^*\lambda AA         @AA\lambda A \\
X @>\mu_g>>         g^*X @>>>                           X \\
@VVV                        @VVV                             @VVV \\
X_0 @=                     X_0            @>g>>        X_0         
\end{CD}\]
In general, for any $g\in\GL_2(\Q_p)$, we have the same diagram but replace the upper left square by
\[\begin{CD}
X^* @<\mu_g^*<< g^*X^*\\
@A\lambda AA     @AA g^*\lambda A\\
X @>>\mu_{g\cdot\frac{p^n}{det(g)}}> g^*X
\end{CD}\]
where $n=v_p(det(g))$. Notice that this makes sense since $\Z_p^\times$ has trivial action on $X_0$, so $g^*X=(g\cdot\frac{p^n}{det(g)})^*X$.
\end{lem}

\begin{proof}
Since I will use some formulas in \cite{Dr} and \cite{Bou-Ca}, I think it's better not to translate their left action of $\GL_2(\Q_p)$ to right action here.  Hence I will follow their convention in this proof.

It's clear that we only need to prove the general case. Thanks to Drinfel'd's lemma (Lemma on Strictness for p-Divisible Groups in the appendix of \cite{Dr}), it suffices to verify this commutative diagram after we reduce modulo $p$. But by Drinfel'd's construction of the universal $p$-divisible group, $X\times \F_p$ is quasi-isogenous of degree $0$ to a constant $p$-divisible group $\Phi_{X_0\times \F_p}$ over $X_0\times \F_p$. Here, recall $\Phi$ is a $p$-divisible group defined over $\overbar{\F_p}$, and $\Phi_{X_0\times \F_p} \defeq \Phi\times_{\overbar{\F_p}} X_0$. $\GL_2(\Q_p)$ acts on $\Phi$ as quasi-isogenies. A detailed description of $\Phi$ can be found in \cite{Bou-Ca} Chapitre \RNum{3} 4.3 or the proof of Proposition 4.3. of \cite{Dr}. Besides, the construction of the `formal polarization' $\lambda$ tells us that $\lambda$ actually comes from a `formal polarization' $\lambda_0$ of $\Phi$ that makes the following diagram commutative:

\[\begin{CD}
X\times \F_p @>\overbar{\lambda} >> X^*\times \F_p\\
@V\rho VV                               @AA\rho^*A\\
\Phi_{X_0\times \F_p}   @>\lambda_{0,X_0\times \F_p}>> \Phi_{X_0\times \F_p}^*
\end{CD}\]
where $\bar{\lambda} \defeq \lambda (\modd~p)$, $\rho$ is the quasi-isogeny and $\rho^*$ is its dual. From the definition of the action of $\GL_2(\Q_p)$, we know how $\rho$ changes under this action (basically the action of $\GL_2(\Q_p)$ on $\Phi$ with some twist of Frobenius, see \cite{Dr} section 2 or \cite{Bou-Ca} Chapitre \RNum{2} section 9). Thus we can translate the diagram of $X$ into a diagram of $\Phi$. It turns out that it suffices to verify the following diagram is commutative:
\[\begin{CD}
\Phi^* @<(\Frob_{\Phi}^{-n}\circ g)^*<< (Fr^{-n})^*\Phi^* \\
@A\lambda_0 AA    @AA(Fr^{-n})^*\lambda_0 A \\
\Phi @>\Frob_{\Phi}^{-n}\circ (g\cdot\frac{p^n}{det(g)})>> (Fr^{-n})^*\Phi  
\end{CD}\]
Here $Fr:\Spec(\overbar{\F_p})\to\Spec(\overbar{\F_p})$ is the arithmetic Frobenius, $\Frob_{\Phi}: (Fr^{-1})^*\Phi\to\Phi$ is the Frobenius morphism over $\Spec(\overbar{\F_p})$. I would like to decompose the diagram as the following diagram (and invert the arrow on the bottom line):

\[\begin{CD}
\Phi^* @<g^*(det(g))^{-1}<<\Phi^* @<(det(g)\Frob_{\Phi}^{-n})^*<<(Fr^{-n})^*\Phi^*\\
@A\lambda_0 AA       @A\lambda_0 AA       @A (Fr^{-n})^*\lambda_0 AA \\
\Phi  @<g^{-1}<<  \Phi @<(\frac{det(g)}{p^n})\Frob_{\Phi}^n<< (Fr^{-n})^*\Phi
\end{CD}\]

First we look at the right square.
\[
(det(g)\Frob_\Phi^{-n})^*=(\frac{det(g)}{p^n})(p^n\Frob_\Phi^{-n})^*=(\frac{det(g)}{p^n})(\Ver_\Phi^n)^*=(\frac{det(g)}{p^n})\Frob_{\Phi^*}^n
\]
where $\Ver_{\Phi^*}$ is the Verschiebung morphism. Now it is easy to see the diagram commutes from the basic property of Frobenius morphism.

As for the left square, the commutativity in fact comes from our explicit choice of $\Phi,\lambda_0$ and the action of $\GL_2(\Q_p)$. See the Remarque in \cite{Bou-Ca} Chapitre \RNum{3} 4.3. which says the Rosati involution associated to $\lambda_0$ is nothing but the canonical involution on $\mathrm{M}_2(\Q_p)$. 
\end{proof}

\begin{rem}
When $g\in \SL_2(\Q_p)$, the calculation above is essentially given in \cite{Bou-Ca} Chapitre \RNum{3} 4.5..
\end{rem}

\section{Structure of \texorpdfstring{$\X_1$}{} and a formal model of \texorpdfstring{$\Sigma_1$}{}} \label{soxfd}

Now let's see how the discussion above helps us study $c_i,d_i$ in \eqref{str1}, \eqref{str2}. The main result is the following

\begin{prop} \label{forp}
There exists an isomorphism $\lambda_1$ from $X_1=X[\Pi]$ to $X[\Pi]^*$, the Cartier dual of $X[\Pi]$, such that 
\begin{enumerate}
\item the following diagram commutes for any $d\in O_D^\times$.
\[\begin{CD}
X[\Pi] @>\lambda_1 >>X[\Pi]^*\\
@V \bar{d} VV                @VV d^*  V\\
X[\Pi] @>\lambda_1 >>X[\Pi]^*
\end{CD}\]
Recall that $\bar{d}$ is the canonical involution of $d$ in $D$.

\item $\lambda_1^*=\lambda_1\circ [-1] =[-1]^*\circ\lambda_1$, where $[-1]$ denotes the action of $-1\in O_D$.
\end{enumerate}
\end{prop}

\begin{proof}
We can take $t=\Pi$ in lemma \ref{lem1}. Then if we restrict to the $p$ torsion points of $X$, we certainly get an isomophism:
$$\lambda_p:X[p]=X[\Pi^{-1}\bar{p}\Pi]\to X^*[p^*]=X^*[p]$$

Notice that $X^*[p]$ is canonically isomorphic to $(X[p])^*$, the Cartier dual of $X[p]$. The inclusion of $X[\Pi]$ into $X[p]$ induces a canonical isomorphism:
$$j:X^*[p]/X^*[\Pi^*]= X^*[p]/((X^*[p])[\Pi^*])\xrightarrow{\sim} (X[p])^*/((X[p])^*[\Pi^*])\xrightarrow{\sim} X[\Pi]^*.$$

Since $\Pi^2=p$, the map $\Pi^*:X^*[p]\to X^*[p]$ gives us an isomorphism:
$$h:X^*[p]/X^*[\Pi^*]\xrightarrow{\sim} X^*[\Pi^*].$$

Finally,  we restrict $\lambda$ to the $\Pi$ torsion points of $X$ and get an isomorphism:
$$\lambda_\Pi:X[\Pi]=X[\Pi^{-1}\overbar{\Pi}\Pi]\to X^*[\Pi^*].$$

Now, we define $\lambda_1=j\circ h^{-1}\circ\lambda_\Pi:X[\Pi]\to X[\Pi]^*$. 

What is the Rosati involution associated to $\lambda_1$? I claim the following diagram commutes:
\[\begin{CD}
X[\Pi]@>\lambda_\Pi>> X^*[\Pi^*] @<h<< X^*[p]/X^*[\Pi^*]@>j>>X[\Pi]^*\\
@V\Pi^{-1}\bar{d}\Pi VV   @Vd^*VV        @V(\Pi d\Pi^{-1})^* VV              @V(\Pi d\Pi^{-1})^* VV \\
X[\Pi]@>\lambda_\Pi>> X^*[\Pi^*] @<h<< X^*[p]/X^*[\Pi^*]@>j>>X[\Pi]^*\\
\end{CD}\]
The left most square is commutative because we have a similar diagram for $\lambda$ and $\lambda_\Pi$ is a restriction of $\lambda$. The right most diagram is commutative because $j$ comes from the canonical quotient map $X^*[p]\simeq (X[p])^*\to X[\Pi]^*$ and this certainly commutes with the dual endomorphism of $O_D$. As for the middle square, notice that $h$ is induced by the map $\Pi^*:X^*[p]\to X^*[p]$ and everything is clear.

Since $\Pi^{-1}\bar{d}\Pi \equiv d (\modd~\Pi O_D)$ and everything in the diagram above is killed by $\Pi$ or $\Pi^*$, we can replace $\Pi^{-1}\bar{d}\Pi$ by $d$ and $(\Pi d\Pi^{-1})^*$ by $\bar{d}^*$, hence get the desired commutative diagram in part (1). 

As for part (2), we use $G,H$ to denote $X[p],X[\Pi]$ respectively. Then $G^*=X^*[p]$. We can decompose $-\Pi:G\to G$ as 
$$G\xrightarrow{q} G/H \xrightarrow{h_{-\Pi}} H \xrightarrow{i} G$$
where $i$ (resp. $q$) is the canonical inclusion of $H$ to $G$ (resp. canonical quotient map of $G$ to $G/H$). $h_{-\Pi}$ is the induced isomorphism. 

Notice that $\Pi^{-1}\bar{\Pi}\Pi= -\Pi$ and $G$ is killed by $p$. We have the following diagram, which is a restriction of the diagram \ref{lem1} to $G$ with $d=\Pi$.
\[\begin{CD}
G@>-\Pi>>G \\
@V\lambda_pVV @VV\lambda_pV \\
G^*@>\Pi^*>>G^*
\end{CD}\]
 
Similarly we can decompose $\Pi^*$ as we did for $-\Pi$ and have the commutative diagram:
\[\begin{CD}
G@>q>> G/H @>h_{-\Pi}>> H @>i>> G\\
@V\lambda_pVV @V\lambda_{G/H}VV @V\lambda_{H}VV @V\lambda_pVV\\
G^*@>i^*>> H^* @>h_{\Pi^*}>> (G/H)^* @>q^*>> G^*
\end{CD}\]
such that the composition of all three maps in the bottom line is $\Pi^*$. $h_{\Pi^*}$ is induced from $\Pi^*$. Thus it's easy to see $([-1]\circ h_{-\Pi})^*=h_{\Pi^*}$ and its dual $h_{\Pi^*}^*=[-1]\circ h_{-\Pi}$.

Since $\lambda$ is symmetric, so is $\lambda_p$ and we certainly have $\lambda_{G/H}^*=\lambda_H$. Now it's not hard to see that our $\lambda_1$ is nothing but $h_{\Pi^*}^{-1}\circ \lambda_H$. So,
\begin{eqnarray*}
\lambda_1^*&=&(h_{\Pi^*}^{-1}\circ \lambda_H)^*=\lambda_H^*\circ(h_{\Pi^*}^{-1})^* = \lambda_{G/H} \circ (h_{\Pi^*}^*)^{-1}\\
&=&\lambda_{G/H} \circ ([-1]\circ h_{-\Pi})^{-1}=\lambda_{G/H} \circ h_{-\Pi}^{-1}\circ [-1]^{-1}=\lambda_1\circ [-1].
\end{eqnarray*}
The last identity comes from the middle square of the diagram above.
\end{proof}

\begin{cor}
$\lambda_1$ induces isomorphisms $\lambda_{\cL_1}: \cL_2^*\xrightarrow{\sim}\cL_1$, $\lambda_{\cL_2}: \cL_1^*\xrightarrow{\sim}\cL_2$. Moreover, $\lambda_{\cL_1}=-\lambda_{\cL_2}^*$ if $p$ is odd $\lambda_{\cL_1}=\lambda_{\cL_2}^*$ if $p=2$.
\end{cor}
\begin{proof}
Using theorem \ref{Rthm}, we can identify $X_1=X[\Pi]$ with $(\cL_1,\cL_2,c_1,c_2,d_1,d_2)$, and the final remark in that tells us we can identify $X[\Pi]^*$ with $(\cL_1^*,\cL_2^*,d_1^*,d_2^*,c_1^*,c_2^*)$. 

Now $\lambda_1$ gives us an isomorphism from $X[\Pi]$ to $X[\Pi]^*$ but this is not $\F=O_D/(\Pi)$-equivariant. For a character $\chi$ of $\F^\times$ and  consider it as a character of $O_D^\times$, we have:
\[
\chi(\bar{d})=\chi(d^p)=\chi^p(d)
\]
for any $d\in O_D^\times$. This is because when we restrict the canonical involution to a quadratic unramified extension of $\Z_p$ inside $O_D$, it is nothing but the non-trivial Galois action. So modulo the uniformizer, it becomes Frobenius automorphism.

Take $\chi=\chi_1$, one of the fundamental characters, then $\chi_1^p=\chi_2$. Thus $\chi_1(\bar{d})=\chi_2(d)$. Similarly, we have $\chi_2(\bar{d})=\chi_1(d)$. From these identities and the commutative diagram in \ref{forp}, it is easy to see $\lambda_1$ really induces  isomorphisms  $\lambda_{\cL_1}: \cL_2^*\xrightarrow{\sim}\cL_1$, $\lambda_{\cL_2}: \cL_1^*\xrightarrow{\sim}\cL_2$. The last identity comes from  the consideration that  difference between $\lambda_1$ and $\lambda_1^*$ is the action of $-1$. And we know $\chi_1(-1)=\chi_2(-1)=-1$ if $p$ is odd and $1$ otherwise.
\end{proof}

From now on, I will assume $p$ is odd.

\begin{cor}
Under the isomorphism $\lambda_{\cL_1}$, we have $-d_1=c_2^*$. More precisely, we have the following commutative diagram:
\[\begin{CD}
\cL_1^{\otimes p} @>-d_1>> \cL_2\\
@A \lambda_{\cL_1}^{\otimes p}AA   @AA \lambda_{\cL_1}^*A\\
(\cL_2^*)^{\otimes p} @>>c_2^*> \cL_1^*
\end{CD}\]
\end{cor}

\begin{proof}
It is easy to see $\lambda_1$ induces a similar diagram by replacing $-d_1$ with $d_1$ and $\lambda_{\cL_1}^*$ with $\lambda_{\cL_2}$. Now the corollary follows from $\lambda_{\cL_1}=-\lambda_{\cL_2}^*$.
\end{proof}

\begin{cor} \label{w}
Under the isomorphism $\lambda_{\cL_1}$, we can identify $d_1:\cL_1^{\otimes p}\to\cL_2$ with a global section of $\cL_1^{\otimes -p-1}$. Similarly, we can identify $d_2$ with a global section of $\cL_1^{\otimes p+1}$. The canonical pairing of $H^0(X_0,\cL_1^{\otimes -p-1})\times H^0(X_0,\cL_1^{\otimes p+1})\to H^0(X_0,\cO_{X_0})$ sends $(d_1,d_2)$ to the constant $-w=-pu$, where $w$ is the constant in theorem \ref{Rthm}, and $u$ is $w/p$.
\end{cor}

\begin{proof}
Recall that $d_2\circ c_2=w\Id_{\lambda_{\cL_1}}$. Then everything follows from last corollary.
\end{proof}

\begin{cor} \label{cst}
Recall that the bialgebra of $X_1$ is isomorphic to $\bigoplus_{0\le i,j \le p-1}\cL_1^{\otimes i}\otimes\cL_2^{\otimes j}$ as an $\cO_{X_0}$- module. The isomorphism $\lambda_{\cL_1}$ gives a global section $\widetilde{\lambda_{\cL_1}}$ of $\cL_1\otimes\cL_2$. Then as a global section of $X_1$, we have:
\[
\widetilde{\lambda_{\cL_1}}^p=-w\widetilde{\lambda_{\cL_1}} 
\]
where everything is comupted inside the bialgebra of $X_1$. 
\end{cor}

\begin{proof}
We only need to verify this locally. Suppose $\cL_1,\cL_2$ are free over an open set $U$ and generated by $x_1,x_2$ such that $x_1\otimes x_2=\widetilde{\lambda_{\cL_1}}$, or equivalently they are dual to each other under $\lambda_{\cL_1}$. Now $d_1,d_2$ are given by two elements $v_1,v_2\in H^0(U,\cO_{X_0})$. So $x_1^p=v_1x_2,x_2^p=v_2x_1$ (see remark \ref{leq}). But from last corollary, we have $v_1v_2=-w$. Thus the product of these two equations is just what we want.
\end{proof}

\begin{rem}
Perhaps it is better to remark here that $\cL_1,\cL_2$ are non-trivial on the formal model but we'll see later that they become trivial on the generic fibre (lemma \ref{lftrv}).
\end{rem}

Now we can describe a formal model of $\Sigma_1$. Recall that $\Sigma_1=\X_1-\X_0$, where $\X_1,\X_0$ are the rigid analytic spaces associated to $X_1,X_0$.

\begin{prop} \label{defsn}
Let $\cA=\bigoplus_{0\le i,j \le p-1}\cL_1^{\otimes i}\otimes\cL_2^{\otimes j}$ be the bialgebra of $X_1$. Then $\cA/(\widetilde{\lambda_{\cL_1}}^{p-1}+w)$ (the closed subscheme defined by the ideal sheaf $(\widetilde{\lambda_{\cL_1}}^{p-1}+w)$) is a formal model of $\Sigma_1$. We will use $\widehat{\Sigma_1^{nr}}$ to denote this formal model.
\end{prop}

\begin{proof}
It suffices to check this locally on $X_0$, so we can assume $\cL_1,\cL_2$ are free. A point $x$ on $\Sigma_1$ gives a morphism $x:\cA\to\C_p$. If it does not factor through $\cA/(\widetilde{\lambda_{\cL_1}}^{p-1}+w)$, $x(\widetilde{\lambda_{\cL_1}})$ has to be $0$ because last corollary tells us $(\widetilde{\lambda_{\cL_1}}^{p-1}+w)\widetilde{\lambda_{\cL_1}}=0$. But 
$$x_1^{p+1}=x_1^px_1=v_1x_2x_1=v_1\widetilde{\lambda_{\cL_1}}$$
so $x(x_1)=0$ and $x(x_2)=0$ by the same argument. Therefore $x$ factors through $\cA$ modulo the ideal sheaf generated by $x_1,x_2$ which is the augmentation ideal. Therefore $x$ is in $\X_0$. Converse is trivial.
\end{proof}

It's easy to see its underlying algebra of $\widehat{\Sigma_1^{nr}}$ is just $\bigoplus_{0\le i,j \le p-1,~(i,j)\neq (p-1,p-1)}\cL_1^{\otimes i}\otimes\cL_2^{\otimes j}$.

\begin{rem}
There exists a natural action of $\GL_2(\Q_p)$ and $O_D^\times$ action on $\widehat{\Sigma_1^{nr}}$. The action of $O_D^\times$ is clear. To see the action of $\GL_2(\Q_p)$, notice that $\widetilde{\lambda_{\cL_1}}$ is a global section of a trivial line bundle on $X_0$, but $H^0(X_0,\cO_{X_0})$ is canonically isomorphic to $\widehat{\Z_p^{nr}}$ (I will prove this later, see lemma \ref{gfb}). So $\GL_2(\Q_p)$ acts on $\widetilde{\lambda_{\cL_1}}$ as a scalar. Recall that $\widetilde{\lambda_{\cL_1}}^p+w\widetilde{\lambda_{\cL_1}}=0$. This implies $\widetilde{\lambda_{\cL_1}}^{p-1}+w$ is $\GL_2(\Q_p)$-invariant. Same argument shows that the action of $O_{D}^\times$ can be extended to $D^{\times}$.
\end{rem}

But how does $\GL_2(\Q_p)$ act on $\widetilde{\lambda_{\cL_1}}$? Here is a direct consequence of lemma \ref{lem2}.

\begin{prop} \label{invp}
$g(\widetilde{\lambda_{\cL_1}})=\chi_1(det(g)/p^n)^{-1}\widetilde{\lambda_{\cL_1}}$, where $g\in\GL_2(\Q_p),n=v_p(det(g))$.
\end{prop}

\section{Local equation of \texorpdfstring{$X_1$}{} and \texorpdfstring{$\widehat{\Sigma_1^{nr}}$}{}} \label{leox}
In order to get a semi-stable model of $\widehat{\Sigma_1^{nr}}$, we need to know the local equation defining it. Recall in section \ref{bfd}, we describe an open covering $\{\widehat{\Omega_e}\hat{\otimes}\widehat{\Z_p^{nr}}\}_e$ of $X_0$, such that:
\[
\widehat{\Omega}_e\hat{\otimes}\widehat{\Z_p^{nr}}\simeq\Spf\frac{\widehat{\Z_p^{nr}}[\zeta,\eta]}{\zeta\eta-p}[\frac{1}{1-\zeta^{p-1}},\frac{1}{1-\eta^{p-1}}]~\widehat{}  
\]

We try to write down the equation of $\widehat{\Sigma_1^{nr}}$ above each $\widehat{\Omega_e}\hat{\otimes}\widehat{\Z_p^{nr}}$. Our first observation is:

\begin{lem} \label{trv}
Any line bundle $\cL$ over $\widehat{\Omega}_e\hat{\otimes}\widehat{\Z_p^{nr}}\simeq\Spf\frac{\widehat{\Z_p^{nr}}[\zeta,\eta]}{\zeta\eta-p}[\frac{1}{1-\zeta^{p-1}},\frac{1}{1-\eta^{p-1}}]~\widehat{}$ is trivial.
\end{lem}

\begin{proof}
Recall we denote $\frac{\widehat{\Z_p^{nr}}[\zeta,\eta]}{\zeta\eta-p}[\frac{1}{1-\zeta^{p-1}},\frac{1}{1-\eta^{p-1}}]~\widehat{}$ by $O_{\zeta,\eta}$. 

The special fibre of $\Spf O_{\zeta,\eta}$ is $\Spec\overbar{\F_p}[\zeta,\eta,\frac{1}{1-\zeta^{p-1}},\frac{1}{1-\eta^{p-1}}]/(\zeta\eta)$. I claim every line bundle $\bar{\cL}$ over it is trivial. Let $\bar{L}$ be $H^0(\Spec O_{\zeta,\eta}/p, \bar{\cL})$. Then we have the exact sequence:
\[
0\to \bar{L} \to \bar{L}/(\zeta \bar{L})\oplus \bar{L}/(\eta \bar{L}) \xrightarrow{-} \bar{L}/(\zeta\bar{L}+\eta\bar{L})\to 0
\]
where the inclusion is the canonical morphism and $-$ is defined by taking their difference. This sequence is exact because $\bar{L}$ is locally free and thus flat over $O_{\zeta,\eta}/p$. Notice that $\bar{L}/(\zeta \bar{L})$ defines a line bundle on $\Spec O_{\zeta,\eta}/(p,\zeta)=\Spec \overbar{\F_p}[\eta,\frac{1}{1-\eta^{p-1}}]$, hence has to be trivial. Same result holds for $\bar{L}/(\eta \bar{L})$. Moreover $\bar{L}/(\zeta\bar{L}+\eta\bar{L})$ is nothing but $\overbar{\F_p}$. Using these, it's not hard to find an element that generates $\bar{L}$. So $\bar{\cL}$ is trivial.

Now we can find an element in $H^0(\Spf O_{\zeta,\eta}, \cL)$ that generates $\cL/p$. But $H^0(\Spf O_{\zeta,\eta}, \cL)$ is $p$-adically complete, so this element actually generates the whole $H^0(\Spf O_{\zeta,\eta},\cL)$. Therefore $\cL$ is trivial. (Here we use the fact that a surjective map between two line bundles has to be an isomorphism.)
\end{proof}

Thanks to this lemma, the restriction of $\cL_1$ on $\widehat{\Omega_e}\hat{\otimes}\widehat{\Z_p^{nr}}$ is trivial. We fix an isomorphism between $\widehat{\Omega_e}\hat{\otimes}\widehat{\Z_p^{nr}}$ and $\Spf O_{\zeta,\eta}$. Suppose $x_1$ is a generator of $H^0(\widehat{\Omega_e}\hat{\otimes}\widehat{\Z_p^{nr}},\cL_1)$, and $x_2\in H^0(\widehat{\Omega_e}\hat{\otimes}\widehat{\Z_p^{nr}},\cL_2)$ is the dual basis under the isomorphism $\lambda_{\cL_1}$ defined in the previous section. Let $v_1,v_2$ be the elements given by $d_1,d_2$ under the basis $x_1,x_2$.  Then we know locally $X_1$ is defined by $x_1^p=v_1x_2,x_2^p=v_2x_1$. 

How to determine $v_1,v_2$? Our strategy is to compare the invariant differential forms of $X_1$ computed in two different ways. First recall that the tangent space $T$ of the universal formal group over $X_0$ is a rank $2$ vector bundle over $X_0$ that naturally splits into a direct sum of two line bundles $T_0,T_1$. So the sheaf of invariant differential forms is its dual, namely $T_0^*\oplus T_1^*$. The action of $\Pi$ on $T_0$ sends $T_0$ (resp. $T_1$) into $T_1$ (resp.$T_0$), which we denote by $\Pi_0$ (resp. $\Pi_1$) in section \ref{bfd}. Thus $\Pi_0^*$ (resp. $\Pi_1^*$) sends $T_1^*$ (resp. $T_0^*$) to $T_0^*$ (resp. $T_1^*$) and the sheaf of invariant forms $\omega_{X_1/X_0}$ of $X_1=X[\Pi]$ is:
$$T_0^*/\Pi_0^*T_1^*\oplus T_1^*/\Pi_1^*T_0^*.$$ 

On the other hand, using corollary \ref{inv}, we know that this is also:
$$\cL_1/d_2(\cL_2^{\otimes p})\oplus \cL_2/d_1(\cL_1^{\otimes p}).$$

It is natural to guess:
\begin{lem} \label{gueq}
\begin{eqnarray}
T_0^*/\Pi_0^*T_1^*\simeq \cL_1/d_2(\cL_2^{\otimes p})\\
T_1^*/\Pi_1^*T_0^*\simeq \cL_2/d_1(\cL_1^{\otimes p})
\end{eqnarray}
\end{lem}
\begin{proof}
If we restrict the action of $O_D$ to $\Z_{p^2}$, it acts by identity on $T_0$ and by conjugation on $T_1$. Recall that we fix an embedding of $\Z_{p^2}$ into $O_D$ in the beginning. This is just the definition of $X$ being `special'. Now our desired identification follows from a simple comparison of the action of $\Z_{p^2}^\times$ in both ways.
\end{proof}

Recall that all irreducible components of the special fibre of $X_0$ are isomorphic to $\mathbb{P}^1_{\overbar{\F_p}}$ such that the singular points are exactly $\mathbb{P}^1(\F_p)$. From the explicit description \eqref{strlie1},\eqref{strlie2} of $\Pi_0,\Pi_1$ and the discussion in section \ref{bfd}, we know that on an odd component of the special fibre $s$, $T_0^*/\Pi_0^*T_1^*$ is isomorphic to $\oplus_{P\in s_{sing}}{i_P}_*\overbar{\F_p}$, where $s_{sing}$ is the set of singular points of the special fibre on $s$, and ${i_P}:P\to s$ is the embedding.

Restrict $\cL_1,\cL_2, ~d_2:\cL_2^{\otimes p}\to\cL_1$ to $s$. From $\cL_1/d_2(\cL_2^{\otimes p})\simeq T_0^*/\Pi_0^*T_1^* \simeq \oplus_{P\in s_{sing}}{i_P}_*\overbar{\F_p}$ (on $s$), it's easy to see $deg(\cL_1|_s)-deg(\cL_2^{\otimes p}|_s)=p+1$. But $\cL_2\simeq\cL_1^*$, so $deg(\cL_2|_s)=-deg(\cL_1|_s)$. This implies:

\begin{lem}
$deg(\cL_1|_s)=1$ for any odd component $s$. Similary, $deg(\cL_2|_{s'})=1$ for any even component $s'$.
\end{lem}

Now we would like to choose some good basis of $\cL_1$ so that $v_1,v_2$ have a good form. Using the isomorphism $\widehat{\Omega_e}\simeq \Spf O_{\zeta,\eta}$, we can identify two irreducible components of its special fibre with $\Spec O_{\zeta,\eta}/(\zeta), \Spec O_{\zeta,\eta}/(\eta)$. Assume the second one is odd and we use $s$ to denote the corresponding component in the special fibre of $X_0$ and use $s'$ for the other component. Moreover $\Spec O_{\zeta,\eta}/(\eta)=\Spec \overbar{\F_p}[\zeta,\frac{1}{1-\zeta^{p-1}}]$ hence has an obvious embedding into $\mathbb{P}^1_{\overbar{\F_p}}$ which can be identified as the embedding into $s$. 

Choose a global section $\widetilde{x_1}$ of $\cL_1|_s$ such that it has a simple zero at infinity under the identification above. It is a basis of $H^0(\Spec O_{\zeta,\eta}/(\eta),\cL_1)$. Then under this basis,
$$d_2:\cL_2^{\otimes p}\simeq {\cL_1^*}^{\otimes p}\to\cL_1,~\widetilde{x_1}^{*~\otimes p} \mapsto c(\zeta^p-\zeta)\widetilde{x_1}$$
for some constant $c\in\overbar{\F_p}^\times$, where $\widetilde{x_1}^*$ is the dual basis of $\widetilde{x_1}^*$.

Notice that $\widetilde{x_1}$ is only defined up to a constant. If we replace $\widetilde{x_1}$ by $d\widetilde{x_1}$, then the constant $c$ is replaced by $d^{-p-1}c$. We can choose $d=c^{1/(p+1)}$ to eliminate $c$. More precisely, we can choose a section, which I still call $\widetilde{x_1}$ by abuse of notations, such that under this basis, $d_2$ is just multiplication by $\zeta^p-\zeta$.

We can do similar thing for $s'$, that means we can choose a basis $\widetilde{x_2}$ of $\cL_2|_{\Spec O_{\zeta,\eta}/(\zeta)}$ such that under this basis, $d_1$ is multiplication by $c'(\eta^p-\eta)$. Here we choose $\widetilde{x_2}$ so that $\widetilde{x_1},\widetilde{x_2}^*$ can glue to a global basis $\overbar{x_1}$ of $\cL_1|_{\Spec O_{\zeta,\eta}/(p)}$ (see the proof of lemma \ref{trv}). A priori we know nothing about the constant $c'$.

Now we can lift $\overbar{x_1}$ to a global basis $x_1$ of $\cL_1|_{\Spec O_{\zeta,\eta}}$, so it determines a basis $x_2$ of $\cL_2|_{\Spec O_{\zeta,\eta}}$ under the isomorphism $\lambda_{\cL_1}$. And $d_1,d_2$ are given by two numbers $v_1,v_2$. The explicit description  
\eqref{strlie1},\eqref{strlie2} and lemma \ref{gueq} imply that:
\begin{eqnarray}
v_2=\zeta u_2\\
v_1=\eta u_1
\end{eqnarray}
for some units $u_1,u_2\in O_{\zeta,\eta}^\times$. Note that $u_1u_2=-u$ because $v_1v_2=-w=-pu$ (\ref{w}) and $\eta\zeta=p$. From our choice of $x_1,x_2$, we have  $v_2\equiv\zeta^p-\zeta (\modd~\eta), v_1\equiv\eta^p-\eta(\modd~\zeta)$, so:
\begin{eqnarray}
u_2\equiv \zeta^{p-1}-1~(\modd~\eta)\\
u_1\equiv c'(\eta^{p-1}-1)~(\modd~\zeta).
\end{eqnarray}
This is because $(\zeta,\eta)$ is a regular sequence in $O_{\zeta,\eta}$. In fact $O_{\zeta,\eta}$ is normal. When we take the product of the identities above considered in $O_{\zeta,\eta}/(\zeta,\eta)\simeq \overbar{\F_p}$, the left hand side is $u_1u_2=-u$, which is $1$ modulo $p$ (see theorem \ref{Rthm}), while the right hand side is just $c'$. Therefore,
\begin{lem}
$c'=1$.
\end{lem}
Notice that $u_2\equiv u_1^{-1}(\modd~p)$, and $(\zeta)\cap(\eta)=(p)$ in $O_{\zeta,\eta}$. It's not hard to see that:
\begin{lem}
\begin{eqnarray}
u_1\equiv -\frac{\eta^{p-1}-1}{\zeta^{p-1}-1} ~(\modd~p)\\
u_2\equiv -\frac{\zeta^{p-1}-1}{\eta^{p-1}-1} ~(\modd~p)
\end{eqnarray}
\end{lem}
Now if we replace our $x_1$ by $rx_1$ for some unit $r\in O_{\zeta,\eta}^\times$, then $x_2$ is replaced by $r^{-1}x_2$ and $u_1$ (resp. $u_2$) is replaced by $r^{p+1}u_1$ (resp. $r^{-p-1}u_2$). Write $u_1=-\frac{\eta^{p-1}-1}{\zeta^{p-1}-1}r_1$, then $r_1\equiv 1(\modd~p)$. Thus $r_1^{1/(p+1)}$ exists in $O_{\zeta,\eta}$. Hence we can modify our $x_1$ to make $u_1=-\frac{\eta^{p-1}-1}{\zeta^{p-1}-1}$. In summary, 
\begin{prop}
We can choose appropriate basis $x_1,x_2$ of $\cL_1,\cL_2$ over $\widehat{\Omega_e}\hat{\otimes}\widehat{\Z_p^{nr}}\simeq \Spf O_{\zeta,\eta}$ such that they are dual to each other under $\lambda_{\cL_1}$, and under this basis,
\begin{eqnarray}
d_1:\cL_1^{\otimes p}\to\cL_2,~~x_1^{\otimes p}\mapsto -\frac{\eta^p-\eta}{\zeta^{p-1}-1}x_2\\
d_2:\cL_2^{\otimes p}\to\cL_1,~~x_2^{\otimes p}\mapsto u\frac{\zeta^p-\zeta}{\eta^{p-1}-1}x_1
\end{eqnarray}
\end{prop}

\begin{cor} \label{strcor}
The restriction of $X_1$ to $\widehat{\Omega_e}\hat{\otimes}\widehat{\Z_p^{nr}}\simeq \Spf O_{\zeta,\eta}$ is defined by  
$$\Spf O_{\zeta,\eta}[x_1,x_2]/(x_1^p+\frac{\eta^p-\eta}{\zeta^{p-1}-1}x_2, x_2^p-u\frac{\zeta^p-\zeta}{\eta^{p-1}-1}x_1)$$
Similarly, the restriction of $\widehat{\Sigma_1^{nr}}$ to $\widehat{\Omega_e}\hat{\otimes}\widehat{\Z_p^{nr}}$ is defined by
$$\Spf O_{\zeta,\eta}[x_1,x_2]/(x_1^p+\frac{\eta^p-\eta}{\zeta^{p-1}-1}x_2, x_2^p-u\frac{\zeta^p-\zeta}{\eta^{p-1}-1}x_1, (x_1x_2)^{p-1}+pu)$$
\end{cor}

\begin{proof}
The first statement follows from the above discussion. As for $\widehat{\Sigma_1^{nr}}$, notice that $x_1x_2$ is just $\widetilde{\lambda_{\cL_1}}$ defined in \ref{cst}. So this is the definition of $\widehat{\Sigma_1^{nr}}$.
\end{proof}

Fix a $\uu_1=(-u)^{1/(p^2-1)}$ in $\Z_p$. If we replace $x_1$ by $\uu_1x_1$, $x_2$ by $\uu_1^px_2$, then our new $x_1,x_2$ are dual to each other under $\uu_1^{-p-1}\lambda_{\cL_1}$. Under this basis, $x_1x_2=\uu_1^{-p-1}\widetilde{\lambda_{\cL_1}}$, 

\begin{cor} \label{strcor1}
The restriction of $X_1$ to $\widehat{\Omega_e}\hat{\otimes}\widehat{\Z_p^{nr}}\simeq \Spf O_{\zeta,\eta}$ is defined by  
$$\Spf O_{\zeta,\eta}[x_1,x_2]/(x_1^p+\frac{\eta^p-\eta}{\zeta^{p-1}-1}x_2, x_2^p+\frac{\zeta^p-\zeta}{\eta^{p-1}-1}x_1)$$
Similarly, the restriction of $\widehat{\Sigma_1^{nr}}$ to $\widehat{\Omega_e}\hat{\otimes}\widehat{\Z_p^{nr}}$ is defined by
$$\Spf O_{\zeta,\eta}[x_1,x_2]/(x_1^p+\frac{\eta^p-\eta}{\zeta^{p-1}-1}x_2, x_2^p+\frac{\zeta^p-\zeta}{\eta^{p-1}-1}x_1, (x_1x_2)^{p-1}-p)$$
\end{cor}

Suppose $e=[s,s']$ and we have \eqref{isom2}, \eqref{isom3}, \eqref{isom1}. Then $\widehat{\Omega_{s'}}\hat{\otimes}\widehat{\Z_p^{nr}}$ is obatined by inverting $\eta$ in $O_{\zeta,\eta}$ and taking the $p$-adic completion.  Therefore, we have:

\begin{cor} \label{strcor2}
The restriction of $X_1$ to $\widehat{\Omega_{s'}}\hat{\otimes}\widehat{\Z_p^{nr}}\simeq \Spf\widehat{\Z_p^{nr}}[\eta,\frac{1}{\eta^p-\eta}]~\widehat{}$ is defined by  
$$\Spf \widehat{\Z_p^{nr}}[\eta,\frac{1}{\eta^p-\eta}]~\widehat{}~[x_1,x_2]/(x_1^p+\frac{\eta^p-\eta}{(p/\eta)^{p-1}-1}x_2, x_2^p+\frac{(p/\eta)^p-(p/\eta)}{\eta^{p-1}-1}x_1)$$
Similarly, the restriction of $\widehat{\Sigma_1^{nr}}$ to $\widehat{\Omega_{s'}}\hat{\otimes}\widehat{\Z_p^{nr}}$ is defined by
$$\Spf \widehat{\Z_p^{nr}}[\eta,\frac{1}{\eta^p-\eta}]~\widehat{}~[x_1,x_2]/(x_1^p+\frac{\eta^p-\eta}{(p/\eta)^{p-1}-1}x_2, x_2^p+\frac{(p/\eta)^p-(p/\eta)}{\eta^{p-1}-1}x_1, (x_1x_2)^{p-1}-p)$$
\end{cor}

\section{The action of \texorpdfstring{$\GL_2(\Q_p)$}{} on \texorpdfstring{$\widehat{\Sigma_1^{nr}}$}{} and a descent \texorpdfstring{$\widehat{\Sigma_1}$}{} to \texorpdfstring{$\Z_{p^2}$}{}} \label{secact}

Recall that we fix an embedding $\Z_{p^2}\hookrightarrow \widehat{\Z_p^{nr}}$. In this section, I want to describe the action of $\GL_2(\Q_p)$ on $\widehat{\Sigma_1^{nr}}$. As a corollary, we can descend the formal model from $\widehat{\Z_p^{nr}}$ to $\Z_{p^2}$ by taking the '$p$-invariants', where $p$ is considered as an element in $\GL_2(\Q_p)$. This descent is not quite canonical. On the other hand, as we explained in the introduction, it suffices to prove theorem \ref{gs} when the central character is trivial on $p$, and this is exactly the descent we are considering here.

Denote the canonical morphism $\widehat{\Sigma_1^{nr}}\to X_0$ by $\pi$ and $\pi^{-1}(\widehat{\Omega_e}\hat{\otimes}\widehat{\Z_p^{nr}})$ by $\widehat{\Sigma_{1,e}^{nr}}$, $\pi^{-1}(\widehat{\Omega_s}\hat{\otimes}\widehat{\Z_p^{nr}})$ by $\widehat{\Sigma_{1,s}^{nr}}$, for edge $e$ and vertex $s$. Then $\{\widehat{\Sigma_{1,e}^{nr}}\}_e$ is an open covering of $\widehat{\Sigma_1^{nr}}$, such that each open set has a nice description as in the previous section. Then the action of $\GL_2(\Q_p)$ on this covering can be identified with the action on Bruhat-Tits tree.

Now let $s'_0$ be the central vertex defined in section \ref{bfd}. Then, $\GL_2(\Z_p)$ acts on $\widehat{\Sigma_{1,s'_0}^{nr}}$. I want to write down explicitly this action under the identification in \ref{strcor2}. Since $\pi$ is $\GL_2(\Z_p)$-equivariant, we only need to describe the action on $x_1,x_2$. However it's clear from the equation in \ref{strcor2} that $x_2$ can be expressed using $x_1$ because $\eta^p-\eta$ is invertible. So it suffices to describe the action on $x_1$.

We first observe that $T_0^*/(\Pi_0^*T_1^*)\simeq \cL_1/d_2(\cL_2^{\otimes p})$ is a free $O_\eta/p$-module of rank one with a basis $x_1$. Recall $O_\eta=\widehat{\Z_p^{nr}}[\eta,\frac{1}{\eta^p-\eta}]~\widehat{}$. In section \ref{bfd}, we have an explicit description \eqref{actd} of the action of $\GL_2(\Z_p)$ on $T_0^*$, which is given by: $f(\eta)e_0^*\mapsto\frac{1}{b\eta+d}f(\frac{a\eta+c}{b\eta+d})e_0^*,g=\g$ for some basis $e_o^*$. So if we write $x_1=f(\eta)e_0^*$, for some $f(\eta)\in (O_\eta/p)^\times$, then the action of $\GL_2(\Z_p)$ on $x_1$ in $O_\eta/p$ is:
\begin{eqnarray} \label{act1w}
g(x_1)=\frac{1}{b\eta+d}f(\frac{a\eta+c}{b\eta+d})f(\eta)^{-1}x_1
\end{eqnarray}

Notice that on $\widehat{\Sigma_{1,s}^{nr}}$, $x_1^{p+1}\equiv(\eta^p-\eta)x_1x_2=(\eta^p-\eta)(-u)^{-1/(p-1)}\widetilde{\lambda_{\cL_1}}~(\modd~p)$. Thanks to proposition \ref{invp}, we know how $g=\g$ acts on the right hand side:
$$g((\eta^p-\eta)(-u)^{-1/(p+1)}\widetilde{\lambda_{\cL_1}})=((\frac{a\eta+c}{b\eta+d})^p-\frac{a\eta+c}{b\eta+d})(-u)^{-1/(p-1)}(ad-bc)^{-1}\widetilde{\lambda_{\cL_1}}$$
Here we use the fact $\chi_1(det(g))\equiv det(g)(\modd~p)$. An easy computation shows this is just $(\frac{1}{b\eta+d})^{p+1}(\eta^p-\eta)(-u)^{-1/(p-1)}\widetilde{\lambda_{\cL_1}}$.

But from \eqref{act1w}, $g(x_1)^{p+1}$ is $(\frac{1}{b\eta+d})^{p+1}(f(\frac{a\eta+c}{b\eta+d})f(\eta)^{-1})^{p+1}x_1^{p+1}$. Compare both expressions, we have:
$$f(\frac{a\eta+c}{b\eta+d})^{p+1}=f(\eta)^{p+1}$$
for any $g=\g\in \GL_2(\Z_p)$.

Since $f(\eta)\in (O_\eta/p)^\times=\overbar{\F_p}[\eta,\frac{1}{\eta^p-\eta}]^\times$, it can only have poles and zeros at $\F_p$-rational points. Now $\GL_2(\Z_p)$ acts transitively on these points, so $f(\eta)$ has to be a constant. In other words,

\begin{prop} \label{modpact}
The action of $\GL_2(\Z_p)$ on the special fibre of $\widehat{\Sigma_{1,s'_0}^{nr}}$ is given by:
$$g(x_1)\equiv\frac{1}{b\eta+d}x_1~(\modd~p),~g=\g\in\GL_2(\Z_p).$$ 
\end{prop}

\begin{cor}
This action factors through $\GL_2(\F_p)$.
\end{cor}

What's the action of $\GL_2(\Z_p)$ on $\widehat{\Sigma_{1,s'_0}^{nr}}$? Using the proposition above, we can write
$$g(x_1)^{p+1}=(\frac{1}{b\eta+d})^{p+1}x_1^{p+1}(1+ph(\eta))$$
for some $h(\eta)\in O_\eta$ which only depends on $g$. Then,
\begin{prop} \label{actk}
$$g(x_1)=\frac{1}{b\eta+d}x_1(1+ph(\eta))^{1/(p+1)}$$
where $(1+ph(\eta))^{1/(p+1)}=1+\frac{1}{p+1}ph(\eta)+\cdots$ is the binomial expansion.
\end{prop}

Now let $e_0$ be the edge that connects the central vertex $s'_0$ and the vertex $s_0$ that corresponds to $\GL_2(\Z_p)\Q_p^\times\cdot w$, where $w=\w$. Then $w$ acts on $\widehat{\Sigma_{1,e_0}^{nr}}$. What is it?

We fix an isomorphism of $\widehat{\Sigma_{1,e_0}^{nr}}$ with the explicit formal scheme described above. On $\widehat{\Omega_{e_0}}\hat{\otimes}\widehat{\Z_p^{nr}}\simeq \Spf O_{\zeta,\eta}$, the action of $w$ is given by:
\[
\eta\mapsto \frac{p}{-\eta}=-\zeta,~\zeta\mapsto \frac{p}{-\zeta}=-\eta
\]
and acts as the (lift) of arithmetic Frobenius on $\widehat{\Z_p^{nr}}$. Notice that $w$ interchanges $\cL_1$ and $\cL_2$ because it acts semi-linearly (over $\widehat{\Z_p^{nr}}$). Using this, it's not hard to see $w$ has the form:
\[
x_1\mapsto w_1x_2,~x_2\mapsto w_2x_1,
\]
where $w_1,w_2\in O_{\zeta,\eta}^\times$. 

An easy computation shows that $w_1,w_2$ must satisfy the following relation:
\begin{eqnarray}
w_1^p=-w_2
\end{eqnarray}

Since $w\in\{g\in\GL_2(\Q_p),det(g)\in p^\Z\}$, we can apply proposition \ref{invp} which tells us $x_1x_2=\widetilde{\lambda_{\cL_1}}$ is invariant by $w$. So, 
\begin{eqnarray}
w_1w_2=1
\end{eqnarray}
Combine these together, we get:

\begin{lem} \label{wact}
Action of $w=\w$ on $\widehat{\Sigma_{1,e_0}^{nr}}$ is given by:
$$x_1\mapsto w_1x_2,~x_2\mapsto w_1^{-1}x_1,$$
where $w_1\in \Z_{p^2}^\times$ satisfies $w_1^{p+1}=-1$. 
\end{lem}

Now we are ready to prove the main result of this section:
\begin{prop} \label{descent}
$\widehat{\Sigma_1^{nr}}$ can be descended to a formal scheme $\widehat{\Sigma_1}$ over $\Z_{p^2}$. In fact, $\widehat{\Sigma_1}=\widehat{\Sigma_1^{nr}}^p$, the formal scheme defined by the $p\in\GL_2(\Q_p)$-invariant sections of $\widehat{\Sigma_1^{nr}}$. 
\end{prop}
\begin{proof}
It suffices to prove this locally, so we only need to work on $\widehat{\Sigma_{1,e}^{nr}}$. Since $\GL_2(\Q_p)$ acts transitively on this covering, and $p$ is in the center of $\GL_2(\Q_p)$, we can just work with $\widehat{\Sigma_{1,e_0}^{nr}}$. $\widehat{\Omega_{e_0}}\hat{\otimes}\widehat{\Z_p^{nr}}$ certainly descends to $\Z_{p^2}$. The question is whether the descents of $\cL_1,\cL_2,d_1,d_2$ are effective. We show this by explicit computations.

Choose $c\in\Z_p^{nr}$ such that $c^{p+1}=v_1w_1^{-1}$, where $v_1$ is a choice of $p-1$-th root of $-1$, then $c$ is a root of unity, $\Fr(c)=c^p$. Define $e=cx_1,e'=c^{-1}x_2$. We have:
$$w^2(e)=w(\Fr(c)w_1x_2)=w(c^pw_1x_2)=c^{p^2}w_1^pw_1^{-1}x_1=c^{p^2-1}w_1^{p-1}e=-e$$
Similarly, $w^2(e')=-e'$. Notice that $p=-w^2$ and $-1$ acts on $x_1$ as $\chi_1(-1)^{-1}=-1$ (the action of $\Z_p^\times$ in $\GL_2(\Q_p)$ is the inverse of the action of it in $O_D^\times$). So $e,e'$ are invariant by $p$, $\cL_1,\cL_2$ can be descended to $\Z_{p^2}$.

What about $d_1,d_2$? Now $d_1:e^{\otimes p}\mapsto -c^{p+1}\frac{\eta^p-\eta}{\zeta^{p-1}-1}e'$. Since $c^{p+1}=(-1)^{1/(p-1)}w_1^{-1}\in\Z_{p^2}$, $d_1$ is defined over $\Z_{p^2}$. Similar argument works for $d_2$. And we're done.
\end{proof}

\begin{rem}
Sometimes $e$ is also denoted for an edge of a graph. I hope that it is clear in the context that whether $e$ refers to an edge or a section of $\cL_1$ (locally).
\end{rem}

\begin{cor}  \label{cact}
\hspace{2em}
\begin{enumerate}
\item The action of $\GL_2(\Q_p)$ can also be defined over $\widehat{\Sigma_1}$.
\item $\widehat{\Sigma_1}$ has an open covering $\{\widehat{\Sigma_{1,e}}\}_e$ indexed by the edges of the Bruhat-Tits tree, such that this identification is $\GL_2(\Q_p)$-equivariant.
\item $\widehat{\Sigma_{1,e}}$ is isomorphic to $\Spf O_{e,e'}=$
$$\Spf \frac{\Z_{p^2}[\zeta,\eta,\frac{1}{1-\zeta^{p-1}},\frac{1}{1-\eta^{p-1}},e,e']~\widehat{}~}
{(e^p+v_1w_1^{-1}\frac{\eta^p-\eta}{\zeta^{p-1}-1}e', e'^p+v_1^{-1}w_1\frac{\zeta^p-\zeta}{\eta^{p-1}-1}e, (ee')^{p-1}-p,\eta\zeta-p)}$$
where $w_1=(-1)^{\frac{1}{p+1}}$ is a $(p^2-1)$-th root of unity, $v_1$ is a choice of $(p-1)$-th root of $-1$.
\item The action of $w$ on $\widehat{\Sigma_{1,e_0}}$ is given by
\[e\mapsto v_1e',~e'\mapsto v_1^{-1}e.\]
\end{enumerate}
\end{cor}

\begin{rem}
The reason that everything can be defined over $\Z_{p^2}$, I believe, is that the universal formal group can be defined over $\Z_{p^2}$. This is because when we formulate the moduli functor it represents, the `unique' $2$-dimensional special formal group of height $4$ and all endomorphisms can be defined over $\F_{p^2}$. 
\end{rem}

\section{A semi-stable model of \texorpdfstring{$\widehat{\Sigma_1}$}{}} \label{ssm}
In this section, our goal is to work out a semi-stable model of $\widehat{\Sigma_1}$ as a formal scheme over $\Z_p$ (not $\Z_{p^2}$!). Notice that $\widehat{\Sigma_1}$ has a structural map to $\Spec\Z_{p^2}$. Hence if we change our base from $\Z_p$ to $O_{F_0}$, then 
\begin{eqnarray}
\widehat{\Sigma_1}\times_{\Spec \Z_p} \Spec O_{F_0}\simeq \widehat{\Sigma_1}\sqcup \widehat{\Sigma_1'}
\end{eqnarray}
Here $\widehat{\Sigma_1'}$ is the same scheme $\widehat{\Sigma_1}$ but with twisted map to $O_{F_0}$. Recall that $F_0$ is the unique unramified quadratic extension of $\Q_p$, and we fix an isomorphism between it and $\Q_{p^2}$ in the beginning. Hence we may identify $\widehat{\Sigma_1}$ as a formal scheme over $O_{F_0}$. Therefore we only need to work over the scheme $\widehat{\Sigma_1}$ as a scheme over $\Spec \Z_{p^2}$, and use the equation above to translate everything into the $\Z_p$-scheme $\widehat{\Sigma_1}$. I hope this won't cause too much confusion.

I say a formal scheme $X$ is a semi-stable curve over $\Spec R$, where $R$ is a complete discrete valuation ring, if:
\begin{enumerate}
\item The generic fibre of $X$ is smooth over the generic fibre of $\Spec R$.
\item The special fibre of $X$ is reduced.
\item Each irreducible component of the special fibre of $X$ is a divisor on $X$.
\item Each singular point has an \'{e}tale neighborhood that is \'{e}tale over $\Spec R[x,y]/(xy-\pi_R)$, where $\pi_R$ is an uniformizer of $R$.
\end{enumerate}

Back to our situation, we first work locally on $\widehat{\Sigma_1}$, so we just work with $\widehat{\Sigma_{1,e}}$. Moreover we can assume $e=e_0$ defined in the previous section and use the results there.

First notice that in $O_{e,e'}$ (see the notation in corollary \ref{cact}), $ee'=\uu_1^{-p-1}\widetilde{\lambda_{\cL_1}}$ (see the equation before corollary \ref{strcor1} and recall in the proof of proposition \ref{descent}, $x_1x_2=ee'$), so it is a globally defined section on $\widehat{\Sigma_1}$, and satisfies $(ee')^{p-1}-p$. Now if we do base change from $\Z_{p^2}$ to $\Z_{p^2}[p^{1/(p-1)}]$, the generic fibre of $\Spf O_{e,e'}[p^{1/(p-1)}]$ will split into $p-1$ connected components. Each connected component corresponds to a choice of $(p-1)$-th root of $p$. Adjoining $ee'/{p^{1(p-1)}}=\uu_1^{-p-1}\widetilde{\lambda_{\cL_1}}/{p^{1(p-1)}}$ into $O_{e,e'}[p^{1/(p-1)}]$, which I would like to call it $O_{e,e'}^1$, the formal scheme also splits into $p-1$ connected components. Namely, 
\[O_{e,e'}^1=\prod_{\varpi_1^{p-1}=p}O_{e,e',\varpi_1}^1.\]

Explicitly, $O_{e,e',\varpi_1}^1$ is:
$$\frac{\Z_{p^2}[p^{1/(p-1)}][\eta,\zeta,e,e',\frac{1}{\eta^{p-1}-1},\frac{1}{\zeta^{p-1}-1}]~\widehat{}}{(e^p+v_1w_1^{-1}\frac{\eta^p-\eta}{\zeta^{p-1}-1}e', e'^p+v_1^{-1}w_1\frac{\zeta^p-\zeta}{\eta^{p-1}-1}e,ee'-\varpi_1)}$$

Now, we have: (write $\varpi_1$ as $p^{1/(p-1)}$)
$$e^{p+1}=e^p\cdot e=-v_1w_1^{-1}\frac{\eta^p-\eta}{\zeta^{p-1}-1}e'e=-w_1^{-1}\frac{\eta^p-\eta}{\zeta^{p-1}-1}v_1p^{1/(p-1)}=-w_1^{-1}\frac{\eta^p-\eta}{\zeta^{p-1}-1}(-p)^{1/(p-1)}$$
Recall $v_1$ is a $p-1$-th root of $-1$. This clearly shows that if we adjoin $p^2-1$-th root of $-p$, then the normalization of this ring contains $e/((-p)^{1/(p-1)})^{1/(p+1)}=e/(-p)^{1/(p^2-1)}$. Similarly, $e'/(-p)^{1/(p^2-1)}$ is also contained in the normalization. This suggests us to define:

\begin{definition} \label{defoF}
Let $\varpi$ be a fixed choice of $(-p)^{1/(p^2-1)}$. Define $F=F_0[\varpi]$, and $O_F$ as the ring of integers inside $F$.
\end{definition}
We change our base from $\Spec \Z_{p^2}$ to $\Spec O_F$ via the fixed identification between $\Q_{p^2}$ and $F_0$, and take the normalization of $O_{e,e',\varpi_1}[\varpi]$ (it's not hard to verify it's integral). Denote the normalization by $\widetilde{O_{e,e',\varpi_1}[\varpi]}$. I claim basically this is just adjoining $e/\varpi,e'/\varpi$. 

\begin{lem} \label{sms}
\[\widetilde{O_{e,e',\varpi_1}[\varpi]}=\frac{O_F[\eta,\zeta,\frac{1}{\eta^{p-1}-1},\frac{1}{\zeta^{p-1}-1},\frac{e}{\varpi},\frac{e'}{\varpi}]~\widehat{}}{((\frac{e}{\varpi})^{p+1}+v_1w_1^{-1}\xi\frac{\eta^p-\eta}{\zeta^{p-1}-1},(\frac{e'}{\varpi})^{p+1}+v_1^{-1}w_1\xi\frac{\zeta^p-\zeta}{\eta^{p-1}-1},\frac{e}{\varpi}\frac{e'}{\varpi}-\xi\varpi^{p-1})},\]
where $\xi=\frac{\varpi_1}{\varpi^{p+1}}$ is a $(p-1)$-th root of $-1$.
\end{lem}
\begin{proof}
It's clear both sides become the same after inverting $p$ and certainly the right hand side is contained in the left hand side. Thus it suffices to prove the right hand side is normal. First, since the generic fibre is smooth, there is no singular point on the generic fibre. Now if we modulo $\varpi$, the uniformizer, it's easy to see the only singular point is the maximal ideal $(e/\varpi,e'/\varpi,\varpi)$. We only need to show $(e/\varpi,e'/\varpi)$ is a regular sequence. Simple calculations indicate that the right hand side is $p$-torsion free so $e/\varpi$ is not a zero divisor. In fact this already proves that the right hand side is integral. Modulo $e/\varpi$, the right hand side becomes $\Z_{p^2}[\varpi]/(\varpi^{p-1})[\zeta,e'/\varpi]/((e'/\varpi)^{p+1}+a(\zeta^p-\zeta))$ for some unit $a$. The element $e'/\varpi$ is clearly neither a zero divisor, nor a unit. So we're done.
\end{proof}

\begin{rem}
The special fibre of $\widetilde{O_{e,e',\varpi_1}[\varpi]}$ has two irreducible components, defined by $\frac{e}{\varpi}=0$ and $\frac{e'}{\varpi}=0$. Each one maps to an irreducible component of the special fibre of $\widehat{\Omega_{e_0}}\times_{\Spec \Z_p}\Spec O_F$, and has the form $\F_{p^2}[x,y,\frac{1}{y^{p-1}-1}]/(y^p-y-cx^{p+1})$, where $c$ is some root of unity. So each irreducible component is smooth and is an open set of an Artin-Schreier curve. In fact, if we do not split these connected components, then the special fibre is isomorphic to $\F_{p^2}[x,y,\frac{1}{y^{p-1}-1}]/((y^p-y)^{p-1}+w_1^{-2}x^{p^2-1})$, which is an open set of a twist of Deligne-Lustzig variety of $\GL_2(\F_p)$ (see section 2 of \cite{DL}). More precisely, if we invert $x$ and define $X=1/x,Y=y/x$, this curve now has the form $(XY^p-YX^p)^{p-1}=-w_1^{-2}$.
\end{rem}

Notice that $\widetilde{O_{e,e',\varpi_1}[\varpi]}$ is not semi-stable, because locally the singular point is defined by $\frac{e}{\varpi}\frac{e'}{\varpi}-\varpi^{p-1}\xi$, where $\xi$ is some unit. To get a semi-stable model, keep blowing up the singular points until our scheme becomes regular. In fact, we need to blow up $[(p-1)/2]$ times. On the level of special fibre, this singular point will be replaced by $p-2$ rational curves in this process. After this, we finally get our desired semi-stable model of $\widehat{\Sigma_{1,e_0}}\times_{\Spec \Z_{p^2}}\Spec O_F$.

So far we have been working locally on $\widehat{\Sigma_1}$, but our construction above can be done globally. First, we change the base to $\Spec O_F$ and adjoin $u_1^{-p-1}\widetilde{\lambda_{\cL_1}}/\varpi^{p-1}$ (equivalently, $\widetilde{\lambda_{\cL_1}}/\varpi^{p-1})$. Here, since the difference between $\varpi^{p-1}$ and a $(p-1)$-th root of $p$ is a $(p-1)$-th root of $-1$, it doesn't matter which one we use. Then our formal scheme will split into $p-1$ connected components, indexed by $(p-1)$-th roots of $-1$. Now take the normalization of each connected component. Call the total space $\wtso$. For each component, it is clear from the above explicit local description that the dual graph of its special fibre is the same as $\widehat{\Omega}$'s, which is nothing but the Bruhat-Tits tree. Finally, blow up each singular point to get rid of singularities and we end up with a semi-stable model of $\widehat{\Sigma_1}\times_{\Spec\Z_{p^2}}\Spec O_F$.

\begin{thm} \label{mst}
$\widehat{\Sigma_1}$ (over $\Spec \Z_{p^2}$) has a semi-stable model $\widehat{\Sigma_{1,O_F}}$ over $O_F$, such that:
\begin{enumerate}
\item $\widehat{\Sigma_{1,O_F}}$ has $(p-1)$ connected components, indexed by $(p-1)$th roots of $-1$.
\item The dual graph of the special fibre of each connected component is the graph adding $p-2$ vertices to each edge of the Bruhat-Tits tree.
\item Vertices that come from Bruhat-Tits tree correspond to some Artin-Schreier curves ($y^{p+1}=c(x^p-x)$ in $\PP^2$, where $c\in\F_{p^2}^\times$). Singular points are points with $y=0$. If we put the $p-1$ connected components together, then a dense open set of it is isomorphic to the Deligne-Lusztig variety of $\GL_2(\F_p)$ over any algebraically closed field. 
\item Other vertices correspond to rational curves. Singular points are zero and infinity.
\end{enumerate}
\end{thm}
\begin{proof}
We only need to prove our assertion for the special fibre. In the previous discussion, we already know the dual graph of the special fibre of each connected component of $\wtso$ is the Bruhat-Tits tree. Since blow-ups replace each singular point by $p-2$ rational curves, everything is clear.
\end{proof}

Let $\hat{\pi}$ (resp. $\tilde{\pi}$) be the canonical map from $\widehat{\Sigma_{1,O_F}}$ (resp. $\wtso$) to $\widehat{\Omega}\times_{\Spec \Z_p}\Spec O_F$. We can define $\widehat{\Sigma_{1,O_F,e}}$ (resp. $\widetilde{\Sigma_{1,O_F,e}}$) as $\hat{\pi}^{-1}(\widehat{\Omega_e}\times_{\Spec \Z_p}\Spec O_F)$ (resp. $\tilde{\pi}^{-1}(\widehat{\Omega_e}\times_{\Spec \Z_p}\Spec O_F)$) for each edge $e$ of the Bruhat-Tits tree. Similarly we can define $\widehat{\Sigma_{1,O_F,s}}=\widetilde{\Sigma_{1,O_F,s}}$ for each vertex $s$. Define $\widehat{\Sigma_{1,O_F,e,\xi}},\widehat{\Sigma_{1,O_F,s,\xi}},\widetilde{\Sigma_{1,O_F,e,\xi}},\widetilde{\Sigma_{1,O_F,s,\xi}}$, where $\xi$ is a $(p-1)$th root of $-1$, as the corresponding connected component of $\widehat{\Sigma_{1,O_F,e}},\widehat{\Sigma_{1,O_F,s}},\widetilde{\Sigma_{1,O_F,e}},\widetilde{\Sigma_{1,O_F,s}}$. Note that in the notation of lemma \ref{sms}, $\widetilde{\Sigma_{1,O_F,e,\xi}}=\Spf\widetilde{O_{e,e',\varpi^{p+1}\xi}[\varpi]}$

In lemma \ref{sms}, we have an explicit description of $\widetilde{\Sigma_{1,O_F,e}}$. To simplify notations, I will use $\e,\e'$ for $\frac{e}{\varpi},\frac{e'}{\varpi}$. Now let $s'$ be an even vertex (for example, central vertex $s'_0$). It's not hard to see that
\begin{eqnarray} \label{eeq}
\widehat{\Sigma_{1,O_F,s',\xi}}=\widetilde{\Sigma_{1,O_F,s',\xi}}\simeq\Spf O_F[\eta,\frac{1}{\eta^p-\eta},\e]/(\e^{p+1}+v_1w_1^{-1}\xi\frac{\eta^p-\eta}{(p/\eta)^{p-1}-1})\\ \label{eeqt}
\widehat{\Sigma_{1,O_F,s'}}=\widetilde{\Sigma_{1,O_F,s'}}\simeq\Spf O_F[\eta,\frac{1}{\eta^p-\eta},\e]/(\e^{p^2-1}+w_1^2(\frac{\eta^p-\eta}{(p/\eta)^{p-1}-1})^{p-1})
\end{eqnarray}

If $s$ is an odd vertex, then similarly we have:
\begin{eqnarray} \label{oeq}
\widehat{\Sigma_{1,O_F,s,\xi}}=\widetilde{\Sigma_{1,O_F,s,\xi}}\simeq\Spf O_F[\zeta,\frac{1}{\zeta^p-\zeta},\e']/(\e'^{p+1}+v_1^{-1}w_1\xi\frac{\zeta^p-\zeta}{(p/\zeta)^{p-1}-1})\\ \label{oeqt}
\widehat{\Sigma_{1,O_F,s}}=\widetilde{\Sigma_{1,O_F,s}}\simeq\Spf O_F[\zeta,\frac{1}{\zeta^p-\zeta},\e']/(\e'^{p^2-1}+w_1^{-2}(\frac{\zeta^p-\zeta}{(p/\zeta)^{p-1}-1})^{p-1})
\end{eqnarray}

\begin{rem} \label{defogv}
If we view $\widehat{\Sigma_1}$ as a $\Z_p$-scheme, then $\widehat{\Sigma_1}\times_{\Spec \Z_p} \Spec O_F$ has a semi-stable model over $\Spec O_F$, which I call $\widehat{\Sigma_{1,O_F}^{(0)}}$. It is canonically isomorphic to $\widehat{\Sigma_{1,O_F}}\sqcup \widehat{\Sigma'_{1,O_F}}$, where $\widehat{\Sigma'_{1,O_F}}$ is isomorphic with $\widehat{\Sigma_{1,O_F}}$ as a scheme, but the structure morphism to $\Spec O_F$ is twisted: $O_F\to O_F$ is the unique automorphism that fixes $\varpi$ and acts as Frobenius on $O_{F_0}$. We use $g_\varphi$ to denote it as an element in $\Gal(F/\Q_p)$. 

From now on, I will use the exponent $(0)$ for everything that is base changed from $\Z_p$ to $O_F$. For example, we can define $\widehat{\Sigma_{1,O_F,s}^{(0)}},\widehat{\Sigma_{1,O_F,s,\xi}^{(0)}},\cdots$. Also we use the exponent $'$ for things with same underlying scheme but with twisted structure morphism to $O_F$. For example $\widehat{\Sigma_{1,O_F,s}'},\widehat{\Sigma_{1,O_F,s,\xi}'},\cdots$. Under these notations, we have $\widehat{\Sigma_{1,O_F,s}^{(0)}}=\widehat{\Sigma_{1,O_F,s}}\sqcup \widehat{\Sigma_{1,O_F,s}'},\cdots$.
\end{rem}

\section{The action of \texorpdfstring{$\GL_2(\Z_p),~\Gal(F/F_0),~O_D^\times$}{} on \texorpdfstring{$\wtso$}{} and \texorpdfstring{$\widehat{\Sigma_{1,O_F}}$}{}} \label{act}

It is clear that we have an action of $\GL_2(\Q_p)$ on $\widehat{\Sigma_1}\times_{\Spec \Z_p}\Spec O_F$ by acting on the first factor, and it extends naturally to our semi-stable $\widehat{\Sigma_{1,O_F}^{(0)}}$. Notice that $\GL_2(\Q_p)$ will interchange $\widehat{\Sigma_{1,O_F}}$ and $\widehat{\Sigma'_{1,O_F}}$, so it does not act on $\widehat{\Sigma_{1,O_F}}$. The reason is that $g\in\GL_2(\Q_p)$ acts on $\Z_{p^2}$ by $\Fr^{v_p(det(g))}$. However, $\GL_2(\Z_p)$ acts on $\widehat{\Sigma_{1,O_F}}$.

So how does $\GL_2(\Z_p)$ act on the central component $\widehat{\Sigma_{1,O_F,s'_0}}$ of $\widehat{\Sigma_{1,O_F}}$? We just have an explicit description above \eqref{eeq},\eqref{eeqt}. We will fix this identification from now on.
\begin{eqnarray} \label{les'0}
\widehat{\Sigma_{1,O_F,s'_0,\xi}}=\Spf O_{F_0}[\varpi][\eta,\frac{1}{\eta^p-\eta},\e]/(\e^{p+1}+v_1w_1^{-1}\xi\frac{\eta^p-\eta}{(p/\eta)^{p-1}-1})\\
\widehat{\Sigma_{1,O_F,s'_0}}=\Spf O_{F_0}[\varpi][\eta,\frac{1}{\eta^p-\eta},\e]/(\e^{p^2-1}+w_1^2(\frac{\eta^p-\eta}{(p/\eta)^{p-1}-1})^{p-1})
\end{eqnarray}

\begin{prop}
\begin{enumerate}
\item The action of $\GL_2(\Z_p)$ on $\widetilde{\Sigma_{1,O_F,s'_0}}=\widehat{\Sigma_{1,O_F,s'_0}}$ is given by 
\begin{eqnarray}
g(\e)\equiv \frac{1}{b\eta+d}\e~(\modd~p),~g=\g\in\GL_2(\Z_p)
\end{eqnarray}
So it factors through $\GL_2(\F_p)$ when acting on the special fibre.
\item $g\in \GL_2(\Z_p)$ maps $\widetilde{\Sigma_{1,O_F,s'_0,\xi}}$ to $\widetilde{\Sigma_{1,O_F,s'_0,\xi\chi_1(det(g))}}$.
\end{enumerate}
\end{prop}
\begin{proof}
Since $\e=\frac{e}{\varpi}$, we can apply proposition \ref{modpact} here and everything is clear except for the claim that how it interchanges connected components. Notice that `$\xi$' component is defined by  $\uu_1^{-p-1}\widetilde{\lambda_{\cL_1}}-\varpi^{p+1}\xi$. So our claim follows from proposition \ref{invp}.
\end{proof}

\begin{cor} \label{iDL}
The identification of the special fibre of $\widetilde{\Sigma_{1,O_F,s'_0}}$ with a Deligne-Lusztig variety is $\GL_2(\F_p)$-equivariant.
\end{cor}

We will come back to this point later when we review Deligne-Lusztig theory.

For $\widehat{\Sigma_{1,O_F}}$, since we change our base from $O_{F_0}$ to $O_F$, there is a natural action of $\Gal(F/F_0)$.
\begin{definition} \label{defoto}
$\tomega_2:\Gal(F/F_0)\to O_{F_0}^\times$ is the character given by $\tomega_2(g)=\frac{g(\varpi)}{\varpi}$.
\end{definition}

Any other character is a multiple of $\tomega_2$. 

\begin{rem}
Another equivalent definition of $\tomega_2$ is as follows: By local class field theory, it suffices to give a character of $F_0^\times$. This character is trivial on $p^\Z$, and on $O_{F_0}^\times$, it is given by first reducing modulo $p$, then taking the inverse of the Teichm\"uller character. Our convention on the local Artin map is that uniformizers correspond to arithmetic Frobenius elements.
\end{rem}

\begin{rem} \label{intdef}
Recall that we defined two characters $\chi_1,\chi_2$ of $(O_D/\Pi)^\times$ (see definition \ref{fch}). Using the above remark, the relation of $\chi_1$ and $\tomega_2$ can be described in the following diagram:
\[\begin{CD}
\Z_{p^2}^\times\simeq O_{F_0}^\times @>Art_{F_0} >> \Gal(\overbar{F_0}/F_0)^{ab}\\
@VVV                               @VV\tomega_2^{-1}V\\
O_D^\times   @>\chi_1>> \Z_{p^2}^\times\simeq O_{F_0}^\times,
\end{CD}\]
where the left arrow is our fixed embedding of $\Z_{p^2}$ into $O_D$, $Art_{F_0}$ is the Artin map in local class field theory, the isomorphism between $\Z_{p^2}^\times$ and $O_{F_0}^\times$ is the one we fixed in the beginning.
\end{rem}

Under the isomorphism \eqref{eeq},\eqref{eeqt},\eqref{oeq},\eqref{oeqt}, we have:
\begin{prop}
The action of $g\in\Gal(F/F_0)$ is given by:
\begin{eqnarray}
g(\e)=\tomega_2(g)^{-1}\e\\
g(\e')=\tomega_2(g)^{-1}\e'
\end{eqnarray}
\end{prop}
This is trivial because $\e=\frac{e}{\varpi},\e'=\frac{e'}{\varpi}$.

The last group action we want to consider here is the action of $O_D^\times$.
\begin{prop}
Under the isomorphisms \eqref{eeq},\eqref{eeqt},\eqref{oeq},\eqref{oeqt}, for $d\in O_D^{\times}$
\begin{eqnarray}
d(\e)=\chi_1(d)\e\\
d(\e')=\chi_2(d)\e'
\end{eqnarray}
\end{prop}

\begin{rem} \label{remact}
The action of $O_D^\times$ on $\widehat{\Sigma'_{1,O_F}}$ is a twist of what we considered above:
\begin{eqnarray}
d(\e)=\Fr(\chi_1(d))\e=\chi_2(d)\e=\chi_1(d)^p\e,\forall d\in O_D^\times.\\
d(\e')=\Fr(\chi_2(d))\e'=\chi_1(d)\e'=\chi_2(d)^p\e',\forall d\in O_D^\times.
\end{eqnarray}
Here I identify $\widehat{\Sigma'_{1,O_F}}$ with $\widehat{\Sigma_{1,O_F}}$ but with twisted structure morphism. And by saying $\chi_2(d)$ I consider it as an element in the `$O_F$' coming from the structure map, not the $\Z_{p^2}$ coming from the original scheme $\widehat{\Sigma_1}$. However, the action of $\Gal(F/F_0)$ is the same, not twisted. Another way to see these is using a $g\in \GL_2(\Q_p)$ with $v_p(det(g))$ odd, then $g$ sends $\widehat{\Sigma_{1,O_F,s}}$ to $\widehat{\Sigma'_{1,O_F,sg}}$. Finally, $g_\varphi \in \Gal(F/\Q_p)$ interchanges $\widehat{\Sigma_{1,O_F}}$ and $\widehat{\Sigma'_{1,O_F}}$ by acting as Frobenius endomorphism on $O_{F_0}$ but fixes other things under the isomorphisms \eqref{eeq},\eqref{eeqt},\eqref{oeq},\eqref{oeqt}.
\end{rem}

\section{Another admissible open covering of Drinfel'd upper half plane and the generic fibre of \texorpdfstring{$\widehat{\Sigma_{1,O_F}}$}{}} \label{adoc}

In this section, we work on the generic fibre of everything we considered before. The main result of this section is a description of the generic fibre $\Sigma_{1,F}$ of $\widehat{\Sigma_{1,O_F}}$ (and a similar result for the generic fibre $\Sigma_{1,F}^{(0)}$ of $\widehat{\Sigma_{1,O_F}^{(0)}}$). 

Recall that $\Sigma_1$ is the generic fibre of $\widehat{\Sigma_1^{nr}}$. The latter one is defined by two line bundles $\cL_1,\cL_2$ and maps $d_1:\cL_1^{\otimes p}\to\cL_2,~d_2:\cL_2^{\otimes p}\to\cL_1$ (see the beginning of section \ref{fp}). Denote by $\cL_{1,\eta},\cL_{2,\eta},d_{1,\eta},d_{2,\eta}$ the restriction of corresponding thing to $\Sigma_1$, the generic fibre.

First we observe:
\begin{lem} \label{lftrv}
Any line bundle over $\X_0$, the generic fibre of the Drinfel'd upper half plane (and base changed to $\widehat{\Z_p^{nr}}$), is trivial.
\end{lem}

To do this we need another admissible open covering of $\X_0$, which is described in \cite{Dre} (``topological" analog) and in \cite{SS} in detail. Let me recall it now. 

Define $U_n(\C_p)=\{z\in C_p,|z|\leq p^n,|z-a|\geq p^{-n},\forall a\in \Q_p\}$, where $|~|$ is the canonical norm on $\C_p$ such that $|p|=p^{-1}$. Notice that we only need finitely many $a$ to define this set, so $U_n$ can be identified as an open set of $P^1$ by removing some open discs. Therefore $U_n$ is an affinoid space. In fact, we can identify it as an affinoid subdomain of a closed unit ball.

\begin{rem} \label{aco}
Another way to construct $U_n$ is by using the formal model we already have. We can define a distance of two vertices of Bruhat-Tits tree by couting the number of edges on the unique path between these two vertices. For example, two adjacent vertices have distance $1$ and any vertex has distance $0$ with itself. Now define $Z_n$ as the set of vertices having distance $\leq n$ with the central vertex. Let $\Omega_{U_n}$ be the union of $\Omega_e$ such that $e$ is an edge between two vertices in $Z_n$ and $\Omega_{U_0}=\Omega_{s'_0}$. Then $U_n$ is the generic fibre of $\Omega_{U_n}$.
\end{rem}

It is clear $U_n\subset U_{n+1}$ and $\bigcup U_n=\Omega$, the Drinfel'd upper half plane. Also it's not hard to verify the open covering $\{U_n\}$ is admissible. Let $O_{U_n}$ be the ring of rigid analytic functions on $U_n$ (over $\Q_p$). The key property we need is:

\begin{lem} \label{dense}
The image of the canonical inclusion $\phi_n:O_{U_{n+1}}\to O_{U_n}$ is dense under the canonical topology on $O_{U_n}$.
\end{lem}
\begin{proof}
Choose $a_1,\cdots,a_m \in \Q_p$ such that $\{B(a_i,p^{-n})\}_i$ is an open covering of $p^{-n}\Z_p$ in $\Q_p$, where $B(a_i,p^{-n})$ is the open ball centered at $a_i$ of radius $p^{-n}$ in $\Q_p$. Now when we define $U_n$, we can use $a_1,\cdots,a_m$ rather than $\forall a\in \Q_p$. Thus,
$$O_{U_n}=\{ F(z)=\sum_{k=0}^{+\infty}b_{0,k} (p^nz)^k+\sum_{i=1}^m \sum_{k=0}^{+\infty} b_{i,k} (\frac{p^n}{z-a_i})^k,~b_{i,k}\in\Q_p,~\lim_{k\to +\infty}b_{i,k}=0,\forall~i\}$$
We define a norm $|~|_n$ on $O_{U_n}$ by $|F(z)|_n=\sup_{i,k}|b_{i,k}|$. This is nothing but the supremum norm: $|f|_n=\sup_{x\in \Spm O_{U_n}}|f(x)|$. Now the $\Q_p$-algebra generated by $z,\frac{1}{z-a_i} ~(i=1,\cdots,m)$ is dense in $O_{U_n}$. But these functions are defined over $\Omega$, so live in $O_{U_{n+1}}$. We're done.
\end{proof}

\begin{rem} \label{norm}
Notice that in fact $p^nz,\frac{p^n}{z-a_i}~(i=1,\cdots,m)$ are affinoind generators of $O_{U_n}$ over $\Q_p$ in the sense there exists a surjective map from the Tate algebra $\Q_p<T_0,\cdots,T_m>$ to $O_{U_n}$ that sends $T_0$ to $p^nz$ and other $T_i$ to $\frac{p^n}{z-a_i}$. If we restrict $p^nz$ or $\frac{p^n}{z-a_i}$ to $U_{n-1}$, by definition of $U_{n-1}$, its norm is less than $1$ (in fact $\leq p^{-1}$). From this description, it's easy to see $U_{n-1}$ is relatively compact in $U_{n}$. See 6.3 of \cite{Bos} for precise definition. A direct corollary of this is that the inclusion map $O_{U_n}\to O_{U_{n-1}}$ is a strictly completely continuous map in the sense of Definition 1, 6.4 of \cite{Bos}. Another consequence is that $\Omega$ is a Stein-space defined in \cite{Kie}.
\end{rem}

Now we come back to the proof of the lemma \ref{lftrv}. We still need one more lemma:
\begin{lem}
Any line bundle on $U_n$ is trivial.
\end{lem}
\begin{proof}
It suffices to prove $O_{U_n}$ is a principal ideal domain. It's obvious that $O_{U_n}$ is regular hence normal. So we only need to show every maximal ideal of $O_{U_n}$ is principal. But we know $U_n$ is an affinoid subdomain of a (one dimensional) closed unit ball by removing several open discs centered at $\Q_p$-points, with radius$\in p^{\Z}$. Our claim follows from the fact that $\Q_p<T>$, the Tate algebra, is a PID (corollary 10, 2.2 of \cite{Bos}).
\end{proof}

\begin{proof}[Proof of lemma \ref{lftrv}]
I learned this argument from \cite{Kie} (proof of Satz 2.4.). Since $\{U_n\}_n$ is an admissible open covering of $\Omega$ and every line bundle on $U_n$ is trivial, a line bundle on $\Omega$ is equivalent with a $1-$cocycle: $\{f_{ij}\}_{i<j},~f_{ij}\in O_{U_i}^\times$, such that:
\[
f_{ij}\phi_{ji}(f_{jk})=f_{ik}
\]
for $i<j<k$. $\phi_{ji}$ is the canonical inclusion from $O_{U_j}$ to $O_{U_i}$. It's easy to see that $f_{12},f_{23},\cdots$ determine all $f_{ij}$. Two cocycles $\{f_{i(i+1)}\},~\{f'_{i(i+1)}\}$ define the same the line bundle if and only if there exists $\{g_i\},~g_i\in O_{U_i}^\times$, such that
\[
f_{i(i+1)}g_i\phi_i(g_{i+1})^{-1}=f'_{i(i+1)},~\forall i\geq 1
\] 
Now let $\{f_{i(i+1)}\}$ be a fixed cocycle. Define $g'_1=1\in O_{U_1}$. Thanks to lemma \ref{dense}, we can find $g'_{i+1} \in O_{U_i},~i\geq 1$ by induction, satisfying:
\[
|1-g'_if_{i(i+1)}\phi_i(g'_{i+1})^{-1}|_i<\frac{1}{2^i}
\]
This implies after modifying our cocycle, we can assume $|1-f_{i(i+1)}|_i<\frac{1}{2^i}$. Now define $g_i=\prod_{j=i}^{\infty}\phi_{ji}(f_{j(j+1)})^{-1}$, here $\phi_{ii}$ is the identity map. Notice that $|f|_j\geq |\phi_{ji}(f)|_{i}$ for $f\in O_{U_j}$, see remark \ref{norm}. So the infinite product makes sense by our assumption. But now $f_{i(i+1)}g_i\phi_i(g_{i+1})^{-1}=1$. Therefore it corresponds to a trivial line bundle.
\end{proof}

Although our proof is working over the base field $\Q_p$, but the argument still works if we change the base to other fields.
\begin{cor}
$\cL_{1,\eta},\cL_{2,\eta}$ are trivial line bundles.
\end{cor}

Now let $E_1$ be a basis of $\cL_{1,\eta}$ and $E_1^*$ the dual basis of $E_1$ under the isomorphism $\lambda_{\cL_1}$. Then $d_1,d_2$ become two elements $U_1,U_2$ in $H^0(\X_0,\cO_{\X_0})$ such that $\X_1$ is now defined by:
\[
\cO_{\X_0}[E_1,E_1^*]/(E_1^p-U_1E_1^*,(E_1^*)^p-U_2E_1)
\]
We know $E_1E_1^*=\widetilde{\lambda_{\cL_1}}$, so $U_1U_2=-w$ (see corollary \ref{cst}). $\Sigma_1$ is:
\[
\cO_{\X_0}[E_1,E_1^*]/(E_1^p-U_1E_1^*,(E_1^*)^p-U_2E_1,(E_1E_1^*)^{p-1}+w)
\]
Since $w$ is invertible on the generic fibre, so is $U_1$. We can write $E_1^*=E_1^pU_1^{-1}$.
\begin{prop}
\begin{eqnarray}
\Sigma_1=\cO_{\X_0}[E_1]/(E_1^{p^2-1}+U_1^{p-1}w)
\end{eqnarray}
In other words, $\Sigma_1$ is $\X_0$ adjoined with a $(p^2-1)$th root of a rigid analytic function on $\X_0$.
\end{prop}

\begin{rem}
If we are careful enough in the beginning and take $E_1$ to be $p\in\GL_2(\Q_p)$-invariant, we can descend our description to $O_{F_0}$. This means we have the same description of the generic fibre $\Sigma_{1,F}$ of $\widehat{\Sigma_{1,O_F}}$.
\end{rem}

\begin{cor} \label{stein}
$\Sigma_{1,F}$ is a Stein-space. 
\end{cor}
\begin{proof}
As we remarked before (remark \ref{norm}), $U_n$ is relatively compact in $U_{n+1}$. It's easy to see the open set of $\Sigma_{1,F}$ above $U_n$, which we denote by $V_{n,F}$ is an affinoid space and relatively compact in $V_{n+1,F}$.
\end{proof}

\section{De Rham cohomology of \texorpdfstring{$\Sigma_{1,F}$}{} and \texorpdfstring{$\Sigma_{1,F}^{(0)}$}{}} \label{dR}
Let $\Omega_{\Sigma_{1,F}}^1$ be the sheaf of holomorphic differential forms on $\Sigma_{1,F}$ and $\Omega_{\Sigma_{1,F}}^0=\cO_{\Sigma_{1,F}}$. Then we can consider the de Rham complex: 
\[
0\to \Omega_{\Sigma_{1,F}}^0\stackrel{d}{\to} \Omega_{\Sigma_{1,F}}^1
\]
where $d$ is the usual derivation. Define the de Rham cohomology as:
\begin{definition}
$H^i_{\dR}(\Sigma_{1,F})\defeq i$th hypercohomology of the de Rham complex. 
\end{definition}

\begin{rem}
In a series of papers (see \cite{Gro2},\cite{Gro1}), Grosse-Kl\"{o}nne introduced a theory of de Rham cohomology for rigid analytic spaces. His approach uses the overconvergent de Rham complex rather than the usual De Rham complex. However in our case, they are the same since $\Omega_{\Sigma_{1,F}}$ is a Stein space (Theorem 3.2 of \cite{Gro2}).
\end{rem}

Thanks to Kiehl, we know that all higher cohomology groups of $\Omega_{\Sigma_{1,F}}^0,\Omega_{\Sigma_{1,F}}^1$ vanish (Satz 2.4.2 of \cite{Kie}). So the de Rham cohomology is nothing but:
\begin{prop}
\begin{eqnarray}
H^0_{\dR}(\Sigma_{1,F})&=&\Ker(H^0(\Sigma_{1,F},\Omega_{\Sigma_{1,F}}^0)\stackrel{d}{\to}H^0(\Sigma_{1,F},\Omega_{\Sigma_{1,F}}^1))=F\\
H^1_{\dR}(\Sigma_{1,F})&=&\coker(H^0(\Sigma_{1,F},\Omega_{\Sigma_{1,F}}^0)\stackrel{d}{\to}H^0(\Sigma_{1,F},\Omega_{\Sigma_{1,F}}^1)) \\
H^i_{\dR}(\Sigma_{1,F})&=&0,\forall i\geq 2
\end{eqnarray}
\end{prop}

We can put a certain topology on $H^1_{\dR}(\Sigma_{1,F})$. This is done by writing:
$$H^0(\Sigma_{1,F},\Omega_{\Sigma_{1,F}}^i)=\varprojlim_n H^0(V_{n,F},\Omega_{\Sigma_{1,F}}^i)$$
for $i=0,1$. See the proof of corollary \ref{stein} for notations. Since each $H^0(V_{n,F},\Omega_{\Sigma_{1,F}}^i)$ is a Banach space and has a canonical topology, we can equip $H^0(\Sigma_{1,F},\Omega_{\Sigma_{1,F}}^i)$ with the projective limit topology. Now $V_{n,F}$ is relatively compact in $V_{n+1,F}$. As we remarked in \ref{norm}, the transition map from $H^0(V_{n+1,F},\Omega_{\Sigma_{1,F}}^i)$ to $H^0(V_{n,F},\Omega_{\Sigma_{1,F}}^i)$ is completely continuous. Using Corollary 16.6 of \cite{NFA}, we have: (notice that completely continuous map between two Banach spaces is compact, see Proposition 18.11 of \cite{NFA})
\begin{prop} \label{ref1}
$H^0(\Sigma_{1,F},\Omega_{\Sigma_{1,F}}^i),~i=0,1$ is a reflexive Fr\'{e}chet space.
\end{prop}
See page 55 of \cite{NFA} for the definition of reflexive.

\begin{prop} \label{clim}
The image of the derivation map $d:H^0(\Sigma_{1,F},\Omega_{\Sigma_{1,F}}^0)\to H^0(\Sigma_{1,F},\Omega_{\Sigma_{1,F}}^1)$ is closed.
\end{prop}
\begin{proof}
This is corollary 3.2. of \cite{Gro1}.
\end{proof}

\begin{cor} \label{Fre}
$H^1_{\dR}(\Sigma_{1,F})$ is a Fr\'{e}chet space.
\end{cor}

But how to compute de Rham cohomology? We need our semi-stable $\widehat{\Sigma_{1,O_F}}$ constructed in the section \ref{ssm}. Let $E(\widehat{\Sigma_{1,O_F}})$ (resp. $V(\widehat{\Sigma_{1,O_F}})$) be the set of singular points (resp. irreducible components) of the special fibre of $\widehat{\Sigma_{1,O_F}}$. By definition, we can identify them as the set of edges (resp. vertices) of the dual graph of the special fibre. Now fix an orientation for each edge $e\in E(\widehat{\Sigma_{1,O_F}})$, and we use $v^+(e)$ (resp. $v^-(e)$) to denote the target (resp. source) vertex of the orientation. 
\begin{definition} \label{tnbhd}
Let $U_e$ (resp. $U_v$) be the tubular neighborhood of the singular point indexed by $e$ (resp. irreducible component indexed by $v$). 
\end{definition}
It is clear that $\{U_v\}_v$ is an admissible open covering of $\Sigma_{1,O_F}$. Hence,

\begin{lem} \label{les}
We have a long exact sequence of de Rham cohomologies: 
\begin{eqnarray*}
0\to H^0_{\dR}(\Sigma_{1,F})
\to \prod_{v\in V(\widehat{\Sigma_{1,O_F}})}H^0_{\dR}(U_v)\stackrel{a}{\to}
\prod_{e\in E(\widehat{\Sigma_{1,O_F}})} H^0_{\dR}(U_e)\\
\stackrel{\partial}{\to}H^1_{\dR}(\Sigma_{1,F})\to
\prod_{v\in V(\widehat{\Sigma_{1,O_F}})}H^1_{\dR}(U_v)\stackrel{b}{\to}
\prod_{e\in E(\widehat{\Sigma_{1,O_F}})} H^1_{\dR}(U_e),
\end{eqnarray*}
where the arrows without notations are canonical restriction maps, and $a,b$ are the canonical restriction map to $v^+(e)$ minus the restriction map to $v^-(e)$ for an element indexed by $e$. 
\end{lem}

Here the de Rham cohomologies of $U_e,U_v$ are defined by the same method as above. We note that they are not affinoid but Stein spaces.

We first look at $U_e$, the tubular neighborhood of a singular point. It's not hard to see from the explicit description in lemma \ref{sms} that $U_e$ is an annulus $\{T, |\varpi|<|T|<1\}$. So its de Rham cohomology is: $H^0_{\dR}(U_e)=F$, generated by the constant function; $H^1_{\dR}(U_e)\simeq F$, generated by $\frac{dT}{T}$, where $T$ is a coordinate of $U_e$.

In lemma \ref{sms}, although we haven't resolved the singularities there, $\frac{d\e}{\e}$ still makes sense on the generic fiber, and it generates all $H^1_{\dR}(U_e)$ for any singular point $e$    above the singularity there. In fact, the process of resolving the singularities $xy-\varpi^n$ is just `dividing' the annulus into several small annuli. For example,  the tubular neighborhood of $xy-\varpi^n$ can be thought as the annulus $\{T,|\varpi|^n<T<1\}$. And for any $e$ above this singular point, $U_e$ can be identified as $\{T,|\varpi|^{l+1}<T<|\varpi|^l\}$ for some $l<n$.

Recall that $O_D^\times$ acts as characters on $\e$, so acts trivially on $H^0_{\dR}(U_e),H^1_{\dR}(U_e)$.

What about $U_v$? There are two possibilities. One is that $v$ corresponds to a rational curve. $U_v$ is an annulus and the result is the same as $U_e$. In particular $O_D^\times$ acts trivially on their de Rham cohomologies.

The other one is more interesting. We will compute it in the next section. Some notations here: recall that every such vertex $v$ can be indexed by $(s,\xi)$, where $s$ in a vertex of the Bruhat-Tits tree and $\xi$ satisfies $\xi^{p-1}=-1$. 
\begin{definition}\label{defosx}
We will use $(s,\xi)$ to denote such kind of vertex $v$ from now on. 
\end{definition}
\begin{definition} \label{u_s}
Denote the irreducible component  indexed by $(s,\xi)$ by $\overbar{U_{s,\xi}}$ and its generic fibre by $U_{s,\xi}$. We also denote the smooth loci of $\overbar{U_{s,\xi}}$ by $\overbar{U_{s,\xi}^0}$ (viewed as a subscheme in the special fibre of $\widehat{\Sigma_{1,O_F}}$). Notice that this is nothing but the special fibre of $\widehat{\Sigma_{1,O_F,s,\xi}}=\wtxs$. Define $\overbar{U_s}=\bigcup_{\xi^{p-1}=-1} \overbar{U_{s,\xi}}$, and similarly $\overbar{U_s}, \overbar{U_s^0}$.
\end{definition}

Recall that in the beginning, we fix a finite extension $E$ of $\Q_p$ that is large enough and define $\chi(E)$ as the set of characters of $(O_D/\Pi)^\times$ with values in $E^\times$.

$O_D^\times$ acts naturally on $H^1_{\dR}(\Sigma_{1,F})\otimes_{\Q_p} E$ by acting on the first factor. Since the action of $O_D^\times$ on $\Sigma_{1,F}$ factors through $O_D^\times/(1+\Pi O_D)$ , we can decompose $H^1_{\dR}(\Sigma_{1,F})\otimes_{\Q_p} E$ as:
\begin{eqnarray} 
H^1_{\dR}(\Sigma_{1,F})\otimes_{\Q_p} E =\bigoplus_{\chi\in\chi(E)} (H^1_{\dR}(\Sigma_{1,F})\otimes_{\Q_p}E)^{\chi},
\end{eqnarray}
where $(H^1_{\dR}(\Sigma_{1,F})\otimes_{\Q_p}E)^{\chi}=\{a, d(a)=(1\otimes \chi(d))a,\forall d\in O_D^\times \}$ is the $\chi$-isotypic component.

Now tensor everything in the long exact sequence of lemma \ref{les} with $E$, and take the $\chi$-isotypic component for a non-trivial character $\chi\in \chi(E)$. As we explained above, $O_D^\times$ acts trivially on the cohomology of any annulus, so only the de Rham cohomology of $U_{s}$ contributes. In other words,

\begin{lem} \label{strdr}
For a non-trivial character $\chi$, 
\begin{eqnarray*}
& (H^1_{\dR}(\Sigma_{1,F})\otimes_{\Q_p}E)^{\chi}\simeq\prod_{s} (H^1_{\dR}(U_s)\otimes_{\Q_p} E)^{\chi}\\
& (H^1_{\dR}(\Sigma_{1,F}^{(0)})\otimes_{\Q_p}E)^{\chi}\simeq\prod_{s} (H^1_{\dR}(U_s^{(0)})\otimes_{\Q_p} E)^{\chi}= \prod_{s} ((H^1_{\dR}(U_s)\oplus H^1_{\dR}(U_s'))\otimes_{\Q_p} E)^{\chi}
\end{eqnarray*}
where $s$ takes value in the set of vertices of the Bruhat-Tits tree.
\end{lem}

It's clear that $\GL_2(\Q_p)$ preserves $(H^1_{\dR}(\Sigma_{1,F}^{(0)})\otimes_{\Q_p} E)^{\chi}$ because the action of $\GL_2(\Q_p)$ commutes with $O_D^\times$. Also $g\in \GL_2(\Q_p)$ induces an isomorphism from $U_s^{(0)}$ to $U_{sg}^{(0)}$, hence an isomorphism from $H^1_{\dR}(U_{sg}^{(0)})$ to $H^1_{\dR}(U_s^{(0)})$. Note that the set of vertices of the Bruhat-Tits tree is nothing but $\GL_2(\Z_p)\Q_p^\times \setminus \GL_2(\Q_p)$. Thus we have the following

\begin{prop} \label{mpind}
As a representation of $\GL_2(\Q_p)$ over $E$, we have:
\[
(H^1_{\dR}(\Sigma_{1,F}^{(0)})\otimes_{\Q_p}E)^{\chi}\simeq \indkg (H^1_{\dR}(U_{s'_0}^{(0)})\otimes_{\Q_p} E)^{\chi}
\]
for any non trivial character $\chi \in \chi(E)$. Recall that $s'_0$ is the central vertex. Here the induction has no restriction on the support.
\end{prop}

\section{A \texorpdfstring{$F_0$}{}-structure of \texorpdfstring{$(H^1_{\dR}(\Sigma_{1,F}^{(0)})\otimes_{\Q_p}E)^{\chi}$}{} and the computation of \texorpdfstring{$H^1_{\dR}(U_{s'_0})$}{}} \label{F_0s}

Recall that $F_0$ is the maximal unramified extension of $\Q_p$ inside $F$ and we fixed an isomorphism between it and $\Q_{p^2}$ in the beginning.

Following Coleman and Iovita in \cite{CI}, we can define a $F_0$/Frobenius structure on the de Rham cohomology $H^1_{\dR}(\Sigma_{1,F}^{(0)})$. This means we can find a $F_0$-linear subspace $H_{F_0}$ equipped with a $\Fr$-linear Frobenius morphism, such that $H_{F_0}\otimes_{F_0}F\simeq H^1_{\dR}(\Sigma_{1,F}^{(0)})$. Let's recall their construction in our situation now. 

By lemma \ref{strdr}, we only need to define a $F_0$/Frobenius structure on each $(H^1_{\dR}(U_s^{(0)})\otimes_{\Q_p}E)^\chi$, in fact, each $(H^1_{\dR}(U_{s,\xi})\otimes_{\Q_p}E)^\chi$ (see the notations here in definition \ref{tnbhd}). Theorem C of \cite{Gro3} tells us we have a natural isomorphism between $H^1_{\dR}(U_{s,\xi})$ and $H^1_{\rig}(\overbar{U_{s,\xi}^0}/F)$, the rigid cohomology of $\overbar{U_{s,\xi}^0}$ with coefficients in $F$ defined in \cite{Ber}. Recall that $\overbar{U_{s,\xi}^0}$ is an open set of $\overbar{U_{s,\xi}}$ by removing ($p+1$) $\F_p$-rational points (each corresponds to an edge connecting $s$). Then we have the following exact sequence:
\begin{eqnarray} \label{es1}
0\to H^1_{\rig}(\overbar{U_{s,\xi}}/F)\to H^1_{\rig}(\overbar{U_{s,\xi}^0}/F)\to F^{\oplus p+1}\to F\to 0
\end{eqnarray}

Explicitly, we can construct an isomorphism $\psi_{s,\xi}: U_{s,\xi}\to F_{1,\xi}$, where $F_{1,\xi}$ is defined as
\[\{(x,y)\in \A^2_F, y^{p+1}=v_1^{-1}w_1\xi(x^p-x), |x-k|> p^{-1/(p-1)},k=0,1,\cdots,p-1,|x|<p^{1/(p-1)}\}\]
for an odd vertex $s$ (even case is similar). If we restrict this isomorphism to the generic fibre of $\widehat{\Sigma_{1,O_F,s,\xi}}$ and use the description in \eqref{oeq}, it is given by:
\[
x\mapsto \zeta,~y\mapsto \e'(1-(p/\zeta)^{p-1})^{1/(p+1)}
\]
where $(1-(p/\zeta)^{p-1})^{1/(p+1)}=1-\frac{1}{p+1}(\frac{p}{\zeta})^{p-1}+\cdots$ is the binomial expansion. The rigid space $F_{1,\xi}$ is clearly an open set of a projective curve $D_{1,\xi}$ in $\PP^2_F$ defined by $y^{p+1}=v_1^{-1}w_1\xi(x^p-x)$. We note that $D_{1,\xi}-F_{1,\xi}$ is a union of $p+1$ closed discs. Each disc is centered at a point with zero $y$-coordinate. We denote them by $C_0,\cdots,C_{p}$. Then, we have:

\begin{eqnarray}\label{es2}
0\to H^1_{\dR}(D_{1,\xi})\to H^1_{\dR}(F_{1,\xi})\stackrel{Res}{\to} \bigoplus_{i=0}^p F\stackrel{sum}{\to} F\to 0
\end{eqnarray}
where $Res$ is the residue map to each $C_i$, and $sum$ is taking the sum. A proof of this can be found in section \RNum{4} of \cite{Col}. Notice that $D_{1,\xi}$ has an obvious formal model over $O_F$ (in fact over $O_{F_0}$!), and its special fibre is nothing but $\overbar{U_{1,\xi}}$. So we have a natural isomorphism between $H^1_{\dR}(D_{1,\xi})$ and $H^1_{rig}(\overbar{U_{1,\xi}})$. Using these isomorphisms, we can identify two exact sequences \eqref{es1},\eqref{es2} with each other.

It is not hard to see $O_D^\times$ acts trivially on the residues. For example, near $x=y=0$, $t=y/(1-x^{p-1})^{1/(p+1)}$ is a local coordinate. $O_D^\times$ acts as a character on $y$ and acts trivially on $x$, hence acts trivially on $\frac{dt}{t}$. Therefore if we tensor the exact sequence \eqref{es1} with $E$ and take the $\chi$-isotypic component, we obtain:
\begin{lem} \label{comp}
\[(H^1_{\dR}(U_s)\otimes_{\Q_p} E)^\chi\simeq (H^1_{\rig}(\overbar{U_s}/F)\otimes_{\Q_p} E)^\chi\]
\end{lem}

Since we have a natural isomorphism $H^1_{\rig}(\overbar{U_s}/F)\simeq H^1_{\crys}(\overbar{U_s}/F_0)\otimes_{F_0}F$, there exists a $F_0$/Frobenius structure on $ (H^1_{\rig}(\overbar{U_s}/F)\otimes_{\Q_p} E)^\chi$ and thus on $(H^1_{\dR}(U_s)\otimes_{\Q_p} E)^\chi$. Here $H^1_{\crys}(\overbar{U_s}/F_0)$ is the first crystalline cohomology of $U_s$ tensored with $\Q_p$. Explicitly, as we mentioned above, $D_{1,\xi}$ can be defined over $F_0$ and its formal model $\widehat{D_{1,O_{F_0},\xi}}$ over $O_{F_0}$ is a smooth lifting of $\overbar{U_{s,\xi}}$. So the de Rham cohomology of $\widehat{D_{1,O_{F_0},\xi}}$ can be identified with the crystalline cohomology of $\overbar{U_{s,\xi}}$. Thus we obtain a $F_0$-linear subspace inside $H^1_{\dR}(D_{1,\xi})$. But to get a Frobenius operator, we need to identify it with the crystalline cohomology.

\begin{rem} \label{evenn}
For an even vertex $s'$, we can define similar objects: 
\[\psi_{s',\xi}:U_{s',\xi}\to F_{0,\xi},D_{0,\xi},\widehat{D_{0,O_{F_0},\xi}}\cdots.\]
\end{rem}

In summary, combining the above results with proposition \ref{mpind}, we have:
\begin{prop} \label{F_0lat}
$(H^1_{\dR}(\Sigma_{1,F}^{(0)})\otimes_{\Q_p}E)^{\chi}$ has a $F_0$/Frobenius structure that comes from the crystalline cohomology of the special fibre of $\whsow$. More precisely, under the identification 
$(H^1_{\dR}(\Sigma_{1,F}^{(0)})\otimes_{\Q_p} E)^{\chi}$ with $\indkg (H^1_{\dR}(U_{s'_0}^{(0)})\otimes_{\Q_p} E)^{\chi}$, then the $F_0$-subspace is 
\[\indkg (H^1_{\crys}(\overbar{U_{s'_0}^{(0)}}/F_0)\otimes_{\Q_p} E)^{\chi},\]
and the Frobenius operator is defined in the obvious way.
\end{prop}

\begin{rem}
We can also define a monodromy operator, but it is zero on $(H^1_{\dR}(\Sigma_{1,F}^{(0)})\otimes_{\Q_p}E)^{\chi}$ for non-trivial $\chi$. The reason is that the definition of monodromy operator uses the cohomologies of the tubes of the singular points, which do not contribute to the cohomology we are interested in. See \cite{CI} for the precise definition of monodromy operator.
\end{rem}

As we remarked before, $\overbar{U_{s'_0}}$ has a close relation with the Deligne-Lusztig variety of $\GL_2(\F_p)$ (corollary \ref{iDL}), which we call $DL$. In fact, the open set 
\[\overbar{U_{s'_0}^0}\simeq\Spec \F_{p^2}[\eta,\e,\frac{1}{\e}]/(\e^{p^2-1}+w_1^2(\eta^p-\eta)^{p-1})\] 
is $\GL_2(\Z_p)$-equivariantly isomorphic with $DL$ over the algebraically closed field (or up to taking a transpose of $\GL_2(\F_p)$). So we can apply Deligne-Lusztig theory established in \cite{DL}. Although \cite{DL} is using $l$-adic cohomology, their results can be applied directly to crystalline cohomology thanks to Katz-Messing \cite{KM} and Gillet-Messing \cite{GM}. Notice that the action of $O_D^\times$ on $\overbar{U_{s'_0}^0}$, which factors through $O_D^\times/(1+\Pi O_D)$ can be identified with the inverse of the action of a non-split torus ($\mathrm{T}(w)^F$ in \cite{DL}) of $\GL_2(\F_p)$.

\begin{thm}
Let $\chi(F_0)$ be the character group of $O_D^\times/(1+\Pi O_D)$ with values in $F_0$ (it's generated by $\chi_1$, see definition \ref{fch}). We can decompose 
\[H^1_{\crys}(\overbar{U_{s'_0}}/F_0)=\bigoplus_{\chi'\in\chi(F_0)} H^1_{\crys}(\overbar{U_{s'_0}}/F_0)^{\chi'}\] 
into the sum of different $\chi'$-isotypic components. Each component has a natural action of $\GL_2(\F_p)$. Then, 
\begin{enumerate}
\item $H^1_{\crys}(\overbar{U_{s'_0}}/F_0)^{\chi'}=0$ if and only if $\chi'=(\chi')^p$. 
\item If $H^1_{\crys}(\overbar{U_{s'_0}}/F_0)^{\chi'}\neq 0$, it's an irreducible representation of $\GL_2(\F_p)$.
\item $H^1_{\crys}(\overbar{U_{s'_0}}/F_0)^{\chi'}\simeq H^1_{\crys}(\overbar{U_{s'_0}}/F_0)^{(\chi')^p}$ and these are the only isomorphisms among these non-zero representations.
\end{enumerate}
\end{thm}

\begin{definition} \label{deforho}
For any $\chi\in\chi(E)$, define $\rho_{\chi}$ as the representation $(H^1_{crys}(\overbar{U_{s'_0}}/F_0)\otimes_{F_0}E)^\chi$ of $\GL_2(\F_p)$. The theorem above guarantees that different choices of embedding $F_0\to E$ give the same representation.
\end{definition}

\begin{rem} \label{relact}
$\Gal(F/F_0)$ also acts on $H^1_{\crys}(\overbar{U_{s'_0}}/F_0)^{\chi'}$. Using the results in section \ref{act}, we have $H^1_{crys}(\overbar{U_{s'_0}}/F_0)^{\chi'}=H^1_{\crys}(\overbar{U_{s'_0}}/F_0)^{\tomega_2^{i(\chi')}}$, the $\tomega_2^{i(\chi')}$-isotypic space for $\Gal(F/F_0)$, where $i(\chi')\in\{0,\cdots,p^2-2\}$ is defined as the unique integer such that $\chi_1^{-i(\chi')}=\chi'$. Using results in remark \ref{intdef}, another equivalent definition is that $\tomega_2^{i(\chi')}$ is the unique character making the following diagram commutative:
\[\begin{CD}
\Z_{p^2}^\times\simeq O_{F_0}^\times @>Art_{F_0} >> \Gal(\overbar{F_0}/F_0)^{ab}\\
@VVV                               @VV\tomega_2^{i(\chi')}V\\
O_D^\times   @>\chi'>> \Z_{p^2}^\times\simeq O_{F_0}^\times.
\end{CD}\]
Recall that $\tomega_2$ is defined in definition \ref{intdef}.
\end{rem}

Now I want to translate the theorem above to our situation. Fix an embedding $\tau:F_0\to E$, and use $\bar{\tau}$ to denote the conjugate embedding. Let $\chi'\in\chi(F_0)$ be the unique character that satisfies $\tau\circ\chi'=\chi$. Recall that $g_\varphi\in \Gal(F/\Q_p)$ is the unique element that fixes $\varpi$ but acts as Frobenius on $F_0$.

\begin{prop} \label{dcrys}
$D_{\crys,\chi}\defeq \Hom_{\GL_2(\F_p)}(\rho_{\chi'},(H^1_{crys}(\overbar{U_{s'_0}^{(0)}}/F_0)\otimes_{\Q_p}E)^{\chi})$ is a free $F_0\otimes_{\Q_p}E$-module of rank $2$. $\Gal(F/\Q_p)$ and Frobenius operator $\varphi$ act on it naturally. In fact, $D_{\crys,\chi}$ is of the form:
\begin{eqnarray*}
& D_{\crys,\chi}=(F_0\otimes_{\Q_p} E)\cdot \ebf_1\oplus (F_0\otimes_{\Q_p}E)\cdot \ebf_2,\\
&\varphi(\ebf_1)=\ebf_2,\varphi(\ebf_2)=(1\otimes c_x)\ebf_1,\\
&g\cdot\ebf_1=(\tomega_2(g)^m\otimes1)\ebf_1,~g\cdot\ebf_2=(\tomega_2(g)^{pm}\otimes1)\ebf_2,
\forall g\in\Gal(F/F_0),\\
&g_{\varphi}\cdot\ebf_1=\ebf_1,~g_{\varphi}\cdot\ebf_2=\ebf_2,
\end{eqnarray*}
with $c_x\in E$ and $v_p(c_x)=1$, $m=i(\chi')$ defined in the above remark.
\end{prop}
\begin{proof}
We can write (using the fact $\overbar{U_{s'_0}^{(0)}}=\overbar{U_{s'_0}}\sqcup\overbar{U_{s'_0}'}$)
\begin{eqnarray*}
(H^1_{\crys}(\overbar{U_{s'_0}^{(0)}}/F_0)\otimes_{\Q_p}E)^{\chi}&=&(H^1_{\crys}(\overbar{U_{s'_0}}/F_0)\otimes_{\Q_p}E)^{\chi}\oplus (H^1_{\crys}(\overbar{U_{s'_0}'}/F_0)\otimes_{\Q_p}E)^{\chi}\\
& =& H^1_{\crys}(\overbar{U_{s'_0}}/F_0)^{\chi'}\otimes_{F_0,\tau}E\oplus H^1_{\crys}(\overbar{U_{s'_0}}/F_0)^{\bar{\chi'}}\otimes_{F_0,\bar{\tau}}E\\
& & \oplus H^1_{\crys}(\overbar{U_{s'_0}'}/F_0)^{\chi'}\otimes_{F_0,\tau}E\oplus H^1_{\crys}(\overbar{U_{s'_0}'}/F_0)^{\bar{\chi'}}\otimes_{F_0,\bar{\tau}}E
\end{eqnarray*}
where $\bar{\chi'}=(\chi')^p$, the conjugate character, satisfies $\bar{\tau}\circ\bar{\chi'}=\chi$.

Recall that we can identify $\overbar{U_{s'_0}}$ with $\overbar{U_{s'_0}'}$ but with different structure map to $\Spec \F_{p^2}$. Using remark \ref{remact}, such identification induces isomorphism between $H^1_{\crys}(\overbar{U_{s'_0}}/F_0)^{\chi'}$ and $H^1_{\crys}(\overbar{U_{s'_0}'}/F_0)^{\bar{\chi'}}$. By definition, $\Hom_{\GL_2(\F_p)}(\rho_{\chi'},H^1_{\crys}(\overbar{U_{s'_0}}/F_0)^{\chi'}\otimes_{F_0,\tau}E)\simeq F_0\otimes_{F_0,\tau}E$ and similar results for other factors of $(H^1_{\crys}(\overbar{U_{s'_0}^{(0)}}/F_0)\otimes_{\Q_p}E)^{\chi}$ follow from Deligne-Lusztig theory. It's easy to see $D_{\crys,\chi}\simeq F_0\otimes_{\Q_p}E^{\oplus 2}$ from these descriptions.

By the remark \ref{relact} and remark \ref{remact}, $\Gal(F/F_0)$ acts via $\tomega_2^m$ (as an $F_0$-vector space) on $H^1_{\crys}(\overbar{U_{s'_0}}/F_0)^{\chi'}\otimes_{F_0,\tau}E, H^1_{crys}(\overbar{U_{s'_0}'}/F_0)^{\bar{\chi'}}\otimes_{F_0,\bar{\tau}}E$  and acts as $\tomega_2^{pm}$ on the other two factors since $i(\bar{\chi'})=i((\chi')^p)=pi(\chi')$. Remark \ref{remact} also tells us that $g_\varphi$ induces an isomorphism between $H^1_{\crys}(\overbar{U_{s'_0}}/F_0)^{\chi'}\otimes_{F_0,\tau}E$ and $H^1_{crys}(\overbar{U_{s'_0}'}/F_0)^{\bar{\chi'}}\otimes_{F_0,\bar{\tau}}E$. Now, choose a generator $\mathbf{f}_1$ of $\Hom_{\GL_2(\F_p)}(\rho_{\chi'},H^1_{crys}(\overbar{U_{s'_0}}/F_0)^{\chi'}\otimes_{F_0,\tau}E)$ and define 
\[
\ebf_1=\mathbf{f}_1+g_\varphi\cdot\mathbf{f}_1, \ebf_2=\varphi(\ebf_1)
\]
where $\varphi$ is the Frobenius operator coming from the crystalline cohomology. We need to verify our claim in the propositon.

First it's easy to see $\ebf_1$ is indeed a generator of $\Hom_{\GL_2(\F_p)}(\rho_{\chi'},H^1_{\crys}(\overbar{U_{s'_0}}/F_0)^{\chi'}\otimes_{F_0,\tau}E\oplus H^1_{crys}(\overbar{U_{s'_0}'}/F_0)^{\bar{\chi'}}\otimes_{F_0,\bar{\tau}}E)$ as a free $F_0\otimes_{\Q_p}E$-module and satisfies $g\cdot\ebf_1=(\tomega_2(g)^m\otimes1)\ebf_1,g\in\Gal(F/F_0)$. Next we verify the desired property of the Frobenius operator $\varphi$. It's induced by the Frobenius endomorphism on $\overbar{U_{s'_0}}$, which is nothing but raising anything to its $p$-th power. So it sends $H^1_{\crys}(\overbar{U_{s'_0}}/F_0)^{\chi'}$ to $H^1_{crys}(\overbar{U_{s'_0}}/F_0)^{(\chi')^p}=H^1_{\crys}(\overbar{U_{s'_0}}/F_0)^{\bar{\chi'}}$. Therefore everything is clear except our claim for $\varphi(\ebf_2)$. This can be done by explicit computations. See the lemma below.
\end{proof}

\begin{lem} \label{c_x}
$c_x=-p\tau(w_1^{-2i})$.
\end{lem}
\begin{proof}
This can be done using Gauss sums. Since $H^1_{\crys}(\overbar{U_{s'_0}}/F_0)^{\chi'}$ is an irreducible representation of $\GL_2(\F_p)$, $\varphi^2$ acts as a scalar $\tilde{c_x}$ on it. It's easy to see $c_x=\tau(\tilde{c_x})$. 

To compute $\tilde{c_x}$, we only need to restrict on one component. So let $\xi$ be a root of $\xi^{p-1}=-1$. Then $\overbar{U_{s'_0,\xi}}$ can be identified as the curve in $\PP^2_{\F_{p^2}}$ defined by $y^{p+1}=v_1w_1^{-1}\xi(x^p-x)$. There is an action of $\mu_{p+1}(\F_{p^2})=\{a\in \F_{p^2}^\times,a^{p+1}=1\}$ on it given by
\[a\cdot x=x,~a\cdot y=ay,~a\in\mu_{p+1}(\F_{p^2}^\times)\]
Let $\tilde{\chi}:\mu_{p+1}(\F_{p^2})\to F_0^{\times}$ be the Teichm\"uller character. It's obvious that $H^1_{\crys}(\overbar{U_{s'_0}}/F_0)^{\chi'}\simeq H^1_{\crys}(\overbar{U_{s'_0,\xi}}/F_0)^{\tilde{\chi}^{-i}}$, the $\tilde{\chi}^{-i}$-isotypical component. Here $i\in\{1,\cdots,p\}$ is the unique number satifying $i\equiv m~\modd~p+1$.

On the other hand, $\F_{p}$ also acts on $\overbar{U_{s'_0,\xi}}$, which comes from the action of an unipotent subgroup of $\GL_2(\F_p)$:
\[b\cdot x=x+1,~b\cdot y=y,~b\in \F_{p}.\]
This action commutes with the action of $\mu_{p+1}(\F_{p^2})$. It's easy to see $F_0$ contains all $p$-th roots of unity. Let $\psi_p:\F_p\to F_0^\times$ be a non-trivial additive character. We view $\tilde{\rho}=\tilde{\chi}^{-i}\times \psi_p$ as a one dimensional representation of $\tilde{G}\defeq\mu_{p+1}(\F_{p^2})\times\F_p$.

Using lemma 1.1. of \cite{Katz}, we know that the eigenvalue of $\varphi^2$ on $(H^1_{\crys}(\overbar{U_{s'_0,\xi}}/F_0)\otimes_{F_0}F)^{\tilde{\rho}}$ is (we will see later that this lemma indeed can be applied to our situation):
\[-S(\overbar{U_{s'_0,\xi}}/\F_{p^2},\tilde{\rho},1)\defeq -\frac{1}{\#\tilde{G}}\sum_{g\in\tilde{G}}tr(\tilde{\rho}(g))\#\Fix(\mathrm{F_{p^2}}g^{-1}),\]
where $\mathrm{F_{p^2}}$ is the Frobenius endomorphism of $\overbar{U_{s'_0,\xi}}$ relative to $\F_{p^2}$ and $\Fix(\mathrm{F_{p^2}}g^{-1})$ is the subset of $\overbar{U_{s'_0,\xi}}(\overbar{\F_{p^2}})$ fixed by $\mathrm{F_{p^2}}g^{-1}$. Following the strategy of lemma 2.1.  of \cite{Katz}, we can express $S(\overbar{U_{s'_0,\xi}}/\F_{p^2},\tilde{\rho},1)$ as the Gauss sum:
\[S(\overbar{U_{s'_0,\xi}}/\F_{p^2},\tilde{\rho},1)=(v_1w_1^{-1}\xi)^{-i(p-1)}\sum_{x\in\F_{p^2}^{\times}} \psi_{p^2}(x)x^{-i(p-1)},\]
where $\psi_{p^2}\defeq\psi_p(\tr_{\F_{p^2}/\F_p}(x))=\psi_p(x^p+x)$. Notice that for any $x\in\F_{p^2}^{\times}$,
\[\sum_{a\in\F_p^{\times}}\psi_{p^2}(ax)=\sum_{a\in\F_p^{\times}}\psi_p(a(x^p+x))=\left\{
\begin{array}{cc}
	-1~\mbox{if } x^p+x\neq 0\\
	p-1~\mbox{if }x^p+x=0
\end{array} \right.\]
From this, it's easy to see $S(\overbar{U_{s'_0,\xi}}/\F_{p^2},\tilde{\rho},1)=w_1^{i(p-1)}p(-1)^i=w_1^{-2i}p$ (recall $v_1^{p-1}=w_1^{p+1}=\xi^{p-1}=-1$). Hence
\[c_x=-p\tau(w_1^{-2i}).\]
\end{proof}

\begin{cor} \label{mnc}
We have a $\Gal(F/\Q_p)\times O_D^\times \times \GL_2(\Q_p)$-equivariant isomorphism:
\begin{eqnarray}
F\otimes_{F_0}D_{\crys,\chi}\otimes_E \indkg\rho_{\chi}\stackrel{\sim}{\to} (H^1_{\dR}(\Sigma_{1,F}^{(0)})\otimes_{\Q_p}E)^{\chi}
\end{eqnarray}
where $\Gal(F/\Q_p)$ acts on first two components, $O_D^\times$ acts on the second, $\GL_2(\Q_p)$ acts on the third. Moreover, $D_{\crys,\chi}\otimes_E \indkg\rho_{\chi}$ maps to the $F_0$ subspace we constructed in proposition \ref{F_0lat}.
\end{cor}
Here we extend $\rho_{\chi}$ to a representation of $\GL_2(\Z_p)\Q_p^\times$ by $p$ acting trivially and $\GL_2(\Z_p)$ acting through $\GL_2(\F_p)$.

\begin{rem} \label{smdual}
It's easy to see the dual representation of $\rho_{\chi}$ is $\rho_{\chi^{-1}}$, we use $<\cdot,\cdot>$ to denote the pairing of them. Then we can construct a pairing:
\begin{eqnarray*}
c-\indkg\rho_{\chi^{-1}}\times \indkg\rho_{\chi}&\longrightarrow &E\\
(f_1,f_2)&\mapsto&\Sigma_{[g]\in \GL_2(\Z_p)\Q_p^\times\setminus \GL_2(\Q_p)}<f_1(g),f_2(g)>,
\end{eqnarray*}
where is the compact induction. More precisely, $c-\indkg\rho_{\chi^{-1}}=\{f:\GL_2(\Q_p)\to \rho_{\chi^{-1}} \mbox{has compact support }\modd  \GL_2(\Z_p)\Q_p^\times,~f(kg)=\rho_{\chi^{-1}}(k)f(g),~k\in\GL_2(\Z_p)\Q_p^\times,g\in\GL_2(\Q_p) \}$ and $\indkg\rho_{\chi}$ is defined similarly without any restrictions on the support. The sum makes sense because it only has finitely many non-zero terms.

This pairing induces an isomorphism $\indkg\rho_{\chi}\simeq (c-\indkg\rho_{\chi^{-1}})^\vee$, the algebraic dual representation. We can rewrite the result in the above corollary  as a $\Gal(F/\Q_p)\times O_D^\times \times \GL_2(\Q_p)$-equivariant isomorphism (theorem \ref{mthi}):
\begin{eqnarray}
F\otimes_{F_0}D_{\crys,\chi}\otimes_E (c-\indkg\rho_{\chi^{-1}})^\vee\stackrel{\sim}{\to} (H^1_{\dR}(\Sigma_{1,F}^{(0)})\otimes_{\Q_p}E)^{\chi}
\end{eqnarray}

By corollary \ref{Fre}, there is a natural Fr\'{e}chet space structure on the right hand side of the above map. In fact, we can describe this topology directly on the left hand side. Choosing a family of representatives of $\GL_2(\Z_p)\Q_p^\times\setminus \GL_2(\Q_p)$, we have a non-canonical isomorphism between $\indkg\rho_{\chi}$ and $\prod_{\GL_2(\Z_p)\Q_p^\times\setminus \GL_2(\Q_p)}\rho_{\chi}$ as $E$-vector spaces. The topology is nothing but the weakest topology on this product such that each projection to $\rho_{\chi}$ is continuous under the canonical (Banach space) topology on $\rho_{\chi}$.
\end{rem}

\section{Some considerations from Galois representations} \label{Gal}

Let's recall what we have on $D_{\crys,\chi}$ (see proposition \ref{dcrys} for more details):
\begin{itemize}
\item Frobenius operator $\varphi$: an $F_0$-semilinear, $E$-linear automorphism.
\item monodromy operator $N$, which is zero here.
\item an action of $\Gal(F/\Q_p)$, which is $F_0$-semilinear, $E$-linear commuting with $\varphi$ and $N$.
\end{itemize}
So if we have a decreasing filtration on $D_F=F\otimes_{F_0}D_{\crys,\chi}$, such that $\Fil^iD_F$ is zero if $i\gg 0$ and is equal to $D_F$ if $i\ll 0$ and  preserved by the action of $\Gal(F/\Q_p)$, $D_{\crys,\chi}$ is called a filtered $(\varphi,N,F/\Q_p,E)$-module of rank $2$. Moreover, if the underlying $(\varphi,N,F,E)$-module is weakly admissible, $D_{\crys,\chi}$ is called weakly admissible. See Definition 2.7. and 2.8. of \cite{Sav} for the precise definition. The importance of such kind of module is that we have the following result (see Corollary 2.10. of \cite{Sav}.)
\begin{thm} \label{wa}
The category of $E$-representations of $G_{\Q_p}$ which become semi-stable when restricted to $G_F$ and the category of weakly admissible $(\varphi,N,F/\Q_p,E)$-modules are equivalent.
Here $G_{\Q_p}$ (resp. $G_F$) is the absolute Galois group of $\Q_p$ (resp. $F$).
\end{thm}

Now I want to classify all two dimensional potentially semi-stable $E$-representations of $G_{\Q_p}$ such that 

\begin{itemize}
\item has Hodge-Tate weights $(0,1)$, and
\item corresponds to $D_{\crys,\chi}$ if we forget about the filtration.
\end{itemize}

\begin{prop} \label{sav}
Any such weakly admissible $(\varphi,N,F/\Q_p,E)$-module is of the form:
\[\Fil^n(D_F)=\left\{
	\begin{array}{lll}
		D_F,~n\le 0\\
		(F\otimes_{\Q_p}E)((\varpi^{(p-1)i}\otimes a)\ebf_1+(1\otimes b)\ebf_2),~n=1\\
		0,~n\ge 2
	\end{array}\right.\]
where $i,j$ are defined as follows: write $m=i+(p+1)j$ with $i\in\{1,\cdots,p\}$ and $j\in\{0,\cdots,p-2\}$, $(a,b)\neq (0,0)\in E^2$.
\end{prop}
\begin{proof}
This is Proposition 2.18. of \cite{Sav}.
\end{proof}
We denote the filtered module in the above proposition by $D_{\chi,[a,b]}$. It's not hard to see $D_{\chi,[a,b]}\simeq D_{\chi^p,[bc_x/p,-a]}$ and $D_{\chi,[a,b]}=D_{\chi,[ca,cb]}$. So we may assume $a=1$ and $v_p(b)\ge 0$ (recall that $c_x$ is defined in proposition \ref{dcrys}). We use $V_{\chi,[1,b]}$ to denote the Galois representation it corresponds to in theorem \ref{wa}.

Now suppose we have an element $f$ in $D_F=F\otimes_{F_0}D_{\crys,\chi}$, how to check whether $f$ is in $\Fil^1(F\otimes_{F_0}D_{\chi,[1,b]})$ for a given $b$ or not?  First assume $f\in \Fil^1(F\otimes_{F_0}D_{\chi,[1,b]})$. Write $f=f_1+f_2,~f_1\in (F\otimes_{\Q_p} E) \cdot\ebf_1,~f_2\in (F\otimes_{\Q_p} E) \cdot\ebf_2$. Then we must have:
\begin{eqnarray*}
&&f_1=(\sum a_k\otimes b_k)(\varpi^{(p-1)i}\otimes 1) \ebf_1\\
&&f_2=(\sum a_k\otimes b_k)(1\otimes b)\ebf_2
\end{eqnarray*}
for some $a_k\in F,b_k\in E$. Notice that $g_{\varphi}\otimes \varphi$ is well-defined on $F\otimes_{F_0}D_{\crys,\chi}$ since $g_\varphi$ acts as Frobenius on $F_0$. Here $g_{\varphi}$ is considered only acting on $F$, not on $D_{\crys,\chi}$. 
\[ (g_{\varphi}\otimes \varphi)(f_1)=(\sum g_{\varphi}(a_k)\otimes b_k)(\varpi^{(p-1)i}\otimes 1) \ebf_2
\]
On the other hand, $g_{\varphi}(f_2)=(\Sigma g_{\varphi}(a_k)\otimes b_k)(1\otimes b)\ebf_2$. Therefore, 
\[
(1\otimes b)(g_{\varphi}\otimes \varphi)(f_1)=(\varpi^{(p-1)i}\otimes 1)g_\varphi(f_2).
\]

A simple dimension counting shows that this condition is even sufficient. Hence,

\begin{prop} \label{cri}
$f\in F\otimes_{F_0}D_{\crys,\chi}$. Write $f=f_1+f_2,~f_1\in (F\otimes_{\Q_p} E) \cdot\ebf_1,~f_2\in (F\otimes_{\Q_p} E) \cdot\ebf_2$. Then $f\in \Fil^1(F\otimes_{F_0}D_{\chi,[1,b]})$ if and only if $(1\otimes b)(g_{\varphi}\otimes \varphi)(f_1)=(\varpi^{(p-1)i}\otimes 1)g_\varphi(f_2)$. 
\end{prop}
\begin{rem}
In practice, we will assume $f$ is fixed by $g_{\varphi}$, then the condition above is simplified as $(1\otimes b)(g_{\varphi}\otimes \varphi)(f_1)=(\varpi^{(p-1)i}\otimes 1)f_2$.
\end{rem}

\section{Construction of Banach space representations of \texorpdfstring{$\GL_2(\Q_p)$}{}} \label{cb}

In this section, I want to construct some Banach space representations $B(\chi,[1,b])$ that should correspond to $V_{\chi,[1,b]}^\vee$ (up to a twist by some character) under the $p$-adic local Langlands correspondence. 

First we define an integral structure $\omega^1$ of $\Omega^1_{\Sigma_{1,F}}$, the sheaf of holomorphic differential forms, on $\wtso$ defined in section \ref{ssm}. Recall that $\wtso$ is a formal model of $\Omega^1_{\Sigma_{1,F}}$ which is not semi-stable, but only has some mild singularities ($xy-\varpi^{{p-1}}$). From now on, I will do all computations on this formal model rather than the semi-stable model.

View $\Omega^1_{\Sigma_{1,F}}$ as a sheaf on $\wtso$ .The coherent sheaf $\omega^1$ will be a subsheaf of it. Recall that $\wtso$ has an open covering $\{\widetilde{\Sigma_{1,O_F,e,\xi}}\}_{e,\xi}$, where $e$ takes value in the set of edges of Bruhat-Tits tree and $\xi^{p-1}=-1$. Using the explicit description of lemma \ref{sms}, we define $\omega^1$ on each $\widetilde{\Sigma_{1,O_F,e,\xi}}$ as the trivial line bundle with a basis $\frac{d\e}{\e}=-\frac{d\e'}{\e'}$ (recall that $\e=e/\varpi,\e'=e'/\varpi$). It's easy to see that this really defines a line bundle $\widetilde{\Sigma_{1,O_F,e,\xi}}$ which becomes $\Omega^1_{\Sigma_{1,F}}$ if we restrict this line bundle to the generic fibre.

\begin{rem}
We can do exactly the same thing on the semi-stable model $\whso$,. But this won't give us any extra sections: the sections on $\widetilde{\Sigma_{1,O_F,e}}$ and $\widehat{\Sigma_{1,O_F,e}}$ will be the same. This can be checked locally around the singularities. So I can do all the computations on $\widetilde{\Sigma_{1,O_F,e}}$.
\end{rem}

\begin{rem} \label{o1F0}
We note that $\omega^1$ in fact has a `$F_0$-structure'. In other words, we can define it on $\widehat{\Sigma_1}$. Using the explicit description in corollary \ref{cact}, locally on $\widehat{\Sigma_{1,e}}$, it is defined as the trivial line bundle generated by $\frac{de}{e}$. Notice that $\frac{de}{e}=\frac{d\e}{\e}$ since $e=\e \varpi$. Hence its pull back to $\wtso$ is $\omega^1$.
\end{rem}

Similarly, we can define the same thing on $\widetilde{\Sigma_{1,O_F}'},\widetilde{\Sigma_{1,O_F}^{(0)}}$, which we still use $\omega^1$ by abuse of notations. Now if we restrict $\omega^1$ to the special fibre, it becomes the dualizing sheaf (over $\Spec \F_{p^2}$). So there is an action of $\GL_2(\Q_p)$ on it. In fact, $\GL_2(\Q_p)$ even acts on $\omega^1$. This can be seen using the explicit description in section \ref{act}. Also, it's clear from the definition that $O_D^\times$ and $\Gal(F/\Q_p)$ act on the global sections of $\omega^1$.

Consider the following maps
\[H^0(\widetilde{\Sigma_{1,O_F}^{(0)}},\omega^1)\otimes_{\Z_p}O_E\hookrightarrow H^0(\Sigma_{1,F}^{(0)},\Omega^1_{\Sigma_{1,F}^{(0)}})\otimes_{\Q_p}E\to H^1_{\dR}(\Sigma_{1,F}^{(0)})\otimes_{\Q_p}E.\]
It is clear that both maps are $\GL_2(\Q_p),O_D^\times,\Gal(F/\Q_p)$-equivariant. Take the $\chi$-isotypic component, where $\chi\in\chi(E)$ (see section \ref{dR}). We get a map: (use corollary \ref{mnc})
\[
f_{\chi}: (H^0(\widetilde{\Sigma_{1,O_F}^{(0)}},\omega^1)\otimes_{\Z_p}O_E)^\chi\to (H^1_{\dR}(\Sigma_{1,F}^{(0)})\otimes_{\Q_p}E)^{\chi}\simeq F\otimes_{F_0}D_{\crys,\chi}\otimes_E \indkg\rho_{\chi}.
\]

Now for each two dimensional Galois representation $V_{\chi,[1,b]}$ of $G_{\Q_p}$ defined in the previous section, we have a free $F\otimes_{\Q_p} E$-module $\Fil^1(F\otimes_{F_0}D_{\chi,[1,b]})$ inside $F\otimes_{F_0}D_{\crys,\chi}$. We note that $\Gal(F/\Q_p)$ acts on this $\Fil^1(F\otimes_{F_0}D_{\chi,[1,b]})$. Define
\begin{eqnarray*}
M(\chi,[1,b])&=&(f_{\chi}^{-1}(\Fil^1(F\otimes_{F_0}D_{\chi,[1,b]})\otimes_E \indkg\rho_{\chi}))^{\Gal(F/\Q_p)}\\
&=& f_{\chi}^{-1}((\Fil^1(F\otimes_{F_0}D_{\chi,[1,b]}))^{\Gal(F/\Q_p)}\otimes_E \indkg\rho_{\chi})
\end{eqnarray*}

In \cite{ST}, Schneider and Teitelbaum introduced a category $\Mfc$ whose objects are all torsion-free and compact, Hausdorff linear-topological $O_E$-modules, and morphisms are all continuous $O_E$-linear maps. Our first result about $M(\chi,[1,b])$ is:
\begin{prop} \label{Mcfe}
$M(\chi,[1,b])$ with the topology induced from $H^0(\Sigma_{1,F}^{(0)},\Omega^1_{\Sigma_{1,F}^{(0)}})\otimes_{\Q_p}E$ is an object in $\Mfc$.
\end{prop}
\begin{proof}
I learned this argument from Proposition 4.2.1. of \cite{Bre}. It is clear that $M(\chi,[1,b])$ is torsion free and Hausdorff. To prove compactness, we use Proposition 15.3.(iii) of \cite{NFA} (c-compactness is equivalent with compactness here since $O_E$ is locally compact (Corollary 6.1.14 of \cite{LCS})). Proposition \ref{ref1} already shows that $H^0(\Sigma_{1,F}^{(0)},\Omega^1_{\Sigma_{1,F}^{(0)}})\otimes_{\Q_p}E$ is a reflexive Fr\'echet space, so it suffices to show $M(\chi,[1,b])$ is closed and bounded (see \cite{NFA} for the definition of boundedness). In fact it's easy to see we only need to prove closedness and boundedness for $H^0(\wtso,\omega^1)$ in $H^0(\Sigma_{1,F}^{(0)},\Omega^1_{\Sigma_{1,F}^{(0)}})$. Recall the topology on $H^0(\Sigma_{1,F}^{(0)},\Omega^1_{\Sigma_{1,F}^{(0)}})$ is defined in section \ref{dR} by:
$$H^0(\Sigma_{1,F},\Omega_{\Sigma_{1,F}}^1)=\varprojlim_n H^0(V_{n,F},\Omega_{\Sigma_{1,F}}^i),$$
where $\{V_{n,F}\}_n$ is an admissible open covering of $\Sigma_{1,F}$ and each $V_{n,F}$ is affinoid and contained in $V_{n+1,F}$. Now $\{\widetilde{\Sigma_{1,O_F,e}}\}_e$ is another admissible open covering. Thus we have:
\begin{itemize}
\item Each $\widetilde{\Sigma_{1,O_F,e}}$ is contained in some $V_{n,F}$.
\item Each $V_{n,F}$ is covered by finitely many generic fibres of $\widetilde{\Sigma_{1,O_F,e}}$.
\end{itemize}
Then closedness follows from the first claim above and boundedness follows from the second.
\end{proof}

Suppose $M$ is an object in $\Mfc$, following \cite{ST}, the $E$-vector space $M^d\defeq\Hom^{\mathrm{cont}}_{O_E}(M,E)$ with the norm $\|f\|=\max_{m\in M}|f(m)|_E$ is a Banach space. 
\begin{definition}
\[B(\chi,[1,b])\defeq (M(\chi,[1,b]))^d=\Hom^{\mathrm{cont}}_{O_E}(M(\chi,[1,b]),E).\]
\end{definition}
It's clear from the definition that this is a Banach space representation of $\GL_2(\Q_p)$.

\begin{rem}
The relation between $B(\chi,[1,b])$ and the Banach representation $B(\pi,\cL)$ defined in the introduction (see definition \ref{dmi} and \ref{dbi}) is  as follows: Take $\pi=\Ind_{O_D^\times\Q_p^\times}^{D^\times}\chi$, where $\chi$ is viewed as a character of $O_D^\times\Q_p^\times$ trivial on $p$. Also $\Fil^1(D_{\chi,[1,b]}\otimes F)$ essentially gives a line '$\cL_b$' in the definition \ref{dmi} by taking $\Gal(F/\Q_p)$-invariants. Then $B(\chi,[1,b])= B(\pi,\cL_b)$.
\end{rem}

Back to the definition of $M(\chi,[1,b])$. By remark \ref{smdual}, we can replace the induced representation by the dual representation of the compact induction. Also by Galois descent, we have $(\Fil^1(F\otimes_{F_0}D_{\chi,[1,b]}))^{\Gal(F/\Q_p)}\simeq E$. Under these isomorphisms, $f_{\chi}$ induces a $\GL_2(\Q_p)$-equivariant map:
\[f_{\chi,[1,b]}:M(\chi,[1,b])\to (c-\indkg\rho_{\chi^{-1}})^\vee\]
It is natural to ask whether such a map is injective or not. The answer is positive.

\begin{prop}\label{inj}
The composition of 
\[H^0(\widetilde{\Sigma_{1,O_F}^{(0)}},\omega^1)^{\chi'}\hookrightarrow H^0(\Sigma_{1,F},\Omega^1_{\Sigma_{1,F}})^{\chi'}\to H^1_{\dR}(\Sigma_{1,F})^{\chi'}\] 
for a non-trivial character $\chi'\in\chi(F)$ is injective.
\end{prop} 
\begin{proof}
Since we assume $\chi'$ is non-trivial, the kernel of the second map is $H^0(\Sigma_{1,F},\cO_{\Sigma_{1,F}})^{\chi'}$. Consider the intersection of $H^0(\Sigma_{1,F},\Omega^0_{\Sigma_{1,F}})^{\chi'}$ and $H^0(\widetilde{\Sigma_{1,O_F}^{(0)}},\omega^1)^{\chi'}$ in $H^0(\Sigma_{1,F},\Omega^1_{\Sigma_{1,F}})^{\chi'}$. It can be viewed as a subset in $H^0(\Sigma_{1,F},\Omega^0_{\Sigma_{1,F}})^{\chi'}$ and we denote it by $H$. On the other hand, we use $J$ to denote the same set but viewed in $H^0(\Sigma_{1,F},\Omega^1_{\Sigma_{1,F}})^{\chi'}$. The induced topology on $H$ and $J$ can be different. Proposition \ref{Mcfe} tells us that $J$ is compact since $H^0(\Sigma_{1,F},\Omega^0_{\Sigma_{1,F}})^{\chi'}$ is closed in $H^0(\Sigma_{1,F},\Omega^1_{\Sigma_{1,F}})^{\chi'}$ (proposition \ref{clim}). Clearly $\SL_2(\Q_p)$ preserves both $H$ and $J$.

Let's recall some notations here. For each connected component of $\wtso$, the dual graph of its special fibre is the Bruhat-Tits tree (see section \ref{ssm}), and $U_{s,\xi}$ is the tubular neighborhood of $\overbar{U_{s,\xi}}$, the irreducible component indexed by $(s,\xi)$ in the special fibre (see definition \ref{u_s}).

Similar to what we did in the beginning of section of \ref{F_0s}, we can prove $U_{s'_0,\xi}$ is isomorphic with 
\[\{z=(x,y)\in\A_F^2,y^{p+1}=v_1w_1^{-1}\xi(x^p-x),|x-k|>p^{-1},k=0,\cdots,p-1,|x|<p \}\]
and its de Rham cohomology is of finite dimension. Since $U_{s'_0,\xi}$ is a Stein space, $H^1_{\dR}(U_{s'_0,\xi})=H^0(U_{s'_0,\xi},\Omega^1)/H^0(U_{s'_0,\xi},\Omega^0)$ (we use $\Omega^i$ for $\Omega^i_{\Sigma_{1,F}}$ for simplicity).

Fix a $\xi$. Under the isomorphism above, we can write $U_{s'_0,\xi}=\bigcup_{\rho<p}U_{s'_0,\xi,\rho}$, where $U_{s'_0,\xi,\rho}\subset U_{s'_0,\xi}$ is defined by the same equation but with $|x-k|\geq\rho^{-1},k=0,\cdots,p-1,|x|\leq\rho$. Then for each $\rho<p$, $H^0(U_{s'_0,\xi,\rho},\Omega^i)$ is a Banach space, and we have $H^0(U_{s'_0,\xi},\Omega^i)=\varprojlim_{\rho\to p} H^0(U_{s'_0,\xi,\rho},\Omega^i)$. So $H^0(U_{s'_0,\xi},\Omega^{i})$ is a Fr\'{e}chet space.

Notice that $O_D^\times$ acts on $U_{s'_0}$, so $H^0(U_{s'_0},\Omega^0)^{\chi'}\hookrightarrow H^0(U_{s'_0},\Omega^1)^{\chi'}$ and the quotient is a finite dimensional space. Thus this inclusion has to be a closed embedding because both of them are Fr\'{e}chet spaces. 

Now consider the canonical maps $H^0(\Sigma_{1,F},\Omega^k)^{\chi'}\to H^0(U_{s'_0},\Omega^k)^{\chi'},k=0,1$. They're clearly continuous and we denote the image of $H$ and $J$ by $H_1$ and $J_1$. Since $J$ is compact, $J_1$ is compact. Hence $H_1$ is also compact in $H^0(\Sigma_{1,F},\Omega^0)^{\chi'}$ because $H^0(U_{s'_0},\Omega^0)^{\chi'}\hookrightarrow H^0(U_{s'_0},\Omega^1)^{\chi'}$ is a closed embedding. We will show this cannot happen unless $H_1=\{0\}$.

Suppose $f$ is a non zero rigid function in $H$. We will prove later that $f$ is unbounded on $\Sigma_{1,F}$ (see the lemma below). For each $U_{s,\xi}$, the maximum principle implies that $f$ must obtain its maximum on the boundary annuli which are the tubes of the singular points on the special fibre. Therefore $f$ is unbounded on $\bigcup_{s':\mbox{even}}U_{s',\xi}$. But we know $\SL_2(\Q_p)$ acts on $\Sigma_{1,F}$ and acts transitively on the set of even vertices. Hence using the action of $\SL_2(\Q_p)$, we can get functions in $H$ with arbitrary large norms when restricted to $U_{s'_0,\xi}$ and $H_1$ cannot be compact. So there is no such $f$.
\end{proof}

\begin{lem} \label{gfb}
Any globally bounded function on a connected component of $\Sigma_{1,F}$ must be a constant.
\end{lem}
\begin{proof}
Fix a connected component $\xi$. Suppose $f$ is such a function. By multiplying $f$ by some powers of $\varpi$, we may assume $f\in H^0(\widehat{\Sigma_{1,O_F,\xi}},\cO_{\widehat{\Sigma_{1,O_F,\xi}}})$. Recall that the special fibre is connected and each irreducible component is a complete curve. Hence, 
\[H^0(\widehat{\Sigma_{1,O_F,\xi}},\cO_{\widehat{\Sigma_{1,O_F,\xi}}}/(\varpi))=\F_{p^2}.\]
Using induction on $n$, we can prove $H^0(\widehat{\Sigma_{1,O_F,\xi}},\cO_{\widehat{\Sigma_{1,O_F,\xi}}}/(\varpi^n))=O_F/(\varpi^n)$. Here we use the fact that $\cO_{\widehat{\Sigma_{1,O_F,\xi}}}$ is flat over the constant sheaf $O_F$. Now the lemma follows from:
\[ H^0(\widehat{\Sigma_{1,O_F,\xi}},\cO_{\widehat{\Sigma_{1,O_F,\xi}}})=\varprojlim_{n}H^0(\widehat{\Sigma_{1,O_F,\xi}},\cO_{\widehat{\Sigma_{1,O_F,\xi}}}/(\varpi^n))=O_F.\]
\end{proof}

\begin{rem}
The proposition is also true if $\chi'$ is trivial. In this case, it is equivalent with the same result on the Drinfel'd upper half plane. See Prop.19 of \cite{Tei} for a proof.
\end{rem}

So we have an injective $\GL_2(\Q_p)$-equivariant map: 
\[f_{\chi,[1,b]}:M(\chi,[1,b])\to (c-\indkg\rho_{\chi^{-1}})^\vee.\]
A simple consideration of the topology (see remark\ref{smdual}) shows that this induces a map:
\[c-\indkg\rho_{\chi^{-1}}\to B(\chi,[1,b]).\]
It is $\GL_2(\Q_p)$-equivariant and has to be injective if $B(\chi,[1,b])$ is non-zero since the left hand side is an irreducible representation of $\GL_2(\Q_p)$. 
If $B(\chi,[1,b])$ is non-zero, or equivalently $M(\chi,[1,b])$ is non-zero, we can define a lattice inside $c-\indkg\rho_{\chi^{-1}}$: 
\[ \Theta(\chi,[1,b])=\{X\in c-\indkg\rho_{\chi^{-1}},<X,f_{\chi,[1,b]}(Y)>\in O_E,\forall Y\in M(\chi,[1,b])  \} \]
where $<~,~>$ is the canonical pairing between $c-\indkg\rho_{\chi^{-1}}$ and $(c-\indkg\rho_{\chi^{-1}})^\vee$. This is equivalent with the intersection of the unit ball of $B(\chi,[1,b])$ with $c-\indkg\rho_{\chi^{-1}}$.

\begin{prop} \label{completion}
$B(\chi,[1,b])$ is the completion of $c-\indkg\rho_{\chi^{-1}}$ with respect to the lattice $\Theta(\chi,[1,b])$ if $M(\chi,[1,b])\neq 0$.
\end{prop}
\begin{proof}
The argument of Proposition 4.3.5 of \cite{Bre} works here. I would like to recall it here. By \cite{ST} Th.1.2., it suffices to prove that the natural map $M(\chi,[1,b])\to \Hom_{O_E}(\Theta(\chi,[1,b]),O_E)$ is a topological isomorphism. The topology on the right hand side is defined by pointwise convergence. Notice that $c-\indkg\rho_{\chi^{-1}}$ can be viewed as the continuous dual space of $(c-\indkg\rho_{\chi^{-1}})^\vee$ with the topology described in remark \ref{smdual} and $M(\chi,[1,b])$ is closed in $(c-\indkg\rho_{\chi^{-1}})^\vee$ since it's already compact. We can apply Corollary 13.5. of \cite{NFA} and get the desired isomorphism. It's also clear from the definition that this is a topological isomorphism.
\end{proof}

So if we can show $M(\chi,[1,b])$ is non-zero and moreover admissible defined in \cite{ST}, we indeed get an admissible Banach space representation of $\GL_2(\Q_p)$, which is a completion of the smooth representation $c-\indkg\rho_{\chi^{-1}}$. This is the goal of the rest of the paper.

\section{Computation of \texorpdfstring{$(H^0(\wtsow,\omega^1)\otimes_{\Z_p}O_E)^{\chi,\Gal(F/\Q_p)}/p$}{}} \label{cmodp}
Our ultimate goal is to prove $M(\chi,[1,b])$ is non-zero and admissible. The method is by explicit computation of its mod $p$ representation. First we review some notations defined in the previous few sections that will be constantly used from now on.

Let $\chi\in \chi(E)$ be a character of $(O_D/\Pi)^\times$ such that $\chi^p\neq \chi$. Since we fix an embedding $\tau:F_0\to E$, we may write $\chi=\tau\circ\chi'$, where $\chi'$ is a character of $(O_D/\Pi)^\times$ with values in $F_0^\times$. Then $\chi'=\chi_1^{-m}$, where $\chi_1$ is one of the fundamental characters (definition \ref{fch}) and $m\in \{1,\cdots,p^2-2\}$. Write $m=i+(p+1)j$ with $i\in\{1,\cdots,p\}$ and $j\in\{0,\cdots,p-2\}$. Finally, $g_\varphi\in\Gal(F/\Q_p)$ is the unique element that fixes $\varpi$ and acts as Frobenius on $F_0$.

Also recall that for any integer $n$, we use $[n]$ to denote the unique integer in $\{0,1,\cdots,p^2-2\}$ congruent to $n$ modulo $p^2-1$. For any $O_{F_0}$-module $A$, we denote $A\otimes_{O_{F_0},\tau}O_E$ by $A_{\tau}$ and $A\otimes_{O_{F_0},\bar{\tau}}O_E$ by $A_{\bar{\tau}}$.

Recall that 
\[\wtsow=\wtso\sqcup \widetilde{\Sigma_{1,O_F}'},\]
and $g_{\varphi}$ interchanges $(H^0(\wtso,\omega^1)\otimes_{\Z_p} O_E)^\chi$ and $(H^0(\widetilde{\Sigma_{1,O_F}'},\omega^1)\otimes_{\Z_p} O_E)^\chi$. Hence a $g_{\varphi}$-invariant element in $(H^0(\wtsow,\omega^1)\otimes_{\Z_p} O_E)^\chi$ is determined by its $(H^0(\wtso,\omega^1)\otimes_{\Z_p} O_E)^\chi$ component. By definition, $M(\chi,[1,b])$ is $g_{\varphi}$-invariant. Hence it suffices to work on $\wtso$. This means that we may identify $H^0(\wtso,\omega^1)$ as the $g_{\varphi}$-invariant sections of $H^0(\wtsow,\omega^1)$. Hence there is a natural action of $\GL_2(\Q_p)$ on it: this is nothing but $g_{\varphi}^{v_p(det(g))}\circ g$.

\begin{definition} \label{defoH}
For any $\chi\in\chi(E),~\chi'\in\chi(F_0)$, we define (see section \ref{ssm} for the definition of these formal schemes)
\begin{eqnarray*}
H^{(0),\chi,\Q_p}&=&(H^0(\wtsow,\omega^1/p)\otimes_{\Z_p}O_E)^{\chi,\Gal(F/\Q_p)}\\
H^{\chi,F_0}_{*}&=&(H^0(\widetilde{\Sigma_{1,O_F,*}},\omega^1/p)\otimes_{\Z_p}O_E)^{\chi,\Gal(F/F_0)}\\
H^{\chi',F_0}_* &=& H^0(\widetilde{\Sigma_{1,O_F,*}},\omega^1/p)^{\chi',\Gal(F/F_0)}\\
H^{\chi',F_0}_{*,?}&=&H^{\chi',F_0}_*\otimes_{O_{F_0},?}O_E=H^0(\widetilde{\Sigma_{1,O_F,*}},\omega^1/p)^{\chi',\Gal(F/F_0)}\otimes_{O_{F_0},?}O_E
\end{eqnarray*}
where $*$ is either a vertex $s$ or an edge $e$ of the Bruhat-Tits tree or nothing, and $?=\tau,\bar{\tau}$.
\end{definition}

It is clear from the definition that if $\chi=\tau\circ\chi'$, then
\begin{eqnarray}
H^{\chi,F_0}_{*}\simeq H^{\chi',F_0}_{*,\tau} \oplus H^{(\chi')^p,F_0}_{*,\bar{\tau}}.
\end{eqnarray}

Also, the discussion above shows that we have a canonical isomorphism:
\[H^{(0),\chi,\Q_p}\simeq H^{\chi,F_0}.\]

\begin{definition}\label{defoAs}
For a vertex $s$ in the Bruhat-Tits tree, we use $A(s)$ to denote the set of vertices adjacent to $s$.
\end{definition}

Now fix $\xi^{p-1}=-1$. We can do all the computation on one $\xi$-component $\wtx$. This is because $O_D^\times$ acts transitively on all connected components.

The goal of this section is to compute $(H^0(\wtsow,\omega^1)\otimes_{\Z_p}O_E)^{\chi,\Gal(F/\Q_p)}/p$. Next lemma implies that this is nothing but $H^{(0),\chi,\Q_p}$.

\begin{lem}
\[H^0(\wtso,\omega^1)/\varpi^n=H^0(\wtso,\omega^1/\varpi^n).\]
\end{lem}
\begin{proof}
Clearly there is an injection from LHS to RHS. Since we have 
\[H^0(\wtso,\omega^1)=\varprojlim_n H^0(\wtso,\omega^1/\varpi^n),\]
we only need to prove the canonical map $H^0(\wtso,\omega^1/\varpi^n)\to H^0(\wtso,\omega^1/\varpi^m),~n>m$ is surjective. Notice that $\omega^1$ is flat over the constant sheaf $O_F$. It suffices to prove $H^1(\wtso,\omega^1/\varpi^n)=0$ for all $n\in \N^+$. Do induction on $n$ and use the flatness again. It turns out that it's enough to show $H^1(\wtso,\omega^1/\varpi)=0$. However, the construction of $\omega^1$ tells us $\omega^1/\varpi$ is the dualizing sheaf on the special fibre. This means that if we restrict $\omega^1/\varpi$ to each irreducible component $V$ of the special fibre, it is $\Omega^1_V(D_{sing})$, where $\Omega^1_V$ is the usual sheaf of differential forms on $V$, $D_{sing}$ is the sum of singular points of $D$ (considered in the whole special fibre) as a divisor. Also, we have the following exact sequence of sheaves:
\[
0\to \omega^1/\varpi \to \prod_{V}i_{V*}(\Omega^1_V(D_{sing}))\to \prod_{E} i_{E*}(\F_{p^2})\to 0
\]
where $E$ (resp. $V$) runs through all singular points (resp. irreducbile components) of the special fibre, $i_E$ (resp. $i_V$) is the corresponding inclusion. Take the long exact sequence of cohomologies of this sequence. $H^0$ of the third map is surjective since the dual graph of the special fibre of each connected component is a tree. $H^1$ of the middle term in the exact sequence above vanishes by Riemann-Roch. So we indeed get the vanishing of $H^1(\wtso,\omega^1)$.
\end{proof}

Hence we only need to compute 
\begin{eqnarray} \label{dectp}
H^{(0),\chi,\Q_p}\simeq H^{\chi,F_0}\simeq H^{\chi',F_0}_{\tau}\oplus H^{(\chi')^p,F_0}_{\bar{\tau}}.
\end{eqnarray}
It's not hard to see that we have an injection:
\[H^{\chi',F_0}\hookrightarrow \prod_{s}H^{\chi',F_0}_s,\]
where $s$ takes values in the set of vertices of Bruhat-Tits tree. Similarly, we have the same injection for $H^{(\chi')^p,F_0}$. Notice that by identifying the sections on $\wtso$ as the $g_{\varphi}$-invariant sections on $\wtsow$, we have an action of $\GL_2(\Q_p)$ on 
\[\prod_s(H^{\chi',F_0}_s\oplus H^{(\chi')^p,F_0}_s)\]
(see the beginning of this section). Explicitly, if $v_p(det(g))$ ($g\in\GL_2(\Q_p)$) is even (resp. odd), $g$ sends $H^{\chi',F_0}_s$ to $H^{\chi',F_0}_{sg}$ (resp. $H^{(\chi')^p,F_0}_{sg}$). From this description, we have an obvious $\GL_2(\Q_p)$-equivariant isomorphism (recall $s'_0$ is the central vertex):
\[\prod_{s}(H^{\chi',F_0}_s\oplus H^{(\chi')^p,F_0}_s)\simeq \indkg H^{\chi',F_0}_{s'_0}\oplus \indkg H^{(\chi')^p,F_0}_{s'_0}\otimes_{\F_{p^2},\Fr}\F_{p^2}.\]

The following lemma basically says that we may identify $H^{\chi',F_0}_s$ with sections of $\omega^1/\varpi$ on $\overbar{U^0_s}$ defined in \ref{u_s}. Notice that $\omega^1/\varpi$ is the dualizing sheaf of the special fibre. 

\begin{lem} \label{pss}
For each vertex $s$ of the Bruhat-Tits tree, we have natural isomorphisms:
\begin{eqnarray*} 
\Psi_{s,\chi'}&:&H^{\chi',F_0}_s\stackrel{\sim}{\to} H^0(\wts,\omega^1/\varpi)^{\chi'}=H^0(\overbar{U_{s}^0},\omega^1/\varpi)^{\chi'}\\
\Psi_{s,(\chi')^p}&:&H^{(\chi')^p,F_0}_s\stackrel{\sim}{\to} H^0(\wts,\omega^1/\varpi)^{(\chi')^p}=H^0(\overbar{U_{s}^0},\omega^1/\varpi)^{(\chi')^p}
\end{eqnarray*}
such that their product $\prod_s(\Psi_{s,\chi'},\Psi_{s,(\chi')^p})$:
\begin{eqnarray*}
\prod_s H^{\chi',F_0}_s \oplus H^{(\chi')^p,F_0}_s\to\prod_s H^0(\overbar{U_{s}^0},\omega^1/\varpi)^{\chi'} \oplus \prod_s H^0(\overbar{U_{s}^0},\omega^1/\varpi)^{(\chi')^p}
\end{eqnarray*}
is $\GL_2(\Q_p)$-equivariant. As usual, $s$ takes value in the set of vertices of Bruhat-Tits tree.
\end{lem}

\begin{proof}
First let's see what happens when $s=s'_0$. Recall that we have a concrete description of $\wtxc,\wtc$ in section \ref{ssm} (\ref{eeq},\ref{eeqt}):
\begin{eqnarray*}
\wtxc\simeq\Spf O_{F_0}[\varpi][\eta,\frac{1}{\eta^p-\eta},\e]/(\e^{p+1}+v_1w_1^{-1}\xi\frac{\eta^p-\eta}{(p/\eta)^{p-1}-1})\\ 
\wtc\simeq\Spf O_{F_0}[\varpi][\eta,\frac{1}{\eta^p-\eta},\e]/(\e^{p^2-1}-w_1^2(\frac{\eta^p-\eta}{(p/\eta)^{p-1}-1})^{p-1})
\end{eqnarray*}
An element of $H^0(\wtc,\omega^1)^{\chi'}$ is determined by its restriction on $\wtxc$. It's easy to see it must have the form (using the results in section \ref{act}):
\[P(\eta)\e^{p+1-i}\frac{d\e}{\e}\]
where $P(\eta)\in O_F[\eta,1/(\eta^{p-1}-1)]~\widehat{}$. Recall that (proposition \ref{sav}) $\chi'=\chi_1^{-m}$, and $m=i+(p+1)j,i\in\{1,\cdots,p\},j\in\{0,\cdots,p-2 \}$. It is $\Gal(F/F_0)$-invariant if and only if
\[P(\eta)=\varpi^{p^2-1-m}F_1(\eta),\]
where $F_1(\eta)\in O_{F_0}[\eta,1/(\eta^{p-1}-1)]~\widehat{}$. Similarly, a section of $H^0(\wtc,\omega^1)^{(\chi')^p}$ fixed by $\Gal(F/F_0)$ must have the form:
\[\varpi^{[-mp]}F_2(\eta)\e^i\frac{d\e}{\e},\]
where $F_2(\eta)\in O_{F_0}[\eta,1/(\eta^{p-1}-1)]~\widehat{}$, and $[-mp]$ is defined in the beginning of this section.

Thus any element $\bar{F}$ of $H^0(\wtc,\omega^1/p)^{\chi',\Gal(F/F_0)}=H^{\chi',F_0}_s$ can be written uniquely as 
\[\varpi^{p^2-1-m}\bar{F_1}(\eta)\e^{p+1-i}\frac{d\e}{\e}\] 
on $\xi$-component, where $\bar{F_1}(\eta)\in \F_{p^2}[\eta,1/(\eta^{p-1}-1)]$. Now define $\Psi_{s'_0,\chi'}(\bar{F})=\bar{F_1}(\eta)\e^{p+1-i}\frac{d\e}{\e}$. Equivalently, it is `multiplication' by $\varpi^{-(p^2-1-m)}$. It's trivial to see this is indeed an isomorphism. We can define $\Psi_{s'_0,(\chi')^p}$ in exactly the same way.

Note that $\Psi_{s'_0,\chi'},\Psi_{s'_0,(\chi')^p}$ are $\GL_2(\Z_p)$-equivariant, we can extend both isomorphisms to any vertex $s$ using the action of $\GL_2(\Q_p)$. Concretely, for an even vertex $s'$, $\Psi_{s',\chi'}$ is `multiplication' by $\varpi^{-(p^2-1-m)}$, $\Psi_{s',(\chi')^p}$ is `multiplication' by $\varpi^{-[-mp]}$. For an odd vertex $s$, $\Psi_{s,\chi'}$ is `multiplication' by $\varpi^{-[-mp]}$, $\Psi_{s,(\chi')^p}$ is `multiplication' by $\varpi^{-(p^2-1-m)}$.
\end{proof}

By abuse of notations, I will identify $H^0(\overbar{U_{s}^0},\omega^1/\varpi)^{\chi'},H^0(\overbar{U_{s}^0},\omega^1/\varpi)^{\chi'}_{\tau}$ with $H^{\chi',F_0}_s,H^{\chi',F_0}_{s,\tau}$ via the isomorphisms in the previous lemma. Notice that $\omega^1/\varpi$ is the sheaf of differential forms on $\overbar{U_{s}^0}$, thus we may view elements in $H^{\chi',F_0}_s$ as meromorphic differential forms on $\overbar{U_s}$.

From now on, I would like to describe the element of $H^{\chi,F_0}$ via its image in $\prod_s H^{\chi',F_0}_{s,\tau} \oplus \prod_s H^{(\chi')^p,F_0}_{s,\bar{\tau}}$. In other words, using the previous lemma, any element $h=(h_1,h_2)$ in $H^{\chi,F_0}\simeq H^{\chi',F_0}_{\tau} \oplus H^{(\chi')^p,F_0}_{\bar{\tau}}$ corresponds to a family of meromorphic differential forms
\[\{(\omega_{s,\tau},\omega_{s,\bar{\tau}})\}_s,\]
where $\omega_{s,\tau}=h_1|_{\wts}\in H^{\chi',F_0}_{s,\tau}$ and $\omega_{s,\bar{\tau}}=h_2|_{\wts}\in H^{(\chi')^p,F_0}_{s,\bar{\tau}}$.

To further determine $H^{\chi,F_0}$, we need to know when such a $\{(\omega_{s,\tau},\omega_{s,\bar{\tau}})\}_s$ comes from a global section. We will give the necessary condition in proposition \ref{klem1} and sufficient condition in proposition \ref{klem2}. To this end, it is crucial to understand the local structure of $\omega^1$ on $\wtx$. Recall that $\wtx$ has an open covering $\{ \wtxe \}_e$ and an explicit description of $ \wtxe$ (lemma \ref{sms}) is:
\[\frac{\Spf O_F[\eta,\zeta,\frac{1}{\eta^{p-1}-1},\frac{1}{\zeta^{p-1}-1},\e,\e']~\widehat{}}
{(\e^{p+1}+v_1w_1^{-1}\xi\frac{\eta^p-\eta}{\zeta^{p-1}-1},\e'^{p+1}+v_1^{-1}w_1\xi\frac{\zeta^p-\zeta}{\eta^{p-1}-1},\e\e'-\varpi^{p-1}\xi)}.\]
Note that $e/\varpi,e'/\varpi$ there is $\e,\e'$ here respectively. Suppose $e=[s,s']$, where $s'$ (resp. $s$) is an even (resp. odd) vertex and corresponds to $\eta$ (resp. $\zeta$). It's not too hard to see

\begin{lem} \label{ld0}
Any element $h$ of $H^0(\widetilde{\Sigma_{1,O_F,[s,s']}},\omega^1)^{\chi',\Gal(F/F_0)}$, when restricted to $\wtxsp$, can be written as the following form: 
\[h=\varpi^{p^2-1-m}f(\eta)\e^{p+1-i}\frac{d\e}{\e}+\varpi^{[-mp]}g(\zeta)\e'^{i}\frac{d\e'}{\e'},\]
where $f(\eta)\in O_{F_0}[\eta,1/(\eta^{p-1}-1)]~\widehat{},~g(\zeta)\in O_{F_0}[\zeta,1/(\zeta^{p-1}-1)]~\widehat{}$.
\end{lem}

\begin{proof}
It suffices to verify this after reducing modulo $p$. Equivalently, we need to show that any $h\in H^{\chi',F_0}_e$ has the form:
\[\varpi^{p^2-1-m}f(\eta)\e^{p+1-i}\frac{d\e}{\e}+\varpi^{[-mp]}g(\zeta)\e'^{i}\frac{d\e'}{\e'},\]
where $f(\eta)\in\F_{p^2}[\eta,1/(1-\eta^{p-1})],~g(\zeta)\in\F_{p^2}[\zeta,1/(1-\zeta^{p-1})]$ when restricted to $\wtxsp$.

Recall that $\omega^1$ is free over $\wtxsp$ with a basis $\frac{d\e}{\e}=-\frac{d\e'}{\e'}$ (see the beginning of the previous section). Hence any element $h$ in $H^0(\wtxsp,\omega^1/p)$ can be written as:
\[\sum_{k=0}^p f_{1,k}(\eta,\zeta)\e^k\frac{d\e}{\e}+\sum_{k=0}^p g_{1,k}(\eta,\zeta)\e'^k\frac{d\e'}{\e'},\]
where $f_{1,k}(\eta,\zeta),g_{1,k}(\eta,\zeta)\in O_F/(p)[\eta,\zeta,1/(1-\eta^{p-1}),1/(1-\zeta^{p-1})]/(\eta\zeta)$. This is because using the explicit description of $\wtxsp$ above, we see that $\e^{p+1},\e'^{p+1},\e\e'$ all can be written as an element only containing $\eta,\zeta$.

Using the results in section \ref{act}, we see that such an element comes from an element in the $\chi'$-isotypic component of $H^0(\widetilde{\Sigma_{1,O_F,[s,s']}},\omega^1/p)$ if and only the coefficients of $\e^k$ (resp. $\e'^k$) are zero unless $k=p+1-i$ (resp. $k=i$). Hence we may write it as:
\begin{eqnarray} \label{lrf}
h=f_{1,p+1-i}(\eta,\zeta)\e^{p+1-i}\frac{d\e}{\e}+ g_{1,i}(\eta,\zeta)\e'^i\frac{d\e'}{\e'},
\end{eqnarray}

Next consider the action of $\Gal(F/F_0)$. Using the results in section \ref{act} once again, it's not hard to see that such an element comes from a Galois-invariant section if and only if
\begin{eqnarray} \label{glrf}
f_{1,p+1-i}(\eta,\zeta)=\varpi^{p^2-1-m}f_2(\eta,\zeta), ~g_{1,i}(\eta,\zeta)=\varpi^{[-mp]}g_2(\eta,\zeta),
\end{eqnarray}
where $f_2(\eta,\zeta),g_2(\eta,\zeta)\in \F_{p^2}[\eta,\zeta,1/(1-\eta^{p-1}),1/(1-\zeta^{p-1})]/(\eta\zeta)$.

Now in order to prove the lemma, we need to `eliminate' the $\zeta$ in $f_2(\eta,\zeta)$ and $\eta$ in $g_2(\eta,\zeta)$. We will prove this under the following assumption:
\[p^2-1-m\geq [-mp].\]
Equivalently, this means $p^2-1-m=[-mp]+i(p-1)$. The other case is similar. 

First we eliminate the  $\eta$ in $g_2(\eta,\zeta)$: we can write 
\[g_2(\eta,\zeta)=f_3(\eta)+g_3(\zeta),\]
such that $g_3(\eta)\in \F_{p^2}[\eta,1/(1-\eta^{p-1})]$ and $g_3(\zeta)\in \F_{p^2}[\zeta,1/(1-\zeta^{p-1})]$. This is because we can think $g_2(\eta,\zeta)$ as a regular function on a union of two irreducible smooth affine curves crossing transversally. Such a decomposition is obtained by restricting this function on each irreducible component (with some modification by some constants).

Notice that $f_3(0)$ makes sense here. Replace $f_3(\eta)$ by $f_3(\eta)-f_3(0)$, we may assume 
\[f_3(\eta)=\eta f_4(\eta)\]
where $f_4(\eta)\in\F_{p^2}[\eta,1/(1-\eta^{p-1})]$. Now in $\cO_{\wtxsp}$, we have
\[\eta=C\e^{p+1},~\mbox{where  } C=-v_1w_1\xi^{-1}\frac{\zeta^{p-1}-1}{\eta^{p-1}-1}.\]
Plug this into \eqref{lrf} and use \eqref{glrf}:
\begin{eqnarray*}
h&=& \varpi^{p^2-1-m}f_2(\eta,\zeta)\e^{p+1-i}\frac{d\e}{\e}+\varpi^{[-mp]}(\eta f_4(\eta)+g_3(\zeta))\e'^i\frac{d\e'}{\e'} \\
&=& \varpi^{p^2-1-m}f_2(\eta,\zeta)\e^{p+1-i}\frac{d\e}{\e}+\varpi^{[-mp]}g_3(\zeta)\e'^i\frac{d\e'}{\e'}+f_4(\eta)\varpi^{[-mp]}C\e^{p+1}\e'^i\frac{d\e'}{\e'}.
\end{eqnarray*}
Since $\e\e'=\varpi^{p-1}\xi$, the last term in the above equation is
\[f_4(\eta)C\varpi^{[-mp]}\varpi^{i(p-1)}\xi^i\e^{p+1-i}\frac{d\e'}{\e'}=f_4(\eta)C\varpi^{p^2-1-m}\xi^i\e^{p+1-i}\frac{d\e'}{\e'}=-\varpi^{p^2-1-m}Cf_4(\eta)\xi^i\e^{p+1-i}\frac{d\e}{\e}\]
by our assumption. In other words, 
\[h=\varpi^{p^2-1-m}(f_2(\eta,\zeta)-Cf_4(\eta)\xi^i)\e^{p+1-i}\frac{d\e}{\e}+\varpi^{[-mp]}g_3(\zeta)\e'^i\frac{d\e'}{\e'}.\]
Hence in \eqref{lrf}, we may assume $g_{1,i}(\eta,\zeta)=\varpi^{[-mp]}g_3(\zeta)$, where $g_3(\zeta)\in \F_{p^2}[\zeta,1/(1-\zeta^{p-1})]$.

Now we are going to eliminate the $\zeta$ in $f_2(\eta,\zeta)$. As before, write $f_2(\eta,\zeta)=f_5(\eta)+\zeta g_5(\zeta)$ and notice that in $\cO_{\wtxsp}$, we can write $\zeta=C'\e'^{p+1}$. Plug this into \eqref{lrf}:
\begin{eqnarray*}
h&=& \varpi^{p^2-1-m}(f_5(\eta)+\zeta g_5(\zeta))\e^{p+1-i}\frac{d\e}{\e}+\varpi^{[-mp]}g_3(\zeta)\e'^i\frac{d\e'}{\e'} \\
&=& \varpi^{p^2-1-m}f_5(\eta)\e^{p+1-i}\frac{d\e}{\e}+\varpi^{p^2-1-m}g_5(\zeta)C'\e'^{p+1}\e^{p+1-i}\frac{d\e}{\e}+\varpi^{[-mp]}g_3(\zeta)\e'^i\frac{d\e'}{\e'}.
\end{eqnarray*}
Here comes the difference between this case and the former case. The middle term actually vanishes:
\[\varpi^{p^2-1-m}g_5(\zeta)C'\e'^{p+1}\e^{p+1-i}\frac{d\e}{\e}=\varpi^{p^2-1-m}g_5(\zeta)C'\varpi^{(p+1-i)(p-1)}\xi^{p+1-i}\e'^i\frac{d\e}{\e}=0,\]
since $\varpi^{p^2-1-m+(p+1-i)(p-1)}=\varpi^{[-mp]+(p+1)(p-1)}=-p\cdot\varpi^{-[mp]}=0$ by our assumption. Hence we may write $h=\varpi^{p^2-1-m}f_5(\eta)\e^{p+1-i}\frac{d\e}{\e}+\varpi^{[-mp]}g_3(\zeta)\e'^i\frac{d\e'}{\e'}$, which is exactly what we want.
\end{proof}

Now suppose $h\in H^{\chi',F_0}_e$. We may assume it has the form:
\[\varpi^{p^2-1-m}f(\eta)\e^{p+1-i}\frac{d\e}{\e}+\varpi^{[-mp]}g(\zeta)\e'^{i}\frac{d\e'}{\e'},\]
where $f(\eta)\in\F_{p^2}[\eta,1/(1-\eta^{p-1})],~g(\zeta)\in\F_{p^2}[\zeta,1/(1-\zeta^{p-1})]$ when restricted to $\wtxsp$. What's its restriction to $\wtxp$? Algebraically, this means that we replace $\zeta$ by $\frac{p}{\eta}=0$ and $\e'$ by $\frac{\varpi^{p-1}\xi}{\e}$. So we have (notice that $\frac{d\e}{\e}=-\frac{d\e'}{\e'}$):
\[h|_{\wtxp}=\varpi^{p^2-1-m}f(\eta)\e^{p+1-i}\frac{d\e}{\e}-\varpi^{[-mp]}g(0)\varpi^{i(p-1)}\xi^i\e^{-i}\frac{d\e}{\e}\]

We make the following assumption in the rest of this section:
\[p^2-1-m\geq [-mp].\]
Equivalently, this means $p^2-1-m=[-mp]+i(p-1)$.

On $\wtxp$, we have 
\[\e^{-i}=\frac{\e^{p+1-i}}{\e^{p+1}}=-\frac{(p/\eta)^{p-1}-1}{v_1w_1^{-1}\xi(\eta^p-\eta)}\e^{p+1-i}\equiv \frac{1}{v_1w_1^{-1}\xi(\eta^p-\eta)}\e^{p+1-i}~(\modd~p).\]
Hence, 
\begin{eqnarray}
h|_{\wtxp}&=&\varpi^{p^2-1-m}f(\eta)\e^{p+1-i}\frac{d\e}{\e}-\varpi^{p^2-1-m}g(0)\xi^i\frac{1}{v_1w_1^{-1}\xi(\eta^p-\eta)}\e^{p+1-i}\frac{d\e}{\e}\\ \label{rese}
&=& \varpi^{p^2-1-m}(f(\eta)+g(0)\xi^{i-1}v_1^{-1}w_1(\frac{1}{\eta}-\frac{\eta^{p-2}}{\eta^{p-1}-1}))\e^{p+1-i}\frac{d\e}{\e}
\end{eqnarray}

Write $F(\eta)=f(\eta)-g(0)\xi^{i-1}v_1^{-1}w_1\frac{\eta^{p-2}}{\eta^{p-1}-1},~C_1=g(0)\xi^{i-1}v_1^{-1}w_1$.

\begin{lem} \label{klem11}
Under the assumption $p^2-1-m\ge [-mp]$,
\[h|_{\wtxp}=\varpi^{p^2-1-m}C_1\frac{\e^{p+1-i}}{\eta}\frac{d\e}{\e}+\varpi^{p^2-1-m}F(\eta)\e^{p+1-i}\frac{d\e}{\e},\]
where $F(\eta)\in \F_{p^2}[\eta,1/(1-\eta^{p-1})],~C_1\in \F_{p^2}$.
\end{lem}

Now if we view $h|_{\wtxp}$ as a differential form on $\overbar{U_{s',\xi}^{0}}$, or what's the same thing, a meromorphic differential form on $\overbar{U_{s',\xi}}$ with poles at the singular points ($\overbar{U_{s',\xi}}$ is viewed as a subvariety in the special fibre of $\wtx$), the order of the pole at the intersection point of $\overbar{U_{s',\xi}}$ and $\overbar{U_{s,\xi}}$ must be $i+1$ (if there is a pole) since $\frac{1}{\eta}$ has order $p+1$ at this point ($\eta=\e=0$) and $\e$ is a uniformizer of this point.

Now restrict $h$ to $\wtxs$. This time we replace $\eta$ by $\frac{p}{\zeta}=0$ and $\e$ by $\frac{\varpi^{p-1}\xi}{\e'}$.
\begin{eqnarray}
h|_{\wtxs}&=&-\varpi^{p^2-1-m}f(0)\varpi^{(p+1-i)(p-1)}\xi^{p+1-i}\e'^{-(p+1-i)}\frac{d\e'}{\e'}+\varpi^{[-mp]}g(\zeta)\e'^i\frac{d\e'}{\e'}\\ \label{resp}
&=&\varpi^{[-mp]}g(\zeta)\e'^i\frac{d\e'}{\e'}
\end{eqnarray}
The first term is zero since $\varpi^{p^2-1-m+(p+1-i)(p-1)}=-p\cdot\varpi^{-[mp]}$ by our assumption. 
\begin{lem} \label{klem12}
Under the assumption $p^2-1-m\ge[-mp]$,
\[h|_{\wtxs}=\varpi^{[-mp]}g(\zeta)\e'^i\frac{d\e'}{\e'},\]
where $g(\zeta)\in\F_{p^2}[\zeta,1/(1-\zeta^{p-1})]$.
\end{lem}
Thus if we view $h|_{\wtxs}$ as a meromorphic differential form on $\overbar{U_{s,\xi}}$, it is holomorphic at the intersection point of $\overbar{U_{s,\xi}}$ and $\overbar{U_{s',\xi}}$. In summary,

\begin{prop} \label{klem1}
Assume $p^2-1-m\geq [-mp]$. Under the identification in lemma \ref{pss}, an element $h$ of $H^{\chi',F_0}=H^0(\wtso,\omega^1/p)^{\chi',\Gal(F/F_0)}$ has the following description:
\begin{enumerate}
\item If $s$ is odd, then $h|_{\wtxs}$ is a holomorphic differential form on $\overbar{U_{s,\xi}}$.
\item If $s'$ is even, then $h|_{\wtxp}$ can have poles at the intersection points of $\overbar{U_{s',\xi}}$ with adjacent components. The order of these poles have to be ($i+1$) if there are poles. Moreover, as an element of the space of meromorphic differential forms on $\overbar{U_{s',\xi}}$ modulo holomorphic differential forms, it is uniquely determined by the restriction of $h$ on the components adjacent to $s'$. In other words, $h|_{\wtxp}$ is holomorphic on $\overbar{U_{s',\xi}}$ if the restriction of $h$ on the components adjacent to $s'$ are zero.
\end{enumerate}
\end{prop}

\begin{proof}
The first part is a direct consequence of lemma \ref{klem12}. The assertion for the order of poles follows from lemma \ref{klem11}. As for the last assertion, using the notations before lemma \ref{klem11}, we know that the pole of $h|_{\wtxp}$ at the intersection point of $\overbar{U_{s',\xi}}$ and $\overbar{U_{s,\xi}}$ is determined by $g(0)$ (in fact this pole is given by $g(0)\xi^{i-1}v_1^{-1}w_1\frac{\e^{p+1-i}}{\eta}\frac{d\e}{\e}$ modulo holomorphic terms). However, $g(0)$ is indeed determined by $h|_{\wtxs}$ since $h|_{\wtxs}=\varpi^{[-mp]}g(\zeta)\e'^i\frac{d\e'}{\e'}$.
\end{proof}

\begin{rem} \label{oeoe}
Under the assumption $p^2-1-m\geq [-mp]$, we have a similar description for elements in $H^{(\chi')^p,F_0}$ while interchanging the descriptions for odd and even vertices. This is obvious if one uses the action of $\GL_2(\Q_p)$.

If we assume $p^2-1-m\leq [-mp]$, an element $h$ of $H^{\chi',F_0}$ has the following similar description:
\begin{enumerate}
\item If $s'$ is even, then $h|_{\wtxp}$ is a holomorphic differential form on $\overbar{U_{s',\xi}}$.
\item If $s$ is odd, then $h|_{\wtxs}$ can have poles at the intersection points of $\overbar{U_{s,\xi}}$ with adjacent components. The order of these poles have to be ($p+2-i$) if there are poles. Moreover, $h|_{\wtxs}$ is holomorphic on $\overbar{U_{s,\xi}}$ if the restriction of $h$ on the components adjacent to $s$ are zero.
\end{enumerate}
\end{rem}

To get a converse result, we need one more lemma to see when we can glue sections on $\wts,\wtp$ to a section on $\wtsp$.
\begin{lem}\label{glem}
Assume $p^2-1-m\geq [-mp]$, and $s'$ is an even vertex and $s\in A(s')$. Given $h_{s'}\in H^{\chi',F_0}_{s'},h_s\in H^{\chi',F_0}_{s}$ such that they have the forms in lemma \ref{klem11} and \ref{klem12} (under the explicit description in lemma \ref{sms}):
\begin{eqnarray} \label{lfp}
h_{s'}|_{\wtxp}&=&\varpi^{p^2-1-m}C_1\frac{\e^{p+1-i}}{\eta}\frac{d\e}{\e}+\varpi^{p^2-1-m}F(\eta)\e^{p+1-i}\frac{d\e}{\e},\\ \label{lfs}
h_s|_{\wtxs}&=&\varpi^{[-mp]}g(\zeta)\e'^i\frac{d\e'}{\e'},
\end{eqnarray}
where $F(\eta)\in \F_{p^2}[\eta,1/(1-\eta^{p-1})],~C_1\in \F_{p^2},~g(\zeta)\in\F_{p^2}[\zeta,1/(1-\zeta^{p-1})]$. Moreover assume
\begin{eqnarray}\label{lcc}
C_1=g(0)\xi^{i-1}v_1^{-1}w_1.
\end{eqnarray}
Then we can find a (unique) section $h\in H^{\chi',F_0}_{[s,s']}$ such that 
\[h|_{\wtp}=h_{s'},~h|_{\wts}=h_s.\]
\end{lem}
\begin{proof}
It is direct to see that the following section $h_{\xi}$ on $\wtxsp$ can be extended to an element in $H^{\chi',F_0}_{[s,s']}$ and satisfies all the conditions:
\[h_{\xi}=\varpi^{p^2-1-m}(F(\eta)+C_1\frac{\eta^{p-2}}{\eta^{p-1}-1})\e^{p+1-i}\frac{d\e}{\e}+\varpi^{[-mp]}g(\zeta)\e'^i\frac{d\e'}{\e'}.\]
\end{proof}

\begin{prop} \label{klem2}
Assume $p^2-1-m\geq [-mp]$. 
\begin{enumerate}
\item Given $h_s\in H^{\chi',F_0}_s$ for each odd vertex $s$ that corresponds to a holomorphic differential form on $\overbar{U_{s,\xi}}$, we can find an element $h$ in $H^{\chi',F_0}$ such that for any odd vertices $s$, 
\[h|_{\wts}=h_s.\]
\item Moreover, we have the following freedom of choosing $h$: given $f_{s'}\in H^{\chi',F_0}_{s'}$ for each even vertex $s'$ that corresponds to a holomorphic differential form on $\overbar{U_{s',\xi}}$, we may find a (unique) element $f$ in $H^{\chi',F_0}$ such that
\begin{eqnarray*}
f|_{\wtp}&=&f_{s'},~ \mbox{for any even vertices } s',\\
f|_{\wts}&=&0,~\mbox{for any odd vertices } s.
\end{eqnarray*}
\end{enumerate}
\end{prop}
\begin{proof}
Both are local questions. The second part is a direct consequence of the previous lemma: for any even vertex $s'$ and $s\in A(s')$, applying the previous lemma with $h_s=0,~h_{s'}=f_{s'}$ (in this case, $C_1=0$), we can glue to a section on $\wtsp$ whose restriction to $\wtp$ (resp. $\wts$) is $f_{s'}$ (resp. zero). Hence we can glue to a global section on $\wtso$.

As for the first part, our strategy is similar. For any even vertex $s'$, we will find a section $h_{s'}\in H^{\chi',F_0}_{s'}$ such that for any vertex $s\in A(s')$, we can use the previous lemma to glue $h_s,h_{s'}$ to a section on $\wtsp$ and obtain a global section on $\wtso$.

By lemma \ref{pss}, we may identify elements in $H^{\chi',F_0}_{s'}$ with differential forms on $\overbar{U_{s'}^0}$. Since $(O_D/\Pi)^\times\simeq\F_{p^2}^\times$ acts transitively on the connected components of $\overbar{U_{s'}^0}$, it is easy to see $\mu_{p+1}(\F_{p^2})=\{a\in\F_{p^2},a^{p+1}=1\}$ fixes $\overbar{U_{s',\xi}^0}$. As we noted in the proof of lemma \ref{c_x},
\[H^0(\overbar{U_{s'}^0},\omega^1/\varpi)^{\chi'}\simeq H^0(\overbar{U_{s',\xi}^0},\omega^1/\varpi)^{\Id^{-i}},\]
where we view $\Id:\mu_{p+1}(\F_{p^2})\to \F_{p^2}^\times$ as a character of $\mu_{p+1}(\F_{p^2})$, and $\Id^{-i}$ is its $(-i)$-th power. We denote the intersection point of $\overbar{U_{s',\xi}}$ with $\overbar{U_{s,\xi}},$ by $P_s$ for $s\in A(s')$. 

Now using lemma \ref{glem}, the question of finding such an $h_{s'}\in H^{\chi',F_0}_{s'}$ is equivalent with finding a meromorphic differential form $\omega_{s'}\in H^0(\overbar{U_{s',\xi}^0},\omega^1/\varpi)^{\Id^{-i}}$ such that
\begin{itemize}
\item it can only have poles at $P_s,s\in A(s')$ with order at most $i+1$ (in fact, it has to be $i+1$ if there is a pole by considering the action of $\mu_{p+1}(\F_{p^2})$).
\item The `leading coefficient' of the pole at $P_s$ is prescribed by $h_s$ for all $s\in A(s')$.
\end{itemize}
More precisely, using the explicit description in lemma \ref{sms}, the first condition allows us to write $\omega_{s'}$ into the form \eqref{lfp}. Also our condition in the proposition allows us to write $h_s$ into the form \eqref{lfs}. Then $C_1$ in \eqref{lfp} is the leading coefficient in this case and we want it to satisfy the equation \eqref{lcc}. 

The existence of such a meromorphic differential form follows from the following
\begin{lem} \label{aorr}
Let $C$ be a smooth geometrically connected curve over $\F_{p^2}$ and $\{P_k\}_k$ be a non-empty finite subset of $C(\F_{p^2})$. Then for $n\geq 2$, the restriction map:
\[H^0(C,\Omega^1_C(nD))\to \bigoplus_k H^0(P_k,\Omega^1_C(nD)\big|_{P_k})\]
is surjective, where $D$ is the divisor $\sum_k P_k$.
\end{lem}
Assume this lemma for the moment. In our case, let $C=\overbar{U_{s',\xi}}, ~\{P_k\}=\{P_s\}$ and $n=i+1$. The prescribed leading coefficients become a family of elements $c_s\in H^0(P_s,\Omega^1_C(nD)\big|_{P_s}),~s\in A(s')$. Notice that the uniformizer for $P_s$ is either $\e$ or $\frac{\e}{\eta}$, hence $\mu_{p+1}(\F_{p^2})$ acts on 
\[H^0(P_s,\Omega^1_C(\sum_k (i+1)P_k)\big|_{P_s})=H^0(P_s,\Omega^1_C((i+1)D)\big|_{P_s})\]
via $\Id^{-i}$. So taking the $\Id^{-i}$-isotypic component of the map in the lemma (which remains surjective since $p+1$ is coprime to $p$), we may find an element in $H^0(\overbar{U_{s',\xi}},\Omega^1(\sum_s (i+1)P_s))^{\Id^{-i}}$ having the correct leading coefficient at each $P_s$ and that's exactly what we want.
\end{proof}
\begin{proof}[Proof of lemma \ref{aorr}]
Consider the following short exact sequence of sheaves
\[0\to \Omega^1_C((n-1)D)\to \Omega^1_C(nD)\to \bigoplus_k \Omega^1_C(nD)\big|_{P_k}\to 0.\]
It suffices to show $H^1(C,\Omega^1_C((n-1)D))$ vanishes. However by Serre duality, this space is dual to $H^0(C,\cO_C(-(n-1)D))$ which is zero since we assume $n\geq 2$. 
\end{proof}

Now, we can prove the main proposition of this section.

\begin{prop} \label{dmp}
Assume $p^2-1-m\geq [-mp]$. There exists a $\GL_2(\Q_p)$-equivariant short exact sequence:
\begin{eqnarray*}
0 \to \indkg H^0(\overbar{U_{s'_0}},\Omega^1_{\overbar{U_{s'_0}}})^{\chi'}_{\tau} \to H^{(0),\chi,\Q_p}\to \indkg H^0(\overbar{U_{s'_0}},\Omega^1_{\overbar{U_{s'_0}}})^{(\chi')^p}_{\bar{\tau}} \to 0.  
\end{eqnarray*}
\end{prop}
\begin{proof}
Write $\prod_s H^{\chi,F_0}_s=\prod_s(H^{\chi',F_0}_{s,\tau}\oplus H^{(\chi')^p,F_0}_{s,\bar{\tau}})$, where as usual, $s$ runs over the vertices of the Bruhat-Tits tree. Define
\begin{eqnarray*}
H_1=\prod_{s':\mbox{even}} H^{\chi',F_0}_{s',\tau} \oplus \prod_{s:\mbox{odd}}H^{(\chi')^p,F_0}_{s,\bar{\tau}} \\
H_2=\prod_{s:\mbox{odd}} H^{\chi',F_0}_{s,\tau}\oplus \prod_{s':\mbox{even}}H^{(\chi')^p,F_0}_{s',\bar{\tau}}.
\end{eqnarray*}
Notice that $\GL_2(\Q_p)$ actually acts on $H_1,H_2$. Then we have a $\GL_2(\Q_p)$-equivariant (split) short exact sequence:
\[0\to H_1\to \prod_s H^{\chi,F_0}_s \to H_2 \to 0.\]

Recall that we have an injection of $H^{(0),\chi,\Q_p}\simeq H^{\chi,F_0}$ into $\prod_s H^{\chi,F_0}_s$. So this short exact sequence induces another short exact sequence:
\[0\to K \to H^{\chi,F_0} \to C \to 0\]
It remains to determine $K,C$.

Let $f$ be an element of $H^{\chi,F_0}$. We will write $f=f_{\tau}+f_{\bar{\tau}}$ under the decomposition $H^{\chi,F_0}\simeq H^{\chi',F_0}_{\tau} \oplus H^{(\chi')^p,F_0}_{\bar{\tau}}$ (see \eqref{dectp}).

Suppose $f$ is in $K$. This means for any odd (resp. even) vertex $s$ (resp. $s'$), 
\[ f_{\tau}|_{\wts}=0~(\mbox{resp. }f_{\bar{\tau}}|_{\wtp}=0).\]
By the second part of proposition \ref{klem1}, we know that $f_{\tau}|_{\wtp}$ corresponds to a holomorphic differential form on $\overbar{U_{s',\xi}}$ for any even vertex $s'$ (tensored with $O_E$). However the second part of proposition \ref{klem2} indicates that $f_{\tau}|_{\wtp}$ can be any holomorphic differential form inside $H^0(\overbar{U_{s'}},\Omega^1_{\overbar{U_{s'}}})^{\chi'}_{\tau}$. Similarly $f_{\bar{\tau}}|_{\wts}$ can be any holomorphic differential form inside $H^0(\overbar{U_{s}},\Omega^1_{\overbar{U_{s}}})^{\chi'}_{\tau}$, where $s$ is an odd vertex. This certainly implies that 
\[K\simeq \indkg H^0(\overbar{U_{s'_0}},\Omega^1_{\overbar{U_{s'_0}}})^{\chi'}_{\tau}.\]

By the first part of proposition \ref{klem1}, we know that $C$ is inside 
\[\prod_{s:\mbox{odd}} H^0(\overbar{U_s},\Omega^1_{\overbar{U_s}})^{\chi'}_{\tau} \oplus \prod_{s':\mbox{even}} H^0(\overbar{U_{s'}},\Omega^1_{\overbar{U_{s'}}})^{(\chi')^p}_{\bar{\tau}},\]
as a subet of $H_2$. However the first part of proposition \ref{klem2} tells us that in fact $C$ is equal to this set. Clearly this is nothing but $\indkg H^0(\overbar{U_{s'_0}},\Omega^1_{\overbar{U_{s'_0}}})^{(\chi')^p}_{\bar{\tau}}$.
\end{proof}

\begin{rem} \label{rsym}
See the beginning of the paper for the notations here. Under the isomorphisms \eqref{les'0}, an element of $H^0(\overbar{U_{s'_0}},\Omega^1_{\overbar{U_{s'_0}}})^{\chi'}$ must have the form $f(\eta)\e^{p+1-i}\frac{d\e}{\e}$ on $\overbar{U_{s'_0,\xi}}$, where $f(\eta)$ is a polynomial of $\eta$ of degree at most $i-2$. Using the results in section \ref{act}, it's not hard to construct a $\GL_2(\F_p)$-equivariant isomorphism:
\begin{eqnarray}
H^0(\overbar{U_{s'_0}},\Omega^1_{\overbar{U_{s'_0}}})^{\chi'} &\to& (\Sym^{i-2}\F_{p^2}^2)\otimes \det {}^{j+1}\\
\eta^r\e^{p+1-i}\frac{d\e}{\e} &\mapsto& x^r y^{i-2-r},
\end{eqnarray}
where $\Sym^{i-2}\F_{p^2}^2$ is the $(i-2)$-th symmetric power of the natural representation of $\GL_2(\F_p)$ on the canonical basis of $\F_{p^2}^2$. 

Similarly, we can identify $H^0(\overbar{U_{s'_0}},\Omega^1_{\overbar{U_{s'_0}}})^{(\chi')^p}$ with $(\Sym^{p-1-i}\F_{p^2}^2)\otimes \det {}^{i+j}$. Then we can rewrite the exact sequence in proposition \ref{dmp} as:
\begin{eqnarray*}
0\to  \sigma_{i-2}(j+1) \to H^{(0),\chi,\Q_p}\to \sigma_{p-1-i}(i+j) \to 0.
\end{eqnarray*}

\end{rem}

\begin{rem} \label{opp}
If we assume $p^2-1-m\leq [-mp]$, then we have the exact sequence of the opposite direction:
\begin{eqnarray*}
0 \to \indkg H^0(\overbar{U_{s'_0}},\Omega^1_{\overbar{U_{s'_0}}})^{(\chi')^p}_{\bar{\tau}} \to H^{(0),\chi,\Q_p}\to \indkg H^0(\overbar{U_{s'_0}},\Omega^1_{\overbar{U_{s'_0}}})^{\chi'}_{\tau}\to 0.  
\end{eqnarray*}
\end{rem}

\section{Computation of \texorpdfstring{$M(\chi,[1,b])/p$}{} (\texorpdfstring{\RNum{1}}{}): \texorpdfstring{$2\leq i\leq p-1$}{}} \label{M1}
In this section, we compute $M(\chi,[1,b])/p$ as a representation of $\GL_2(\Q_p)$ when $i\in\{2,\cdots,p-1\}$. The strategy is as follows. We first identify the crystalline cohomology with the de Rham cohomology of some formal scheme. Then $H^{\chi,F_0}$ will map to some meromorphic differential forms on this formal scheme. Now any cohomology class of the de Rham cohomology can be expressed using 
$1$-hypercocycles and any meromorphic differential form can be naturally viewed as a $1$-hypercocycle. The question becomes how to write this $1$-hypercocycle into some 'good form'. This will be done by explicit calculations. We keep the notations in the last section.

Consider the composite of the following maps, which we denote by $\iota$,
\begin{eqnarray*}
H^{\chi',F_0} \to H^0(\Sigma_{1,F},\Omega^1)^{\chi'} \to H^1_{\dR}(\Sigma_{1,F})^{\chi'}\simeq \prod_s H^1_{\dR}(U_{s})^{\chi'}\simeq \prod_s H^1_{\crys}(\overbar{U_s}/F_0)^{\chi'}\otimes_{F_0}F.
\end{eqnarray*}
See section \ref{dR} and section \ref{F_0s} for the notations. Our first result is about the image of $\iota$. We denote the first crystalline cohomology of $\overbar{U_s}$ (over $\Spec\F_{p^2}$) by $H^1_{\crys}(\overbar{U_s}/O_{F_0})$. It is not hard to see that this is a lattice inside $H^1_{\crys}(\overbar{U_s}/F_0)=H^1_{\crys}(\overbar{U_s}/O_{F_0})\otimes_{O_{F_0}}F_0$. 

\begin{prop} \label{ims}
$\iota(H^{\chi',F_0}) \subset \prod_s H^1_{\crys}(\overbar{U_s}/O_{F_0})^{\chi'}\otimes_{O_{F_0}}O_F$.
\end{prop}
\begin{proof}
We only deal with the even case, that is to say, for an even vertex $s'$, we will prove that the image of $H^{\chi',F_0}$ in $H^1_{\crys}(\overbar{U_{s'}}/F_0)^{\chi'}\otimes_{F_0}F$ is actually inside $H^1_{\crys}(\overbar{U_{s'}}/O_{F_0})\otimes_{O_{F_0}}O_F$. The odd case is similar.

First let's recall some results in section \ref{F_0s}. See the discussion below lemma \ref{comp}. We constructed an isomorphism $\psi_{s',\xi}:U_{s',\xi}\to F_{0,\xi}$ (recall that $U_{s',\xi}$ is the tubular neighborhood of $\overbar{U_{s',\xi}}$ in $\Sigma_{1,F}$), where
\[F_{0,\xi}\defeq\{(x,y)\in \A^2_F,y^{p+1}=v_1w_1^{-1}\xi(x^p-x),|x-k|>p^{-1/(p-1)},k=0,\cdots,p-1,|x|<p^{1/(p-1)}\}\] 
Cleary $F_{0,\xi}$ is an open set in a projective curve $D_{0,\xi}$ in $\PP^2_F$ defined by $y^{p+1}=v_1w_1^{-1}\xi(x^p-x)$. The curve $D_{0,\xi}$ has an obvious formal model $\widehat{D_{0,O_{F_0},\xi}}$ over $O_{F_0}$. Its special fibre can be canonically identified with $\overbar{U_{s',\xi}}$. Hence we can identify $H^1_{\crys}(\overbar{U_{s',\xi}}/O_{F_0})$ with $H^1_{\dR}(\widehat{D_{0,O_{F_0},\xi}})$.

\begin{definition} \label{vsx}
For $s\in A(s')$, let $V_{s,\xi}$ be the affine open formal subscheme of $\widehat{D_{0,O_{F_0},\xi}}$ whose underlying space is the union of $\overbar{U_{s',\xi}^0}$ and the intersection point of $\overbar{U_{s',\xi}}$ and $\overbar{U_{s,\xi}}$. Also we define $V_{c,\xi}=\bigcap_{s_v\in A(s')}V_{{s_v},\xi}$ (it is equal to $V_{s_1,\xi}\cap V_{s_2,\xi}$ for any $s_1\neq s_2\in A(s')$).
\end{definition}
Hence $\mathcal{C}=\{V_{s,\xi}\}_{s\in A(s')}$ is an open covering of $\widehat{D_{0,O_{F_0},\xi}}$. Any element in $H^1_{\dR}(\widehat{D_{0,O_{F_0},\xi}})$ can be represented as a $1$-hypercocycle $(\{\omega_s\}_{s\in A(s')},\{f_{s_1,s_2}\}_{s_1,s_2\in A(s)})$, where $\omega_s\in H^0(V_{s,\xi},\Omega^1_{V_{s,\xi}})$, and $f_{s_1,s_2}\in H^0(V_{s_1,\xi}\cap V_{s_2,\xi},\cO_{V_{s_1,\xi}\cap V_{s_2,\xi}})$, such that
\[df_{s_1,s_2}=\omega_{s_1}|_{V_{s_1}\cap V_{s_2}}-\omega_{s_2}|_{V_{s_1}\cap V_{s_2}}.\]
Two $1$-hypercocycles $(\{\omega_s\},\{f_{s_1,s_2}\}),(\{\omega'_s\},\{f'_{s_1,s_2}\})$ represent the same cohomology class if and only there exists a family of functions $\{g_s\}_{s\in A(s)},g_s\in H^0(V_{s,\xi},\cO_{V_{s,\xi}})$, such that
\[\omega_s-\omega'_s=dg_s,~f_{s_1,s_2}-f'_{s_1,s_2}=g_{s_1}|_{V_{s_1}\cap V_{s_2}}-g_{s_2}|_{V_{s_1}\cap V_{s_2}}.\]

Given a differential form $\omega$ on $F_{0,\xi}$, we view it as a cohomology class in $H^1_{\dR}(F_{0,\xi})$. How to relate it with a $1$-hypercocycle in $H^1_{\dR}(D_{0,\xi})=H^1_{\dR}(\widehat{D_{0,O_{F_0},\xi}})\otimes_{F_0}F$ as above? 
\begin{definition}\label{defoWZ}
Since the generic fibre of $\widehat{D_{0,O_{F_0},\xi}}$ becomes $D_{0,\xi}$ when tensored with $F$, the generic fibre of $V_{s,\xi}$ corresponds to an open rigid subspace of $D_{0,\xi}$, which we denote by $W_{s,\xi}$. We also define $Z_{s,\xi}=W_{s,\xi}\cap F_{0,\xi}$.
\end{definition}
If $\omega$ is in the $\chi'$-isotypic component and $\chi'\neq \chi'^p$, we will see later that we can find a rigid analytic function $f_s$ on $Z_{s,\xi}$ for each $s\in A(s')$ such that $\omega|_{Z_{s,\xi}}-df_s$ can be extended to a holomorphic differential form $\omega_s$ on $W_{s,\xi}$. Define 
\begin{eqnarray} \label{fs12}
f_{s_1,s_2}=f_{s_2}|_{W_{s_1,\xi}\cap W_{s_2,\xi}}-f_{s_1}|_{W_{s_1,\xi}\cap W_{s_2,\xi}}.
\end{eqnarray} 
Then $(\{\omega_s\},\{f_{s_1,s_2}\})$ is an element in $H^1_{\dR}(D_{0,\xi})$, whose image in $H^1_{\dR}(F_{0,\xi})$ is $\omega$.

Roughly speaking, what we did above is `removing' the poles of $\omega$ so that $\omega$ can be extended to a hypercocycle on $D_{0,\xi}$.

Now apply the above abstract discussion to our situation. Let $s'$ be an even vertex and $s\in A(s')$. Under the isomorphism in lemma \ref{sms}, then
\begin{eqnarray*}
W_{s,\xi}&=&\{(x,y)\in D_{0,\xi}, |x-k|=1,k=1,\cdots,p-1,|x|\leq 1\}\\
Z_{s,\xi}&=&\{(x,y)\in W_{s,\xi},|x|>p^{-1/(p-1)}\}.
\end{eqnarray*} 
Recall that in lemma \ref{ld0}, we showed that a section $\omega$ of $H^0(\wtso,\omega^1)^{\chi'}$ has the following form when restricted to $\wtxsp$
\begin{eqnarray} \label{rrop}
\omega\big|_{\wtxsp}=\varpi^{p^2-1-m}f(\eta)\e^{p+1-i}\frac{d\e}{\e}+\varpi^{[-mp]}g(\zeta)\e'^{i}\frac{d\e'}{\e'},
\end{eqnarray}
where $f(\eta)\in O_{F_0}[\eta,1/(\eta^{p-1}-1)]~\widehat{},~g(\zeta)\in O_{F_0}[\zeta,1/(\zeta^{p-1}-1)]~\widehat{}$.

Hence if we restrict it on $\wtxp$ (replace $\zeta$ by $\frac{p}{\eta}$ and $\e'$ by $\frac{\varpi^{p-1}\xi}{\e}$):
\begin{eqnarray} \label{rop}
\omega|_{\wtxp}=\varpi^{p^2-1-m}f(\eta)\e^{p+1-i}\frac{d\e}{\e}-\varpi^{[-mp]}g(\frac{p}{\eta})\varpi^{i(p-1)}\xi^i\e^{-i}\frac{d\e}{\e},
\end{eqnarray}
where $f(\eta)\in O_{F_0}[\eta,\frac{1}{\eta^{p-1}-1}]~\widehat{},~g(\frac{p}{\eta})\in O_{F_0}[\frac{p}{\eta},\frac{1}{(\frac{p}{\eta})^{p-1}-1}]~\widehat{}~\subset O_{F_0}[[\frac{p}{\eta}]]$. Notice that the restriction of $\psi_{0,\xi}$ on the generic fibre of $\wtxp$ has the form:
\begin{eqnarray} \label{cov}
x\mapsto \eta,~y\mapsto \e(1-(p/\eta)^{p-1})^{1/(p+1)}
\end{eqnarray}

\begin{lem} \label{tchl1}
Under the isomorphism $\psi_{0,\xi}$, the $1$-form $\omega$ has the following form on $Z_{s,\xi}$
\begin{eqnarray} \label{fooe}
\varpi^{p^2-1-m}(F(x)y^{p+1-i}+G(\frac{p}{x})y^{-i})\frac{dy}{y},
\end{eqnarray}
where $F(x)\in O_{F_0}[x,\frac{1}{x^{p-1}-1}]~\widehat{},~G(\frac{p}{x})=\sum_{n=0}^{+\infty}a_n(\frac{p}{x})^n,a_n\in O_{F_0},\forall n$. Moreover, using \eqref{rop}, 
\begin{enumerate}
\item $f(x)\equiv F(x)~\modd~pO_{F_0}[x,\frac{1}{x^{p-1}-1}]~\widehat{}$. 
\item  $a_0\equiv -\xi^ig(0)~\modd~pO_{F_0}$ if $p^2-1-m\geq [-mp]$ and $a_0\equiv 0~\modd~pO_{F_0}$ otherwise. When $i=p$, $\frac{a_0}{p}\equiv -\xi^ig(0)~\modd~pO_{F_0}$.
\end{enumerate}
\end{lem}

Assume this lemma for the moment. Hence we can write $\omega$ on $Z_{s,\xi}$ as
\[\varpi^{p^2-1-m}(F(x)y^{p+1-i}+G(\frac{p}{x})y^{-i})\frac{dy}{y},\]
where $F(x)\in O_{F_0}[x,\frac{1}{x^{p-1}-1}]~\widehat{},~G(\frac{p}{x})=\sum_{n=0}^{+\infty}a_n(\frac{p}{x})^n\in O_{F_0}[[\frac{p}{x}]]$. Certainly $F(x)y^{p+1-i}\frac{dy}{y}$ extends to $W_{s,\xi}$, so we only need to `remove' the poles of the other term (essentially the pole at $x=y=0$). On $Z_{s,\xi}^0\defeq\{(x,y)\in Z_{s,\xi},|x|<1\}$, we can write:
\[x=\sum_{n=1}^{+\infty}c_ny^{(p+1)n},\]
where $c_n\in O_{F_0},c_1=v_1^{-1}w_1\xi^{-1}\in O_{F_0}^{\times}$. Thus a simple computation shows that on $Z_{s,\xi}^0$,
\begin{lem}  \label{abp}
\[\sum_{n=0}^{+\infty}a_n(\frac{p}{x})^n y^{-i}\frac{dy}{y}=\sum_{n=-\infty}^{+\infty}b_n y^{-n(p+1)-i-1}dy,\]
where $b_n\in O_{F_0},\forall n\in\Z$ and for $n\geq 0,v_p(b_n)\geq n$. Moreover $b_0\equiv a_0~\modd~p$. 
\end{lem}
Now define
\begin{eqnarray} \label{f_s}
f_s=\varpi^{p^2-1-m}\sum_{n=0}^{+\infty} \frac{b_n}{-n(p+1)-i}y^{-n(p+1)-i}.
\end{eqnarray}
It can be viewed as a rigid analytic function on $Z_{s,\xi}$. Also, it is clear from the above computation that $\omega-df_s$ can be extended to a holomorphic differential form $\omega_s$ on $W_{s,\xi}$. Do the same thing for each $s\in A(s')$, we can define $\omega_s,f_{s_1,s_2}$ as we explained before. $(\{\omega_s\},\{f_{s_1,s_2}\})$ is the $1$-hypercocycle in $H^1_{\dR}(D_{0,\xi})\simeq H^1_{\dR}(\widehat{D_{0,O_{F_0},\xi}})\otimes_{O_{F_0}}F$ that represents $\omega$.

Notice that for $i\in\{1,\cdots,p-1\}$, $v_p(\frac{b_n}{-n(p+1)-i})\geq 0$ since $v_p(b_n)\geq n$. When $i=p$, $b_0\equiv a_0\equiv 0~\modd~p$ since we are in the case $p^2-1-m\leq [-mp]$. We still have $v_p(\frac{b_n}{-n(p+1)-i})\geq 0$. In fact, equality only can happen when $n=0$. Therefore all the coefficients appeared in $\omega_s,f_{s_1,s_2}$ will be integral. In other words,
\[(\{\omega_s\},\{f_{s_1,s_2}\})\in H^1_{\dR}(\widehat{D_{0,O_{F_0},\xi}})\otimes_{O_{F_0}}O_F.\]
\end{proof}

\begin{proof}[Proof of lemma \ref{tchl1}]
We only give a sketch of computations here. Using the notations in \eqref{rop}, it suffices to deal with the case $g(\frac{p}{\eta})=0$ and $f(\eta)=0$ separately.
\begin{enumerate}
\item Assume $g(\frac{p}{\eta})=0$. Plug \eqref{cov} into \eqref{rop}. A direct computation shows that $\omega$ has the following form:
\[\varpi^{p^2-1-m}f(x)(1+(\frac{p}{x})^{p-1}G_1(x))y^{p+1-i}\frac{dy}{y},\]
where $G_1(x)\in O_{F_0}[x,\frac{1}{x^{p-1}-1}]~\widehat{}~[[\frac{p}{x}]]$. Let $G_2(x)\in O_{F_0}[x,\frac{1}{x^{p-1}-1}]~\widehat{}~[[\frac{p}{x}]]$ be
\[G_2(x)=v_1w_1^{-1}\xi(x^{p-1}-1)G_1(x)f(x)(\frac{p}{x})^{p-2}.\]
Clearly we can decompose $G_2(x)$ as:
\[G_2(x)=F_3(x)+G_3(\frac{p}{x}),\]
where $F_3(x)\in O_{F_0}[x,\frac{1}{x^{p-1}-1}]~\widehat{}~,G_3(\frac{p}{x})\in O_{F_0}[[\frac{p}{x}]]$. Replacing $F_3(x)$ by $F_3(x)-F_3(0)$, we may assume $F_3(x)\in xO_{F_0}[x,\frac{1}{x^{p-1}-1}]~\widehat{}~$. Since there is a $(\frac{p}{x})^{p-2}$ in the definition of $G_2(x)$, it is easy to see (for example, expand $G_2(x)$ as an element in $F_0[[x,\frac{1}{x}]]$) that the constant term of $G_3(\frac{p}{x})$ is divisible by $p$ (in fact $p^{p-2}$). Recall that we assume $p$ is odd hence at least $3$.

Now $\varpi^{-(p^2-1-m)}\omega$ can be written as
\[f(x)y^{p+1-i}\frac{dy}{y}+p(\frac{F_3(x)}{x})(\frac{1}{(x^{p-1}-1)v_1w_1^{-1}\xi})y^{p+1-i}\frac{dy}{y}+p\frac{G_3(\frac{p}{x})}{(x^p-x)v_1w_1^{-1}\xi}y^{p+1-i}\frac{dy}{y}.\]
Notice that $y^{p+1}=v_1w_1^{-1}\xi(x^p-x)$. The last term is nothing but 
\[pG_3(\frac{p}{x})y^{-i}\frac{dy}{y}.\]
Now let 
\[F(x)=f(x)+p(\frac{F_3(x)}{x})(\frac{1}{(x^{p-1}-1)v_1w_1^{-1}\xi}),~G(\frac{p}{x})=pG_3(\frac{p}{x}).\]
It is clear they satisfy all the conditions in the lemma. So we're done in this case.
\item Assume $f(\eta)=0$. When $p^2-1-m\geq [-mp]$, we can write $\varpi^{-(p^2-1-m)}\omega$ as
\[-g(\frac{p}{x})\xi^i(1+(\frac{p}{x})^{p-1}H(x))y^{-i}\frac{dy}{y},\]
where $H(x)\in O_{F_0}[x,\frac{1}{x^{p-1}-1}]~\widehat{}~[[\frac{p}{x}]]$. Decompose $-g(\frac{p}{x})\xi^i H(x)(\frac{p}{x})^{p-2}$ into
\[-g(\frac{p}{x})\xi^i H(x)(\frac{p}{x})^{p-2}=F_1(x)x^2+Ax+H_1(\frac{p}{x}),\]
where $F_1(x)\in O_{F_0}[x,\frac{1}{x^{p-1}-1}]~\widehat{}~, A\in pO_{F_0},~H_1(\frac{p}{x})\in O_{F_0}[[\frac{p}{x}]]$. Notice that $A$ is divisible by $p$ since there is a $(\frac{p}{x})^{p-2}$ in the expression. Then $\varpi^{-(p^2-1-m)}\omega$ is
\[-g(\frac{p}{x})\xi^i y^{-i}\frac{dy}{y}+ pxF_1(x)y^{-i}\frac{dy}{y}+ pA y^{-i}\frac{dy}{y}+\frac{p}{x}H_1(\frac{p}{x})y^{-i}\frac{dy}{y}\]
Using $y^{p+1}=v_1w_1^{-1}\xi(x^p-x)$, the second term is $pF_1(x)\frac{1}{v_1w_1^{-1}\xi(x^{p-1}-1)}y^{p+1-i}\frac{dy}{y}$. It's easy to see the following $F(x),G(\frac{p}{x})$ actually work.
\[F(x)=pF_1(x)\frac{1}{v_1w_1^{-1}\xi(x^{p-1}-1)},~G(\frac{p}{x})=-g(\frac{p}{x})\xi^i+pA+\frac{p}{x}H_1(\frac{p}{x}).\]
When $p^2-1-m\geq [-mp]$ does not hold, then 
\[\omega=-\varpi^{[-mp]}g(\frac{p}{\eta})\varpi^{i(p-1)}\xi^i\e^{-i}\frac{d\e}{\e}=p\varpi^{p^2-1-m}g(\frac{p}{\eta})\xi\e^{-i}\frac{d\e}{\e}.\]
Repeat the previous argument and it's direct to see the claim in the lemma is true.
\end{enumerate}
\end{proof}

In the previous proposition, we showed how to turn a differential form $\omega\in H^{\chi',F_0}$, when restricted to $U_{s'}$, into a $1$-hypercocycle $(\{\omega_s\},\{f_{s_1,s_2}\})$ inside the de Rham cohomology $H^1_{\dR}(\widehat{D_{0,O_{F_0},\xi}})\otimes_{O_{F_0}}O_F$ (via the isomorphism $\psi_{s',\xi}$). It is crucial to understand the mod $p$ properties of this hypercocycle. Essentially, we need to understand $f_s$ in \eqref{f_s} modulo $p$ (recall that $f_{s_1,s_2}=f_{s_2}-f_{s_1}$, see \eqref{fs12}). 

Fix an even vertex $s'$ and $s\in A(s')$. It is clear from our definition that $\varpi^{-(p^2-1-m)}f_s\in H^0(V_{c,\xi},\cO_{V_{c,\xi}})$. Recall that $V_{c,\xi}=\bigcap_{s_v\in A(s')}V_{{s_v},\xi}$.

\begin{lem} \label{b_0}
Using the notations in the proof of proposition \ref{ims}. 
\begin{enumerate}
\item When $i=p$,
\[\varpi^{-(p^2-1-m)}f_s\equiv \frac{b_0y^{-p}}{-p}\equiv \frac{a_0y^{-p}}{-p}\equiv \xi^p g(0)y^{-p}~\modd~pH^0(V_{c,\xi},\cO_{V_{c,\xi}}).\]
\item When $i\in\{1,\cdots,p-1\}$ and $p^2-1-m\leq [-mp]$,
\[\varpi^{-(p^2-1-m)}f_s\in pH^0(V_{c,\xi},\cO_{V_{c,\xi}}).\]
\item When $p^2-1-m\geq [-mp]$, we have 
\[\varpi^{-(p^2-1-m)}f_s\equiv \frac{b_0y^{-i}}{-i}\equiv \frac{a_0y^{-i}}{-i}\equiv \frac{\xi^i g(0)y^{-i}}{i}~\modd~pH^0(V_{c,\xi},\cO_{V_{c,\xi}})\]
except the case $i=p-1$ and the case $p=3,i=1$. I claim in these exceptional cases, we can find another $1$-hypercocycle $(\{\omega'_{s_v}\},\{f'_{s_1,s_2}\})$ in the same cohomology class of $(\{\omega_{s_v}\},\{f_{s_1,s_2}\})$ such that we can write $f'_{s_1,s_2}=f'_{s_2}-f'_{s_1}$ for any $s_1,s_2\in A(s')$ and
\begin{eqnarray*}
\varpi^{-(p^2-1-m)}f'_s&\equiv& \frac{b_0y^{-i}}{-i}~\modd~pH^0(V_{c,\xi},\cO_{V_{c,\xi}}),\\
f'_{s_v}&=&f_{s}~\mbox{for any }s_v\neq s
\end{eqnarray*}
\end{enumerate}
\end{lem}

\begin{proof}[Proof of lemma \ref{b_0}]
Everything is clear by lemma \ref{abp} and the definition of $f_s$ in \eqref{f_s} except for the exceptional cases. First we assume $i=p-1$, then
\[\varpi^{-(p^2-1-m)}f_s\equiv \frac{b_0y^{-(p-1)}}{-(p-1)}+\frac{b_1y^{-2p}}{-2p}~\modd~pH^0(V_{c,\xi},\cO_{V_{c,\xi}}).\]
This makes sense since $v_p(b_1)\geq 1$. Now define $g_{s_v}\in H^0(V_{s_v,\xi},\cO_{V_{s_v,\xi}}),s_v\in A(s')$ as:
\[g_{s_v}=\left\{
\begin{array}{ll}
-\frac{b_1}{-2pv_1^2w_1^{-2}\xi^2}\frac{(x^{2p-4}-2x^{p-3})y^2}{(x^{p-1}-1)^2}& \mbox{if }s_v=s\\
-\frac{b_1}{-2pv_1^2w_1^{-2}\xi^2}\frac{y^2}{x^2} &\mbox{if } s_v\neq s
\end{array}\right.\]

Hence define $\omega'_{s_v}=\omega_{s_v}+\varpi^{p^2-1-m}dg_{s_v},f'_{s_1,s_2}=f_{s_1,s_2}+\varpi^{p^2-1-m}(g_{s_1}-g_{s_2})$, the hypercocycle $(\{\omega_{s_v}\},\{f_{s_1,s_2}\})$ and $(\{\omega'_{s_v}\},\{f'_{s_1,s_2}\})$ are in the same cohomology class. A simple computation shows the following identity in $H^0(V_{c,\xi},\cO_{V_{c,\xi}})$,
\[-\frac{b_1}{-2pv_1^2w_1^{-2}\xi^2}\frac{(x^{2p-4}-2x^{p-3})y^2}{(x^{p-1}-1)^2}+\frac{b_1}{-2pv_1^2w_1^{-2}\xi^2}\frac{y^2}{x^2}=\frac{b_1y^{-2p}}{-2p}\]
Thus if we define $f'_s=f_s-\frac{b_1y^{-2p}}{-2p},f'_{s_v}=f_{s_v},s_v\neq s$, they satisfy
\[f'_{s_1,s_2}=f'_{s_2}-f'_{s_1},\]
and clearly have the property we want.

The case $i=1,p=3$ can be done by the same method. This time 
\[\varpi^{-(p^2-1-m)}f_s\equiv b_0y^{-1}+\frac{b_2}{-9}y^{-9}\equiv b_0y^{-1}+\frac{b_2}{-9v_1^3w_1^{-3}\xi^3}(\frac{x^3y^3}{(x^2-1)^3}-\frac{y^3}{x^3})~\modd~3H^0(V_{c,\xi},\cO_{V_{c,\xi}}).\]
We can define $g_{s_v}$ similarly. I omit the details here.
\end{proof}

\begin{rem} \label{odp}
In the odd case, things are similar. We only restrict ourselves to the case $p^2-1-m\geq[-mp]$. Let $s$ be an odd vertex. We also have $\psi_{s,\xi}$ (see the beginning of section \ref{F_0s}). Let $\omega$ be an element of $H^{\chi',F_0}$. Similarly to lemma \ref{tchl1}, $\omega$ has the form (using \eqref{resp}):
\begin{eqnarray} \label{fooo}
\varpi^{[-mp]}(F(x)y^i+pG(\frac{p}{x})y^{-p-1+i})\frac{dy}{y},
\end{eqnarray}
where $F(x)\in O_{F_0}[x,\frac{1}{x^{p-1}-1}]~\widehat{},G(\frac{p}{x})=\sum_{n=0}^{+\infty}a_n(\frac{p}{x})^n,a_n\in O_{F_0},\forall n$. All the above arguments work here and we can define a $1$-hypercocycle $(\{\omega_{s'}\},\{f_{s'_1,s'_2}\})$ that represents $\omega$. Notice that there is a `$p$' in front of $G(\frac{p}{x})$ in the formula \eqref{fooo}. Thus when $2\leq i\leq p$ (resp. $i=1$), 
\begin{eqnarray*}
&\varpi^{-[-mp]}f_{s'}\in pH^0(V_{c,\xi},\cO_{V_{c,\xi}})~(\mbox{resp. } H^0(V_{c,\xi},\cO_{V_{c,\xi}}))\\
&\varpi^{-[-mp]}f_{s'_1,s'_2}\in pH^0(V_{s'_1,\xi}\cap V_{s'_2,\xi},\cO_{V_{s'_1,\xi}\cap V_{s'_2,\xi}}) ~(\mbox{resp. } H^0(V_{s'_1,\xi}\cap V_{s'_2,\xi},\cO_{V_{s'_1,\xi}\cap V_{s'_2,\xi}})).
\end{eqnarray*}
Here $V_{s',\xi},s'\in A(s)$ is defined similarly.
\end{rem}

Before stating the main result of this section, we still need to do some extra work. Most results here can be found in \cite{HJ}. Since $\widehat{D_{0,O_{F_0},\xi}}$ is a curve in $\PP^2$, the Hodge-de Rham spectral sequence gives us the following exact sequence:
\begin{eqnarray} \label{hfil}
0\to H^0(\widehat{D_{0,O_{F_0},\xi}},\Omega^1_{\widehat{D_{0,O_{F_0},\xi}}})\to H^1_{\dR}(\widehat{D_{0,O_{F_0},\xi}})\to H^1(\widehat{D_{0,O_{F_0},\xi}},\cO_{\widehat{D_{0,O_{F_0},\xi}}})\to 0.
\end{eqnarray}
And each group in this exact sequence is a finite free $O_{F_0}$-module. If we use a $1$-hypercocycle $(\{\omega_s\},\{f_{s_1,s_2}\})$ to represent a cohomology class in $H^1_{\dR}(\widehat{D_{0,O_{F_0},\xi}})$. Then every element in $H^0(\widehat{D_{0,O_{F_0},\xi}},\Omega^1_{\widehat{D_{0,O_{F_0},\xi}}})$ can be identified as the hypercocycle with all $f_{s_1,s_2}=0$. And the map to $H^1(\widehat{D_{0,O_{F_0},\xi}},\cO_{\widehat{D_{0,O_{F_0},\xi}}})$ is just mapping the hypercocycle to $\{f_{s_1,s_2}\}$, which is considered as a $1$-cocycle. Similarly, we have 
\[0\to H^0(\overbar{U_{s'_0,\xi}},\Omega^1_{\overbar{U_{s'_0,\xi}}})\to H^1_{\dR}(\overbar{U_{s'_0,\xi}})\to H^1(\overbar{U_{s'_0,\xi}},\cO_{\overbar{U_{s'_0,\xi}}})\to 0,\]
which can be identified with the reduction $\modd~p$ of the previous exact sequence.

Recall that the de Rham cohomology of $\widehat{D_{0,O_{F_0},\xi}}$ can be identified as the crystalline cohomology of $\overbar{U_{s'_0,\xi}}$, it is equipped with a Frobenius operator $\varphi$. It is important to understand the relationship between $\varphi$ and the above exact sequence. Denote $\bigcup_{\xi^{p-1}=-1} \widehat{D_{0,O_{F_0},\xi}}$ by $\widehat{D_{0,O_{F_0}}}$.

\begin{lem}\label{phi1}
Under the isomorphism between $H^1_{\dR}(\widehat{D_{0,O_{F_0},\xi}})$ and $H^1_{\crys}(\overbar{U_{s'_0,\xi}}/O_{F_0})$,
\begin{enumerate}
\item $\varphi(H^1_{\dR}(\widehat{D_{0,O_{F_0}}})^{\chi'})\subset H^1_{\dR}(\widehat{D_{0,O_{F_0}}})^{(\chi')^p}$.
\item $\varphi(H^1_{\dR}(\widehat{D_{0,O_{F_0}}})^{\chi'})\subset H^0(\widehat{D_{0,O_{F_0}}},\Omega^1_{\widehat{D_{0,O_{F_0}}}})^{(\chi')^p} +pH^1_{\dR}(\widehat{D_{0,O_{F_0}}})^{(\chi')^p}$.
\item The above inclusion is in fact an equality and $\varphi$ induces an isomorphism between $H^1(\overbar{U_{s'_0}},\cO_{\overbar{U_{s'_0}}})^{\chi'}$ and $H^0(\overbar{U_{s'_0}},\Omega^1_{\overbar{U_{s'_0}}})^{(\chi')^p}$.
\end{enumerate}
\end{lem}

\begin{proof}
See section 3 of \cite{HJ}, especially proposition 3.5. Although our curve is slightly different from the curve in that paper, all arguments in their paper work here.
\end{proof}

\begin{rem} \label{vpi}
A variant of lemma \ref{phi1} is that $\varphi$ induces an isomorphism 
\[H^0(\widehat{D_{0,O_{F_0}}},\Omega^1_{\widehat{D_{0,O_{F_0}}}})^{\chi'}+pH^1_{\dR}(\widehat{D_{0,O_{F_0}}})^{\chi'}\stackrel{\sim}{\longrightarrow}pH^1_{\dR}(\widehat{D_{0,O_{F_0}}})^{(\chi')^p}.\] 
This follows from the fact that $\varphi^2$ is a scalar $c_x$ on these spaces and $v_p(c_x)=1$. See proposition \ref{dcrys}. A direct corollary is that $\varphi$ induces an isomorphism between 
\[(H^0(\widehat{D_{0,O_{F_0}}},\Omega^1_{\widehat{D_{0,O_{F_0}}}})^{\chi'}+pH^1_{\dR}(\widehat{D_{0,O_{F_0}}})^{\chi'})/pH^1_{\dR}(\widehat{D_{0,O_{F_0}}})^{\chi'}\] 
and 
\[pH^1_{\dR}(\widehat{D_{0,O_{F_0}}})^{(\chi')^p}/(pH^0(\widehat{D_{0,O_{F_0}}},\Omega^1_{\widehat{D_{0,O_{F_0}}}})^{(\chi')^p}+p^2H^1_{\dR}(\widehat{D_{0,O_{F_0}}})^{(\chi')^p}),\] 
which can be viewed as an isomorphism between $H^0(\overbar{U_{s'_0}},\Omega^1_{\overbar{U_{s'_0}}})^{\chi'}$ and $H^1(\overbar{U_{s'_0}},\cO_{\overbar{U_{s'_0}}})^{(\chi')^p}$.
\end{rem}

In fact, we can write down the isomorphism between $H^1(\overbar{U_{s'_0}},\cO_{\overbar{U_{s'_0}}})^{\chi'}$ and $H^0(\overbar{U_{s'_0}},\Omega^1_{\overbar{U_{s'_0}}})^{(\chi')^p}$ explicitly (lemma \ref{ef} below). Some notations here. As before (see \eqref{les'0}), we may identify $\overbar{U_{s'_0,\xi}}$ with the projective curve defined by $\e^{p+1}=v_1w_1^{-1}\xi(\eta^p-\eta)$ and the singular points of $\overbar{U_{s'_0}}$ (considered in the special fibre of $\wtx$) are those points with $\e=0$.
\begin{definition}\label{defoAs'0}
We write $A(s'_0)=\{s_0,\cdots,s_{p-1},s_{\infty}\}$, where for $k=0,\cdots,p-1$, $s_k$ is the vertex that corresponds to $\eta=k,\e=0$ in $\overbar{U_{s'_0,\xi}}$ and $s_\infty$ corresponds to the point $\eta=\infty,\e=0$ (equivalently, if we use projective coordinates $[\eta,\e,1]$, then this point is $[1,0,0]$).
\end{definition}
\begin{definition}\label{V0i}
Let $V_0$ be the open set of $\overbar{U_{s'_0,\xi}}$ that is the complement of the point $\eta=\infty,\e=0$. We also define $V_{\infty}$ as the complement of $\eta=\e=0$.
\end{definition}

Using the notations in definition \ref{vsx}, it is clear that set theoretically, $V_0$ is the union of $V_{s_0,\xi},\cdots,V_{s_{p-1},\xi}$ and $V_{\infty}$ is the union of $V_{s_1,\xi},\cdots,V_{s_{p-1},\xi},V_{s_{\infty},\xi}$.  By abuse of notations, we also view $V_0,V_\infty$ as open affine formal subschemes of $\widehat{D_{0,O_{F_0},\xi}}$.

Notice that $V_0,V_{\infty}$ is an open covering of $\widehat{D_{0,O_{F_0},\xi}}$. Hence every cohomology class of $H^1_{\dR}(\widehat{D_{0,O_{F_0},\xi}})$ can be represented by a $1$-hypercocycle $(\omega_0,\omega_{\infty},f_{0,\infty})$ as before. Every element of $H^1(\overbar{U_{s'_0},\xi},\cO_{\overbar{U_{s'_0},\xi}})$ can be represented by an element in $H^0(V_0\cap V_{\infty},\cO_{V_0\cap V_{\infty}})$, viewed as a $1$-cocycle. It is easy to see that
\begin{lem} \label{st1o}
$H^1(\overbar{U_{s'_0}},\cO_{\overbar{U_{s'_0}}})^{\chi'}$ has a basis (when restricted to $\overbar{U_{{s'_0},\xi}}$)
\[ \frac{\e^{p+1-i}}{\eta^k},k=1,\cdots,p-i.\]
If $i=p$, then $H^1(\overbar{U_{s'_0}},\cO_{\overbar{U_{s'_0}}})^{\chi'}=0$.
\end{lem}

Hence we may view $\frac{\e^{p+1-i}}{\eta^k}$ as an element in $H^1(\overbar{U_{s'_0}},\cO_{\overbar{U_{s'_0}}})^{\chi'}$. Then $\varphi(\frac{\e^{p+1-i}}{\eta^k})$, as a $1$-hypercocycle, is $(0,0,\frac{\e^{(p+1-i)p}}{\eta^{pk}})$. A direct computation shows that
\begin{lem}\label{ef}
$\varphi(\frac{\e^{p+1-i}}{\eta^k})$ is the same as the holomorphic differential form 
\[(v_1w_1^{-1}\xi)^{p-i}(-1)^{p-i-k}k\begin{pmatrix} p-i \\ k\end{pmatrix} \eta^{p-i-k}\e^{i-1}d\e.\] 
\end{lem}

\begin{rem} 
At some point, we will need to translate a $1$-cocycle inside $H^1(\overbar{U_{{s'_0},\xi}},\cO_{\overbar{U_{{s'_0},\xi}}})$ using the open covering $\{V_{s,\xi}\}_{s\in A(s'_0)}$ to a $1$-cocycle using the open covering $\{V_0,V_\infty\}$. This is done as follows. If we start with a $1$-cocycle $\{f'_{s,s''}\}$, we can find another $1$-cocycle $\{f_{s,s''}\}$ that represents the same cohomology class and $f_{s_0,s_\infty}$ can be extended to a section in $H^0(V_0\cap V_\infty,\cO_{V_0\cap V_\infty})$. Then $f_{s_0,s_\infty}$ can be viewed as a $1$-cocycle of the covering $\{V_0,V_\infty\}$. In fact, this is just what we want.
\end{rem}
\begin{example} \label{kex}
Let's compute one example here. Consider the following $1$-cocycle $\{f_{s,s''}\}$:
\[f'_{s,s''}=f'_{s''}-f'_{s},~\mbox{where }f'_{s_0}=\e^{-i},f'_s=0 \mbox{ for }s\neq s_0.\]
Then clearly $f'_{s_0,s_\infty}$ has poles on $V_0\cap V_\infty$. But we can modify this cocycle a little bit:
define
\[g_{s_0}=\frac{\eta^{p-2}\e^{p+1-i}}{v_1w_1^{-1}\xi(\eta^{p-1}-1)}\in H^0(V_{s_0,\xi},\cO_{V_{s_0,\xi}}),~g_s=0\mbox{ for }s\neq s_0.\]
And let 
\[f_{s,s''}=f'_{s,s''}-g_{s''}+g_{s}.\]
Then $\{f_{s,s''}\}$ and $\{f'_{s,s''}\}$ represent the same cohomology class. Moreover,
\[f_{s_0,s_\infty}=f'_{s_0,s_\infty}+g_{s_0}=-f_{s_0}+g_{s_0}=-\e^{-i}+\frac{\eta^{p-2}\e^{p+1-i}}{v_1w_1^{-1}\xi(\eta^{p-1}-1)}=\frac{\e^{p+1-i}}{v_1w_1^{-1}\xi\eta}.\]
(using $\e^{p+1}=v_1w_1^{-1}\xi(\eta^p-\eta)$) clearly extends to $V_0\cap V_\infty$. Hence, 
\[\frac{\e^{p+1-i}}{v_1w_1^{-1}\xi\eta},\]
viewed as a $1$-cocycle of the covering $\{V_0,V_\infty\}$, represents the same cohomology class as $\{f'_{s.s''}\}$.
\end{example}

A combination of remark \ref{odp} and the lemma \ref{phi1} gives us the following
\begin{lem} \label{rm3}

Assume $p^2-1-m\geq [-mp]$ and $i\neq 1$. Let $s$ be an odd vertex and $\omega\in H^{\chi',F_0}$. 
\begin{enumerate}
\item Using the method in the proof of proposition \ref{ims}, we may view $\varpi^{-[-mp]}\omega$ as a cohomology class inside $H^1_{\crys}(\overbar{U_s}/O_{F_0})^{\chi'}$. Then 
\[\varphi(\varpi^{-[-mp]}\omega)\in pH^1_{\crys}(\overbar{U_s}/O_{F_0})^{(\chi')^p}.\]
Or equivalently, using remark \ref{vpi},
\begin{eqnarray*}
\varpi^{-[-mp]}\omega\in \varphi(H^1_{\crys}(\overbar{U_s}/O_{F_0})^{(\chi')^p}).
\end{eqnarray*}
\item In fact, lemma \ref{klem1} shows that $\varpi^{-[-mp]}\omega$ modulo $p$ is a holomorphic differential form inside 
\[H^0(\overbar{U_s},\Omega^1_{\overbar{U_s}})=\varphi(H^1_{\dR}(\overbar{U_s})),\]
which is nothing but $\varpi^{-[-mp]}\omega$ considered as a cohomology class in $H^1_{\dR}(\overbar{U_s})$. In particular, if 
\[\omega\big|_{\wts}\in pH^0(\wts,\omega^1),\]
then the cohomology class of $\varpi^{-[-mp]}\omega$ is inside $pH^1_{\crys}(\overbar{U_s}/O_{F_0})$.
\end{enumerate}
\end{lem}

\begin{proof}
See remark \ref{odp} for the notations here. Let $(\{\omega_{s'}\},\{f_{s'_1,s'_2}\})$ be the $1$-hypercocycle defined there that represents $\omega$ in $F\otimes_{F_0}H^1_{\crys}(\overbar{U_s}/F_0)^{\chi'}$ (by identifying the crystalline cohomology of $\overbar{U_s}$ with the de Rham cohomology of $\widehat{D_{0,O_{F_0}}}$). As we noted in remark \ref{odp}, 
\[\varpi^{-[-mp]}f_{s'_1,s'_2}\in pH^0(V_{s'_1,\xi}\cap V_{s'_2,\xi},\cO_{V_{s'_1,\xi}\cap V_{s'_2,\xi}}) .\]
Hence if we reduce modulo $p$, all these $\varpi^{-[-mp]}f_{s'_1,s'_2}$ vanish. This means that the image of $\varpi^{-[-mp]}\omega$ in $H^1_{\dR}(\overbar{U_{s}})$ actually lies inside $H^0(\overbar{U_s},\Omega^1_{\overbar{U_s}})$. Now our first claim is a direct consequence of lemma \ref{phi1}. The rest of the lemma follows from (see remark \ref{odp})
\[\varpi^{-[-mp]}f_{s'}\in pH^0(V_{c,\xi},\cO_{V_{c,\xi}}).\]
Thus when we restricted everything on the special fibre of $V_{c,\xi}$ (equivalently, $\overbar{U_{s,\xi}^0}$), 
\[\varpi^{[-mp]}\omega_{s'}=\varpi^{-[-mp]}\omega-df_{s'}\equiv\varpi^{-[-mp]}\omega \mod~pH^0(V_{c,\xi},\Omega^1_{V_{c,\xi}}).\]
This indicates that the cohomology class of $\varpi^{-[-mp]}\omega$ is just the $1$-form $\varpi^{-[-mp]}\omega$ after reducing modulo $p$.

\end{proof}
\begin{rem}\label{vorm3}
Using the action of $\GL_2(\Q_p)$, it's not hard to see that if we replace $s$ by an even vertex $s'$ and $\omega\in H^{\chi',F_0}$ by $\omega\in H^{(\chi')^p,F_0}$, we have a similar result:
\[\varpi^{-[-mp]}\omega\in \varphi(H^1_{\crys}(\overbar{U_{s'}}/O_{F_0})^{\chi'}),\]
and exactly the same statement for the second part.
\end{rem}

Similarly we may obtain the following lemma by combining lemma \ref{b_0} and lemma \ref{phi1}.
\begin{lem} \label{rm4}
Let $\omega\in H^0(\wtso,\omega^1)^{\chi',\Gal(F/F_0)}$ and $s'$ be an even vertex. Assume 
\begin{eqnarray}\label{arm4}
\omega\big|_{\wts}\in pH^0(\wts,\omega^1)^{\chi',\Gal(F/F_0)}
\end{eqnarray}
for any $s\in A(s')$. 
\begin{enumerate}
\item Then the image of $\varpi^{-(p^2-1-m)}\omega$ in $H^1_{\crys}(\overbar{U_{s'}}/O_{F_0})^{\chi'}$ is actually inside
\[\varphi(H^1_{\crys}(\overbar{U_s}/O_{F_0})^{(\chi')^p}).\]
Or equivalently, if we view $\varpi^{-(p^2-1-m)}\omega$ as an element inside $H^1_{\crys}(\overbar{U_{s'}}/O_{F_0})^{\chi'}$,
\[\varphi(\varpi^{-(p^2-1-m)}\omega)\in pH^1_{\crys}(\overbar{U_s}/O_{F_0})^{(\chi')^p}.\]
\item Assume $i\neq p$. Lemma \ref{klem1} shows that under these conditions $\varpi^{-(p^2-1-m)}\omega$ modulo $p$ is a holomorphic differential form inside 
\[H^0(\overbar{U_{s'}},\Omega^1_{\overbar{U_s}})=\varphi(H^1_{\dR}(\overbar{U_{s'}})),\]
and we may identify it with the cohomology class of $\varpi^{-(p^2-1-m)}\omega$ in $H^1_{\dR}(\overbar{U_{s'}})$. In particular, if 
\[\omega\big|_{\wtp}\in pH^0(\wtp,\omega^1),\]
then the cohomology class of $\varpi^{-(p^2-1-m)}\omega$ is inside $pH^1_{\crys}(\overbar{U_{s'}}/O_{F_0})$.
\item Assume $i=p$. We have a slightly weaker result: assume
\begin{eqnarray} \label{ipa}
\omega\big|_{\wtp}\in pH^0(\wtp,\omega^1),
\end{eqnarray}
then the cohomology class of $\varpi^{-(p^2-1-m)}\omega$ is inside $pH^1_{\crys}(\overbar{U_{s'}}/O_{F_0})$.
\end{enumerate}
\end{lem}

\begin{proof}
First we prove the first two parts. If $i=p$, we know that (lemma \ref{st1o})
\[H^1(\overbar{U_{s'}},\cO_{\overbar{U_{s'}}})^{\chi'}=H^0(\overbar{U_{s'}},\Omega^1_{\overbar{U_{s'}}})^{(\chi')^p}=0.\]
Hence lemma \ref{phi1} tells us that
\[\varphi(H^1_{\crys}(\overbar{U_s}/O_{F_0})^{(\chi')^p})=H^1_{\crys}(\overbar{U_{s'}}/O_{F_0})^{\chi'}.\]
So in this case, the first part is trivially true.

Now assume $i\neq p$. We need to use some results in lemma \ref{b_0}. See the notations there. We can represent $\omega\big|_{\wtp}$ as a $1$-hypercocycle $(\{\omega_s\},\{f_{s_1,s_2}\})$ and for $s\in A(s')$, there exists 
\[f_s\in \varpi^{p^2-1-m}H^0(V_{c,\xi},\cO_{V_{c,\xi}}),\] 
such that $f_{s_1,s_2}=f_{s_2}-f_{s_1}$. Hence as in the proof of the previous lemma, it suffices to prove 
\[\varpi^{-(p^2-1-m)}f_s\in pH^0(V_{c,\xi},\cO_{V_{c,\xi}}).\]

If $p^2-1-m\leq [-mp]$ and $i\neq p$, this already follows from the second part of lemma \ref{b_0}. So we only need to treat the case $p^2-1-m\geq[-mp]$. Then the desired result follows from our condition that $\omega\big|_{\wts}\in pH^0(\wts,\omega^1)^{\chi',\Gal(F/F_0)}$. More precisely, using the notations in the proof of proposition \ref{ims} (especially \eqref{rrop}), we can write
\begin{eqnarray} \label{ipsp}
\omega\big|_{\wtxsp}=\varpi^{p^2-1-m}f(\eta)\e^{p+1-i}\frac{d\e}{\e}+\varpi^{[-mp]}g(\zeta)\e'^{i}\frac{d\e'}{\e'}.
\end{eqnarray}
Notice that lemma \ref{b_0} tells us that 
\[\varpi^{-(p^2-1-m)}f_s\equiv \frac{\xi^i g(0)y^{-i}}{i}~\modd~pH^0(V_{c,\xi},\cO_{V_{c,\xi}}).\]
It suffices to show $g(0)\in pO_{F_0}$. But by lemma \ref{klem12}, 
\[g(\zeta)\in pO_{F_0}[\zeta,\frac{1}{\zeta^{p-1}-1}]~\widehat{}\]
since we assume $\omega\big|_{\wts}\in pH^0(\wts,\omega^1)^{\chi',\Gal(F/F_0)}$. So we're done for the first two parts.

As for the last claim, we keep using the notations $(\{\omega_s\},\{f_{s_1,s_2}\})$ and $\{f_s\}$ as above. Notice that we already assume $\omega\big|_{\wtp}\in pH^0(\wtp,\omega^1)$. Hence if we can show 
\begin{eqnarray} \label{ipfs}
\varpi^{-(p^2-1-m)}f_s\in pH^0(V_{c,\xi},\cO_{V_{c,\xi}}),
\end{eqnarray}
then we would know that both $\omega_s$ and $f_{s_1,s_2}$ are divisible by $p$. Thus it suffices to prove \eqref{ipfs}.

But using the notations \eqref{ipsp} above and lemma \ref{b_0}, which says that
\[\varpi^{-(p^2-1-m)}f_s\equiv \xi^p g(0)y^{-p}~\modd~pH^0(V_{c,\xi},\cO_{V_{c,\xi}}),\]
we only need to show $g(0)$ is divisible by $p$. But by our assumption \eqref{ipa}, 
\[f(\eta)\in pO_{F_0}[\eta,\frac{1}{\eta^{p-1}-1}]~\widehat{}~.\]
Since we also assume $\omega\big|_{\wts}\in pH^0(\wts,\omega^1)$, hence
\[g(\zeta)\in pO_{F_0}[\zeta,\frac{1}{\zeta^{p-1}-1}]~\widehat{}~.\]
See the computations around lemma \ref{klem11}. Therefore $g(0)\in pO_{F_0}$.
\end{proof}

\begin{rem}\label{vorm4}
Using the action of $\GL_2(\Q_p)$, we can get a variant of the previous lemma. Let $\omega\in H^0(\wtso,\omega^1)^{(\chi')^p,\Gal(F/F_0)}$ and $s$ be an odd vertex. Assume 
\[\omega|_{\wtp}\in pH^0(\wtp,\omega^1)^{(\chi')^p,\Gal(F/F_0)}\]
for any $s'\in A(s)$. Then the cohomology class of $\varpi^{-(p^2-1-m)}\omega$ in $H^1_{\crys}(\overbar{U_{s'}}/O_{F_0})^{(\chi')^p}$ is inside
\[\varphi(H^1_{\crys}(\overbar{U_s}/O_{F_0})^{\chi'}).\]
And we have a similar result for the last two parts: if we assume 
\[\omega\big|_{\wts}\in pH^0(\wts,\omega^1),\]
then the cohomology class of $\varpi^{-(p^2-1-m)}\omega$ is inside $pH^1_{\crys}(\overbar{U_{s}}/O_{F_0})$.
\end{rem}

\begin{rem}
When $i=p$, we will see later on in section \ref{M2} that the second part of the lemma is actually true (lemma \ref{i1s'0}).
\end{rem}

Now let's recall the construction of $M(\chi,[1,b])$ in section \ref{cb}. First we write 
\[\prod_s (F\otimes_{F_0}H^1_{\crys}(\overbar{U_s}/F_0)\otimes_{\Q_p} E)^{\chi}=F_1\oplus F_2,\]
where
\begin{eqnarray} \label{f1f2}
F_1\defeq \prod_{s':\mbox{even}} F\otimes_{F_0}H^1_{\crys}(\overbar{U_{s'}}/F_0)^{\chi'}_{\tau} \oplus \prod_{s:\mbox{odd}} F\otimes_{F_0}H^1_{\crys}(\overbar{U_{s}}/F_0)^{(\chi')^p}_{\bar{\tau}} \\
F_2\defeq \prod_{s':\mbox{even}} F\otimes_{F_0}H^1_{\crys}(\overbar{U_{s'}}/F_0)^{(\chi')^p}_{\bar{\tau}} \oplus \prod_{s:\mbox{odd}} F\otimes_{F_0}H^1_{\crys}(\overbar{U_{s}}/F_0)^{\chi'}_{\tau}  
\end{eqnarray}
It is clear from the previous lemma that $g_{\varphi}\otimes\varphi\otimes \Id_E$ sends $F_1$ to $F_2$. Let $f$ be an element of $(H^0(\wtso,\omega^1)\otimes O_E)^{\chi,\Gal(F/F_0)}$. By proposition \ref{inj}, we have an injective map $(H^0(\wtso,\omega^1)\otimes O_E)^{\chi,\Gal(F/F_0)}$ into $\prod_s (F\otimes_{F_0}H^1_{\crys}(\overbar{U_s}/F_0)\otimes_{\Q_p} E)^{\chi}$. Let $(f_1,f_2)$ be the decomposition of the image of $f$ into $F_1\oplus F_2$. Then,\begin{lem}
\[M(\chi,[1,b])=\{f\in (H^0(\wtso,\omega^1)\otimes O_E)^{\chi,\Gal(F/F_0)}, (1\otimes b)(g_{\varphi}\otimes\varphi\otimes\Id_E)(f_1)=(\varpi^{(p-1)i}\otimes 1)f_2\}.\]
Here $1\otimes b$ and $\varpi^{(p-1)i}\otimes 1$ are viewed as elements in $F\otimes_{\Q_p} E$.
\end{lem}
\begin{proof}
Considering $(H^0(\wtso,\omega^1)\otimes O_E)^{\chi,\Gal(F/F_0)}= (H^0(\wtsow,\omega^1)\otimes O_E)^{\chi,\Gal(F/\Q_p)}$, the lemma follows from proposition \ref{cri} and the remark below it.
\end{proof}

Thus we can rewrite $M(\chi,[1,b])$ as the kernel of the composite of the following maps:
\begin{eqnarray} \label{thb}
\theta_b:(H^0(\wtso,\omega^1)\otimes O_E)^{\chi,\Gal(F/F_0)}\to\prod_s (F\otimes_{F_0}H^1_{\crys}(\overbar{U_s}/F_0)\otimes_{\Q_p} E)^{\chi}\stackrel{L_b}{\to} F_2,
\end{eqnarray}
where $L_b:F_1\oplus F_2\to F_2$ is defined as:
\[(f_1,f_2)\mapsto -(1\otimes b)(g_{\varphi}\otimes\varphi\otimes\Id_E)(f_1)+(\varpi^{(p-1)i}\otimes 1)f_2.\] 
To understand the image of $\theta_b$, we define: 
\begin{definition}
\begin{eqnarray}\label{J_1}
J_1&\defeq&\prod_{s':\mbox{even}} H^1_{\crys}(\overbar{U_{s'}}/O_{F_0})^{\chi'}_{\tau} \oplus \prod_{s:\mbox{odd}} H^1_{\crys}(\overbar{U_{s}}/O_{F_0})^{(\chi')^p}_{\bar{\tau}} \subset F_1\\ \label{J_2}
J_2&\defeq&(\varpi^{p^2-1-m}g_{\varphi}\otimes\varphi\otimes\Id_{O_E})(J_1)\subset F_2.
\end{eqnarray}
\end{definition}

\begin{lem}
Under the assumption $p^2-1-m\geq [-mp]$ and $2\leq i \leq p-1$, we have
\[\theta_b((H^0(\wtso,\omega^1)\otimes O_E)^{\chi,\Gal(F/F_0)})\subset J_2\]
\end{lem}

\begin{rem}
We will see that the lemma also holds for $i=1,p$ in the next section.
\end{rem}

\begin{proof}
Let $\omega$ be an element in $(H^0(\wtso,\omega^1)\otimes O_E)^{\chi,\Gal(F/F_0)}$. Write $\omega=(\omega_1,\omega_2)$ as the decomposition into $F_1\oplus F_2$. By proposition \ref{ims}, we have $\omega_1\in\varpi^{p^2-1-m}J_1$. Hence 
\[L_b((\omega_1,0))=-(g_{\varphi}\otimes \varphi \otimes b(\omega_1))\in J_2.\]

It remains to prove that $L_b((0,\omega_2))\in J_2$, or equivalently, $(\varpi^{(p-1)i}\otimes1)\omega_2\in J_2$.
Using the action of $\GL_2(\Q_p)$, we only need to check this for one odd vertex. In other words, it suffices to show the following statement: Let $s$ be an odd vertex and $\omega_{2,s}$ be the image of $\omega$ inside $F\otimes_{F_0}H^1_{\crys}(\overbar{U_{s}}/F_0)^{\chi'}_{\tau}$. Then 
\[(\varpi^{(p-1)i}\otimes 1)\omega_{2,s}\in (\varpi^{p^2-1-m}g_{\varphi}\otimes \varphi \otimes \Id_{O_E})(H^1_{\crys}(\overbar{U_{s}}/O_{F_0})^{(\chi')^p}_{\bar{\tau}}). \]
But this is nothing but the first part of lemma \ref{rm3}.
\end{proof}

By abuse of notation, we use $\theta_b$ to denote the map $(H^0(\wtso,\omega^1)\otimes O_E)^{\chi,\Gal(F/F_0)}\to J_2$. Also we use $\bar{\theta}_b$ to denote the modulo $p$ map of $\theta_b$, that is:
\[\bar{\theta}_b:H^{\chi,F_0}=(H^0(\wtso,\omega^1)\otimes O_E/p)^{\chi,\Gal(F/F_0)} \to J_2/p.\]

Recall that we have an exact sequence (proposition \ref{dmp}):
\begin{eqnarray*}
0 \to \indkg H^0(\overbar{U_{s'_0}},\Omega^1_{U_{s'_0}})^{\chi'}_{\tau} \to H^{\chi,F_0} \to \indkg H^0(\overbar{U_{s'_0}},\Omega^1_{U_{s'_0}})^{(\chi')^p}_{\bar{\tau}}\to 0.  
\end{eqnarray*}

As for $J_2/p$, it's obvious that 
\[J_2/p\simeq\indkg \varphi(H^1_{\crys}(\overbar{U_{s'_0}}/O_{F_0})^{\chi'}_{\tau})/p.\] 
Using lemma \ref{phi1}, the filtration 
\[p\varphi(H^1_{\crys}(\overbar{U_{s'_0}}/O_{F_0})^{\chi'})\subset pH^1_{\crys}(\overbar{U_{s'_0}}/O_{F_0})^{(\chi')^p}\subset \varphi(H^1_{\crys}(\overbar{U_{s'_0}}/O_{F_0})^{\chi'})\]
induces the following exact sequence:
\[0\to H^1(\overbar{U_{s'_0}},\cO_{\overbar{U_{s'_0}}})^{(\chi')^p}\to \varphi(H^1_{\crys}(\overbar{U_{s'_0}}/O_{F_0})^{\chi'})/p\to H^0(\overbar{U_{s'_0}},\Omega^1_{\overbar{U_{s'_0}}})^{(\chi')^p}\to 0. \]
Another way to see this is that $J_2/p$ is canonically isomorphic with $J_1/p$, and $J_1/p$ has the usual exact sequence for de Rham cohomology. In a word, we have:
\[0\to \indkg H^1(\overbar{U_{s'_0}},\cO_{\overbar{U_{s'_0}}})^{(\chi')^p}_{\bar{\tau}}\to J_2/p \to \indkg H^0(\overbar{U_{s'_0}},\Omega^1_{\overbar{U_{s'_0}}})^{(\chi')^p}_{\bar{\tau}}\to 0.\]

\begin{lem} \label{cd6}
Assume $p^2-1-m\geq [-mp]$ and $i\in\{2,\cdots,p-1\}$. Then $\bar{\theta}_b$ induces the following commutative diagram:
\[\begin{CD} 
\indkg H^0(\overbar{U_{s'_0}},\Omega^1_{\overbar{U_{s'_0}}})^{\chi'}_{\tau} @>>> H^{\chi,F_0} @>>> \indkg H^0(\overbar{U_{s'_0}},\Omega^1_{\overbar{U_{s'_0}}})^{(\chi')^p}_{\bar{\tau}} \\
@V\bar{\theta}_{b,1}VV @V\bar{\theta}_{b}VV @VV\bar{\theta}_{b,2}V\\
\indkg H^1(\overbar{U_{s'_0}},\cO_{\overbar{U_{s'_0}}})^{(\chi')^p}_{\bar{\tau}}@>>> J_2/p@>>> \indkg H^0(\overbar{U_{s'_0}},\Omega^1_{\overbar{U_{s'_0}}})^{(\chi')^p}_{\bar{\tau}}.
\end{CD}\]
\end{lem}
\begin{proof}
Let $\omega$ be an element of $(H^0(\wtso,\omega^1)\otimes O_E)^{\chi,\Gal(F/F_0)}$ whose $\modd~p$ reduction lies in $\indkg H^0(\overbar{U_{s'_0}},\Omega^1_{\overbar{U_{s'_0}}})^{\chi'}_{\tau}$. We need to show that
\[\theta_b(\omega)\in \varpi^{p^2-1-m}\indkg pH^1_{\crys}(\overbar{U_{s'_0}}/O_{F_0})^{(\chi')^p}_{\bar{\tau}}\subset J_2.\]

Write $\omega=\omega_{\tau}+\omega_{\bar{\tau}}$ as in the decomposition 
\[(H^0(\wtso,\omega^1)\otimes_{\Q_p} O_E)^{\chi,\Gal(F/F_0)}=H^0(\wtso,\omega^1)_{\tau}^{\chi',\Gal(F/F_0)}\oplus H^0(\wtso,\omega^1)_{\bar{\tau}}^{(\chi')^p,\Gal(F/F_0)}.\] 
It is clear from the construction in the proof of proposition \ref{dmp} that $\omega$ is in $H_1$ modulo $p$ (see the notations there). This means that
\begin{eqnarray} \label{h1c}
\omega_{\tau}|_{\wts} \in pH^0(\wts,\omega^1)^{\chi'}_{\tau} ~(\mbox{resp. } \omega_{\bar{\tau}}|_{\wtp} \in pH^0(\wtp,\omega^1)^{(\chi')^p}_{\bar{\tau}})
\end{eqnarray}
for any odd vertex $s$ (resp. for any even vertex $s'$). Then by lemma \ref{rm3}, we know that the image of $\varpi^{-[-mp]}\omega_{\tau}$ in $H^1_{\crys}(\overbar{U_s}/O_{F_0})^{\chi'}_{\tau}$ actually lies in
\[pH^1_{\crys}(\overbar{U_s}/O_{F_0})^{\chi'}_{\tau}\]
for any odd vertex $s$. Similarly the image of $\varpi^{-[-mp]}\omega_{\bar{\tau}}$ will be in $pH^1_{\crys}(\overbar{U_{s'}}/O_{F_0})^{(\chi')^p}_{\bar{\tau}}$ for any even vertex $s'$. Let $\omega=(\omega_1,\omega_2)$ be the decomposition of $\omega$ into $F_1\oplus F_2$. Then the discussion before indicates that 
\[(\varpi^{(p-1)i}\otimes 1)\omega_2\in \varpi^{p^2-1-m}\indkg pH^1_{\crys}(\overbar{U_{s'_0}}/O_{F_0})^{(\chi')^p}_{\bar{\tau}}\subset J_2.\]

It remains to prove that 
\[(g_{\varphi}\otimes \varphi\otimes \Id_E)(\omega_1)\in \varpi^{p^2-1-m}\indkg pH^1_{\crys}(\overbar{U_{s'_0}}/O_{F_0})^{(\chi')^p}_{\bar{\tau}}\subset J_2.\] 
This follows from lemma \ref{rm4} (the condition in that lemma is satisfied since we have \eqref{h1c}).
\end{proof}

Now we can state the main theorem of this paper.

\begin{thm} \label{mt1}
$\bar{\theta}_{b,1},\bar{\theta}_{b,2}$ are surjective. More precisely, if we identify $H^0(\overbar{U_{s'_0}},\Omega^1_{\overbar{U_{s'_0}}})^{(\chi')^p}_{\bar{\tau}}$ with $(\Sym^{p-1-i}(O_E/p)^2)\otimes \det {}^{i+j}$ (see remark \ref{rsym}), and identify 
\[H^0(\overbar{U_{s'_0}},\Omega^1_{\overbar{U_{s'_0}}})^{\chi'}_{\tau}\simeq H^1(\overbar{U_{s'_0}},\cO_{\overbar{U_{s'_0}}})^{(\chi')^p}_{\bar{\tau}}\] 
with $(\Sym^{i-2}(O_E/p)^2)\otimes \det {}^{j+1}$, where the isomorphism is induced by $\varphi$ (remark \ref{vpi}), then $\bar{\theta}_{b,1},\bar{\theta}_{b,2}$ are given by:
\begin{eqnarray*}
\bar{\theta}_{b,1}:\sigma_{i-2}(j+1) &\to& \sigma_{i-2}(j+1)\\
X&\mapsto& -bX+((-1)^{j+1}\tau(w_1^i))T(X)\\
\bar{\theta}_{b,2}:\sigma_{p-1-i}(i+j) &\to& \sigma_{p-1-i}(i+j)\\
X&\mapsto& X-((-1)^{j+1}\tau(w_1^{-i})b)T(X),
\end{eqnarray*}
where $T$ is the Hecke operator defined in \cite{Bre2}. See the beginning of the paper for its definition.
\end{thm}

We list some direct consquences of this theorem.
\begin{cor}
$\bar{\theta}_b$ is surjective.
\end{cor}

\begin{cor} \label{by1}
$\theta_b:(H^0(\wtso,\omega^1)\otimes O_E)^{\chi,\Gal(F/F_0)}\to J_2$ is surjective and we have the following exact sequence:
\begin{eqnarray}
0\to M(\chi,[1,b]) \to (H^0(\wtso,\omega^1)\otimes O_E)^{\chi,\Gal(F/F_0)}\to J_2\to 0.
\end{eqnarray}
Applying the functor $M\mapsto M^d=\Hom^{\mathrm{cont}}_{O_E}(M,E)$ defined in section \ref{cb}, we get:
\begin{eqnarray}
0\to J_2^d\to ((H^0(\wtso,\omega^1)\otimes O_E)^{\chi,\Gal(F/F_0)})^d\to B(\chi,[1,b])\to 0.
\end{eqnarray}
Notice that the kernel and the middle term of this exact sequence do not depend on $b$. In fact, the unitary representation $J_2^d$ is the completion of $c-\indkg\rho_{\chi^{-1}}$ with respect to the lattice $c-\indkg \rho_{\chi^{-1}}^o$ ,where $\rho_{\chi^{-1}}^o\subset \rho_{\chi^{-1}}$ is an $O_E$-lattice. It is the universal unitary completion of $c-\indkg\rho_{\chi^{-1}}$.
\end{cor}

\begin{proof}
Recall that $B(\chi,[1,b])=(M(\chi,[1,b]))^d$ defined in section \ref{cb}. The surjectivity of $\theta_b$ follows from the surjectivity of $\bar{\theta}_b$ and the fact that $J_2$ and $(H^0(\wtso,\omega^1)\otimes O_E)^{\chi,\Gal(F/F_0)}$ are $p$-adically complete. The explicit description of $J_2^d$ follows from the obvious isomorphism between $J_2$ and $J_1$, which is clearly isomorphic to $\indkg H^1_{\crys}(\overbar{U_{s'_0}}/O_{F_0})^{\chi'^{-1}}_{\tau}$. It is easy to verify that it satisfies the universal property. 
\end{proof}

\begin{cor} \label{mc1}
Under the assumption $p^2-1-m\geq [-mp],i\in\{2,\cdots,p-1\}$, as a representation of $\GL_2(\Q_p)$,
\begin{eqnarray*}
0\to\left\{ \begin{array}{ll}X\in \sigma_{i-2}(j+1),\\ c(\chi,b)X=T(X)\end{array}\right\}
\to M(\chi,[1,b])/p\to 
\left\{\begin{array}{ll}X\in \sigma_{p-1-i}(i+j),\\ X=c(\chi,b)T(X)\end{array}\right\}
\to 0,
\end{eqnarray*}
where $c(\chi,b)=(-1)^{j+1}\tau(w_1^{-i})b\in O_E/p$. Thus $B(\chi,[1,b])$ is non-zero and admissible.
\end{cor}

\begin{rem} \label{mc3}
If we assume $p^2-1-m\leq [-mp],i\in\{2,\cdots,p-1\}$, the same proof will yield a similar exact sequence:
\begin{eqnarray*}
0\to\left\{ \begin{array}{ll}X\in \sigma_{p-1-i}(i+j),\\ X=c(\chi,b)T(X)\end{array}\right\}
\to M(\chi,[1,b])/p\to 
\left\{\begin{array}{ll}X\in \sigma_{i-2}(j+1),\\ c(\chi,b)X=T(X)\end{array}\right\}
\to 0.
\end{eqnarray*}
\end{rem}

\begin{proof}[Proof of Theorem \ref{mt1}]
First we introduce some notations. 
\begin{definition} \label{lastld}
Let $\omega\in (H^0(\wtso,\omega^1)\otimes O_E)^{\chi,\Gal(F/F_0)}$. Then
\begin{enumerate}
\item  $\omega_{\tau}+\omega_{\bar{\tau}}$ will be the decomposition of $\omega$ in 
\[(H^0(\wtso,\omega^1)\otimes O_E)^{\chi}=H^0(\wtso,\omega^1)_{\tau}^{\chi'}\oplus H^0(\wtso,\omega^1)_{\bar{\tau}}^{(\chi')^p}.\] 
We will use $\omega_{\tau,s,\xi}$ (resp. $\omega_{\bar{\tau},s,\xi}$) to denote the restriction of $\omega_{\tau}$ (resp. $\omega_{\bar{\tau}}$) on $\wtxs$, where $s$ is a vertex of the Bruhat-Tits tree and $\xi^{p-1}=-1$. 
\item $\omega=\omega_1+\omega_2$ will be its decomposition into $F_1\oplus F_2$ (see \eqref{f1f2}). For an even (resp. odd) vertex $s'$ (resp. $s$), we define $\omega_{1,s'}$ (resp. $\omega_{1,s}$) as the image of $\omega$ inside $F\otimes_{F_0} H^1_{\crys}(\overbar{U_{s'}}/F_0)^{\chi'}_{\tau}$ (resp. $F\otimes_{F_0} H^1_{\crys}(\overbar{U_{s}}/F_0)^{(\chi')^p}_{\bar{\tau}}$). It is clear that $\omega_{2,s'},\omega_{2,s}$ can be defined similarly. We also use $\omega_{1,s,\xi}$ to denote its image in $F\otimes_{F_0} H^1_{\crys}(\overbar{U_{s,\xi}}/F_0)^{(\chi')^p}_{\bar{\tau}}$ and define $\omega_{1,s',\xi},\omega_{2,s,\xi},\omega_{2,s',\xi}$ similarly.
\end{enumerate}
\end{definition}

In fact, proposition \ref{ims} tells us that for an even (resp. odd) vertex $s'$ (resp. $s$),
\[\omega_{1,s'}\in\varpi^{p^2-1-m}H^1_{\crys}(\overbar{U_{s'}}/O_{F_0})^{\chi'}_{\tau}~(\mbox{resp. }\omega_{1,s}\in\varpi^{p^2-1-m}H^1_{\crys}(\overbar{U_{s}}/O_{F_0})^{(\chi')^p}_{\bar{\tau}}),\]
and
\[\omega_{2,s'}\in\varpi^{[-mp]}H^1_{\crys}(\overbar{U_{s'}}/O_{F_0})^{(\chi')^p}_{\bar{\tau}}~(\mbox{resp. }\omega_{2,s}\in\varpi^{[-mp]}H^1_{\crys}(\overbar{U_{s}}/O_{F_0})^{\chi'}_{\tau}).\]

Now we start to prove the surjectivity of 
\[\bar{\theta}_{b,2}:\indkg H^0(\overbar{U_{s'_0}},\Omega^1_{\overbar{U_{s'_0}}})^{(\chi')^p}_{\bar{\tau}}\to \indkg H^0(\overbar{U_{s'_0}},\Omega^1_{\overbar{U_{s'_0}}})^{(\chi')^p}_{\bar{\tau}}.\]

Consider $[\Id,x^ky^{p-1-i-k}]$ in (See the beginning of the paper for the notations here)
\[\indkg (\Sym^{p-1-i}(O_E/p)^2)\otimes \det {}^{i+j}\simeq \indkg H^0(\overbar{U_{s'_0}},\Omega^1_{\overbar{U_{s'_0}}})^{(\chi')^p}_{\bar{\tau}}.\]  

Let $\bar{\omega}\in H^{\chi,F_0}$ be a lift of $[\Id,x^ky^{p-1-i-k}]$ in the first row of lemma \ref{cd6} and let $\omega\in(H^0(\wtso,\omega^1)\otimes O_E)^{\chi,\Gal(F/F_0)}$ be a lift of $\bar{\omega}$. 

It is clear that we may assume $\omega_{\tau}=0$. Then our choice of $\omega$ implies that 
\begin{lem} \label{os}
Under the identification in \eqref{les'0},
\begin{eqnarray}\label{os'0}
\omega_{\bar{\tau},s'_0,\xi}\equiv \varpi^{[-mp]}\eta^k\e^i\frac{d\e}{\e}~\mod~pH^0(\wtxc,\omega^1)_{\bar{\tau}} \\
\omega_{\bar{\tau},s',\xi}\in pH^0(\wtxs,\omega^1)_{\bar{\tau}}~\mbox{ for any even vertex } s'\neq s'_0.
\end{eqnarray}
\end{lem}

Using this and remark \ref{vorm3}, we know that for any even vertex $s'\neq s_0$,
\[\varpi^{(p-1)i}\omega_{2,s'}\in \varpi^{p^2-1-m}pH^1_{\crys}(\overbar{U_{s'}}/O_{F_0})^{(\chi')^p}_{\bar{\tau}},\]
and considered as elements in $H^0(\overbar{U_{s'_0}},\Omega^1_{\overbar{U_{s'_0}}})\subset H^1_{\dR}(\overbar{U_{s'_0}})$,
\begin{eqnarray*} 
\varpi^{-[-mp]}\omega_{2,s'_0,\xi}\equiv \eta^k\e^i\frac{d\e}{\e}\mod pH^1_{\crys}(\overbar{U_{s'_0}}/O_{F_0})_{\bar{\tau}}^{(\chi')^p}.
\end{eqnarray*}

Similarly, remark \ref{vorm4} tells us that for any odd vertex $s\notin A(s'_0)$,
\[\varphi(\varpi^{-(p^2-1-m)}\omega_{2,s})\in pH^1_{\crys}(\overbar{U_{s}}/O_{F_0})^{(\chi')^p}_{\bar{\tau}}.\]

Hence it is clear from the definition of $\theta_b$ that we may write
\begin{lem}\label{fma}
\begin{eqnarray} 
\bar{\theta}_{b,2}([\Id,x^ky^{p-1-i-k}])=[\Id,v_{s'_0}]+\sum_{s\in A(s'_0)}[g_s^{-1},v_s],
\end{eqnarray}
where $g_s$ is a chosen representative in the coset defined by $s$. Recall that we identify the set of vertices of Bruhat-Tits tree with $\GL_2(\Z_p)\Q_p^\times\setminus\GL_2(\Q_p)$.
\end{lem}

Since $\omega_{1,s'_0}=0$, it follows from \eqref{os'0} that $v_{s'_0}=x^ky^{p-1-i-k}$.

To determine other terms, we recall some results in section \ref{secact}. Recall that $s_0$ is the vertex that  corresponds to $\eta=\e=0$. As a coset, it corresponds to $\GL_2(\Z_p)\Q_p^{\times}\cdot w$, where $w=\w$. Then $\widetilde{\Sigma_{1,O_F,[s'_0,s_0],\xi}}$ is isomorphic to
\[\Spf \frac{O_F[\eta,\zeta,\frac{1}{\eta^{p-1}-1},\frac{1}{\zeta^{p-1}-1},\e,\e']~\widehat{}}
{(\e^{p+1}+v_1w_1^{-1}\xi\frac{\eta^p-\eta}{\zeta^{p-1}-1},\e'^{p+1}+v_1^{-1}w_1\xi\frac{\zeta^p-\zeta}{\eta^{p-1}-1},\e\e'-\varpi^{p-1}\xi)}.\]
such that the following lemma is true.
\begin{lem} \label{aw1}
The action of $w$ sends $\widetilde{\Sigma_{1,O_F,[s'_0,s_0],\xi}}$ to $\widetilde{\Sigma_{1,O_F,[s'_0,s_0],\xi^p}}=\widetilde{\Sigma_{1,O_F,[s'_0,s_0],-\xi}}$. Explicitly, it is given by (see corollary \ref{cact}, recall that $\e=\frac{e}{\varpi}$):
\[\eta\mapsto -\zeta,\zeta\mapsto-\eta, \e\mapsto v_1\e',\e'\mapsto v_1^{-1}\e.\]
\end{lem}

Now we come back to our situation.  Using lemma \ref{ld0}, the restriction of $\omega_{\bar{\tau}}$ on $\widetilde{\Sigma_{1,O_F,[s'_0,s_0]}}$ can be written as:
\[\omega\big|_{\widetilde{\Sigma_{1,O_F,[s'_0,s_0],\xi}}}=\varpi^{[-mp]}f(\eta)\e^i\frac{d\e}{\e}+\varpi^{p^2-1-m}g(\zeta)\e'^{p+1-i}\frac{d\e'}{\e'},\]
where $f(\eta)\in O_{F_0}[\eta,\frac{1}{\eta^{p-1}-1}]~\widehat{},g(\zeta)\in O_{F_0}[\zeta,\frac{1}{\zeta^{p-1}-1}]~\widehat{}$. Since $\omega$ is in the $(\chi')^p$-isotypic component, we must have (using results in section \ref{act}):
\[\omega|_{\widetilde{\Sigma_{1,O_F,[s'_0,s_0],-\xi}}}=\varpi^{[-mp]}f(\eta)\e^i(-1)^{-(i+j)}\frac{d\e}{\e}+\varpi^{p^2-1-m}g(\zeta)\e'^{p+1-i}(-1)^{-(j+1)}\frac{d\e'}{\e'}.\]

By our construction of $\omega$,
\[\omega_{\bar{\tau},s'_0,\xi}=\omega_{\bar{\tau}}|_{\wtxc}\equiv \varpi^{[-mp]}\eta^k\e^i\frac{d\e}{\e}~\modd~pH^0(\wtxc,\omega^1).\]
Hence, 
\begin{lem} \label{feta}
$f(\eta)\equiv \eta^k~\modd~pO_{F_0}[\eta,\frac{1}{\eta^{p-1}-1}]~\widehat{}$.
\end{lem} 

I would like do all the computations on the central component, so we define 
\begin{eqnarray} \label{h_s0}
h_{s_0}=(w^{-1})^{*}(\omega_{\bar{\tau}})\in H^0(\wtso,\omega^1)_{\tau}.
\end{eqnarray}
Then a direct consequence of lemma \ref{aw1} is that (notice that $w$ maps $(-\xi)$-component to $\xi$-component) 

\begin{lem} \label{hs0}
$h_{s_0}|_{\wtxc}=(w^{-1})^{*}(\omega_{\bar{\tau},s_0,-\xi})$ has the form
\[\varpi^{p^2-1-m}\tilde{g}(-\eta)\e^{p+1-i}(-v_1^{-1})^{p+1-i}(-1)^{j+1}\frac{d\e}{\e}+\varpi^{[-mp]}\tilde{f}(-\zeta)\e'^i(-v_1)^i (-1)^{i+j}\frac{d\e'}{\e'},\]
where $\tilde{f}(-\zeta)=\Fr(f(-\zeta))$, applying Frobenius operator on the coefficients, and $\tilde{g}(-\eta)$ is defined similarly. 
\end{lem}

In fact, by lemma \ref{feta}, we know that $\tilde{f}(-\zeta)\equiv (-\zeta)^k~\modd~pO_{F_0}[\zeta,\frac{1}{\zeta^{p-1}-1}]~\widehat{}$.

Now we need to compute the cohomology class of $\varpi^{-(p^2-1-m)}h_{s_0}$ inside $H^1_{\crys}(\overbar{U_{s'_0}}/O_{F_0})_{\tau}$ (modulo $p$). Following the strategy in the proof of proposition \ref{ims} (see the notations there), we may use a $1$-hypercocycle $(\{\omega_s\},\{f_{s_1,s_2}\})$ to represent $h_{s_0}$. Also recall that $f_{s_1,s_2}=f_{s_2}-f_{s_1}$ (all considered as elements in $\varpi^{p^2-1-m}H^0(V_{c,\xi},\cO_{V_{c,\xi}})$). By definition of $\bar{\theta}_{b,2}$, we only need to know the image of $\varphi(\varpi^{-(p^2-1-m)}h_{s_0})$ in
\[H^1_{\dR}(\overbar{U_{s'_0}})_{\bar{\tau}}=H^1_{\crys}(\overbar{U_{s'_0}}/O_{F_0})_{\bar{\tau}}/pH^1_{\crys}(\overbar{U_{s'_0}}/O_{F_0})_{\bar{\tau}}.\]
Hence lemma \ref{phi1} tells us that we only need to know the image of $\varpi^{-(p^2-1-m)}h_{s_0}$ inside
\[H^1(\overbar{U_{s'_0}},\cO_{\overbar{U_{s'_0}}})_{\tau}.\]
In other words, we only concern about the mod $p$ properties of $f_{s}$.

Since $w$ interchanges $s'_0$ and $s_0$, hence $w(s)\neq s'_0$ for any odd vertex $s\neq s_0$. Then it follows from lemma \ref{os} that for any $s\in A(s'_0),s\neq s_0$,
\[h_{s_0}\big|_{\wts}\in pH^0(\wts,\omega^1)_{\tau}.\]
Therefore proof of lemma \ref{rm4} implies that for any $s\in A(s'_0),~s\neq s_0$,
\[\varpi^{-(p^2-1-m)}f_s\in pH^0(V_{c,\xi},\cO_{V_{c,\xi}})_{\tau}.\]
Moreover, lemma \ref{b_0} tells us that (compare lemma \ref{hs0} with \eqref{rrop} and notice that $g(\zeta)$ there is $\tilde{f}(-\zeta)(-v_1)^i (-1)^{i+j}$ here)
\[\varpi^{-(p^2-1-m)}f_{s_0}\equiv \frac{\xi^i\tilde{f}(0)(-v_1)^i (-1)^{i+j}y^{-i}}{i}\mod pH^0(V_{c,\xi},\cO_{V_{c,\xi}})_{\tau}.\]

Recall that the identification of $\overbar{U_{s'_0,\xi}}$ and the special fibre of $\widehat{D_{0,O_{F_0},\xi}}$ is given by:
\[x\mapsto \eta,~y\mapsto\e.\]

\begin{lem}
The image of $\varpi^{-(p^2-1-m)}h_{s_0}$ in $H^1(\overbar{U_{s'_0,\xi}},\cO_{\overbar{U_{s'_0,\xi}}})_{\tau}$ is the following $1$-cocycle $\{f'_{s,s''}\}$ if we use the open covering $\{V_{s,\xi}\}_s$:
\[f'_{s,s''}=f'_{s''}-f'_{s},~\mbox{where }f'_{s_0}=\xi^i\tilde{f}(0)(-v_1)^i (-1)^{i+j}i^{-1}\e^{-i},f'_s=0 \mbox{ for }s\neq s_0.\]
\end{lem}

Now we want to write this cohomology class as a $1$-cocycle $f_{0,\infty}$ of the open covering $\{V_0,V_\infty\}$ (definition \ref{V0i}). But this is already computed in example \ref{kex}:

\begin{lem} \label{end}
The image of $\varpi^{-(p^2-1-m)}h_{s_0}$ in $H^1(\overbar{U_{s'_0,\xi}},\cO_{\overbar{U_{s'_0,\xi}}})_{\tau}$ is the following $1$-cocycle $\{f_{0,\infty}\}$ if we use the open covering $\{V_0,V_{\infty}\}$:
\[f_{0,\infty}=\tilde{f}(0)(-1)^ji^{-1}w_1(v_1\xi)^{i-1}\frac{\e^{p+1-i}}{\eta}.\]
\end{lem}

Thanks to lemma \ref{ef}, a simple computation shows that  
\begin{lem} The image of $\varphi(\varpi^{-(p^2-1-m)}h_{s_0})$ in $H^1_{\dR}(\overbar{U_{s'_0,\xi}})_{\bar{\tau}}$ is 
\[\varphi(\tilde{f}(0)(-1)^ji^{-1}w_1(v_1\xi)^{i-1}\frac{\e^{p+1-i}}{\eta})=f(0)(-1)^{i+j+1}w_1^i \eta^{p-1-i}\e^i\frac{d\e}{\e}\in H^0(\overbar{U_{s'_0,\xi}},\Omega^1_{\overbar{U_{s'_0,\xi}}})_{\bar{\tau}}.\]
\end{lem}

Recall that in the isomorphism 
\[H^0(\overbar{U_{s'_0}},\Omega^1_{U_{s'_0}})^{(\chi')^p} \to (\Sym^{p-1-i}\F_{p^2}^2)\otimes \det {}^{i+j},\]
$\eta^{p-1-i}\e^i\frac{d\e}{\e}$ is identified with $x^{p-1-i}$. By lemma \ref{feta}, $f(0)=1$ if $k=0$ and $f(0)=0$ otherwise. Hence considering the definition of $\bar{\theta}_{b,2}$, it follows from the previous lemma that
\begin{lem}
\[[w,v_{[w^{-1}]}]=\left\{ \begin{array}{ll}[w,(-1)^{j+1}\tau(w_1^{-i})bx^{p-1-i}] & \mbox{if }k=0,\\
0 & \mbox{otherwise}\end{array}\right.\]
\end{lem}

Now we compute the same term of $T([\Id,x^ky^{p-1-i-k}])$ (see the beginning of the paper for notations here):
\[[w,\varphi_r(w^{-1})(x^ky^{p-1-i-k})]=[w,\we\circ \varphi_r(\wo)(x^ky^{p-1-i-k})],\]
which is non-zero if and only if $k=0$. When $k=0$, 
\begin{eqnarray*}
&&[w,\varphi_r(w^{-1})(y^{p-1-i})]\\
&=&[w,\we\circ \varphi_r(\wo)(y^{p-1-i})]\\
&=&[w,\we(y^{p-1-i})]=[w,x^{p-1-i}].
\end{eqnarray*}

\begin{lem} \label{T}
\[T([\Id,x^ky^{p-1-i-k}])=\left\{\begin{array}{ll}
[w,x^{p-1-i}]+\mbox{other terms} & k=0\\
{[}w,0]+\mbox{other terms} & k\neq0
\end{array}\right.\]
\end{lem}

Since $\GL_2(\Z_p)$ acts transitively on $A(s'_0)$, it's clear from the above computation that
\[\bar{\theta}_{b,2}([\Id,x^ky^{p-1-i-k}])=[\Id,x^ky^{p-1-i-k}]-((-1)^{j+1}\tau(w_1^{-i})b)T([\Id,x^ky^{p-1-i-k}]).\]
Notice that $\bar{\theta}_{b,2}$ is $\GL_2(\Q_p)$-equivariant. Therefore,
\[\bar{\theta}_{b,2}=\Id-((-1)^{j+1}\tau(w_1^{-i})b)T.\]

As for $\bar{\theta}_{b,1}$, the computation is almost the same. I omit the details here.
\end{proof}

\section{Computation of \texorpdfstring{$M(\chi,[1,b])/p$}{} (\texorpdfstring{\RNum{2}}{}): \texorpdfstring{$i=1,p$}{}} \label{M2}
In this section, we deal with the case $i=1,p$. We keep the notations used in the last two sections. Now proposition \ref{dmp} becomes:

\begin{prop}
\begin{enumerate}
\item If $i=1$, there exists a $\GL_2(\Q_p)$-equivariant isomorphism:
\[H^{\chi,F_0} \stackrel{\sim}{\to} \indkg H^0(\overbar{U_{s'_0}},\Omega^1_{\overbar{U_{s'_0}}})^{(\chi')^p}_{\bar{\tau}}. \]
\item {}If $i=p$, 
\[H^{\chi,F_0} \stackrel{\sim}{\to} \indkg H^0(\overbar{U_{s'_0}},\Omega^1_{\overbar{U_{s'_0}}})^{\chi'}_{\tau}. \]
\end{enumerate}
\end{prop}

\begin{proof}
Notice that when $i=1,~H^0(\overbar{U_{s'_0}},\Omega^1_{\overbar{U_{s'_0}}})^{\chi'}=0$. So everything follows from proposition \ref{dmp} and remark \ref{opp}.
\end{proof}

In fact, we can see the above isomorphisms in the following way. If $i=p$, for any $\bar{h}\in H^{\chi,F_0}$, the restriction of $\bar{h}_{\tau}$ (resp. $\bar{h}_{\bar{\tau}}$) on $\wtp$ (resp. $\wts$) for an odd (resp. even) vertex $s'$ (resp. $s$) corresponds to a holomorphic differential form on $\overbar{U_{s'}}$ (resp. $\overbar{U_s}$) under the isomorphism in lemma \ref{pss}. Hence we can define the above map. The case $i=1$ is similar.

As we promised before, we have the following
\begin{lem} Assume $i=1$ or $p$, 
\[\theta_b((H^0(\wtso,\omega^1)\otimes O_E)^{\chi,\Gal(F/F_0)})\subset J_2.\]
\end{lem}
\begin{proof}
See \eqref{J_1},\eqref{J_2} for the definitions of $J_1,J_2$. First we assume $i=p$. Then by lemma \ref{phi1}, we have
\[\varphi(H^1_{\crys}(\overbar{U_{s'}}/O_{F_0})^{\chi'})=pH^1_{\crys}(\overbar{U_{s'}}/O_{F_0})^{(\chi')^p}.\]
Thus we may identify $J_2=(\varpi^{p^2-1-m}\otimes\varphi\otimes\Id_{O_E})(J_1)$ with (recall $F_2$ is a $F\otimes_{\Q_p} E$-module)
\[(\varpi^{p^2-1-m}\otimes 1)(\prod_{s':\mbox{even}} pH^1_{\crys}(\overbar{U_{s'}}/O_{F_0})^{\chi'}_{\bar{\tau}} \oplus \prod_{s:\mbox{odd}} pH^1_{\crys}(\overbar{U_{s'}}/O_{F_0})^{(\chi')^p}_{\tau}) \subset F_2.\]

Let $\omega\in (H^0(\wtso,\omega^1)\otimes O_E)^{\chi,\Gal(F/F_0)}$ and $\omega=\omega_1+\omega_2$ be the decomposition of $\omega$ into $F_1\oplus F_2$. By definition $\varphi(\omega_1)\in J_2$. Since $[-mp]+i(p-1)=(p^2-1)+p^2-1-m$ in this case, proposition \ref{ims} implies that $\varpi^{i(p-1)}\omega_2\in J_2$. Hence $\theta_b(\omega)\in J_2$.

Now assume $i=1$, then $\varphi(H^1_{\crys}(\overbar{U_{s'}}/O_{F_0})^{\chi'})=pH^1_{\crys}(\overbar{U_{s'}}/O_{F_0})^{(\chi')^p}$. Hence 
\[J_2=\varpi^{p^2-1-m}(\prod_{s':\mbox{even}} H^1_{\crys}(\overbar{U_{s'}}/O_{F_0})^{\chi'}_{\bar{\tau}} \oplus \prod_{s:\mbox{odd}} H^1_{\crys}(\overbar{U_{s'}}/O_{F_0})^{(\chi')^p}_{\tau}).\]
So lemma follows directly from proposition \ref{ims}.
\end{proof}

Let $\bar{\theta}_b:H^{\chi,F_0}\to J_2/p$ be the mod $p$ map of $\theta_b$. It is clear that 
\[J_2/p\simeq \indkg H^1_{\dR}(\overbar{U_{s'_0}})^{(\chi')^p}_{\bar{\tau}}\simeq\left\{
\begin{array}{cc}
	\indkg H^1(\overbar{U_{s'_0}},\cO_{\overbar{U_{s'_0}}})^{(\chi')^p}_{\bar{\tau}},~i=p\\
	\indkg H^0(\overbar{U_{s'_0}},\Omega^1_{\overbar{U_{s'_0}}})^{(\chi')^p}_{\bar{\tau}},~i=1
\end{array} \right. .\]

We can now state our main results of this section.
\begin{thm} \label{mt1p}
$\bar{\theta}_b$ is surjective. More precisely,
\begin{enumerate}
\item Assume $i=p$. If we consider the following isomorphism induced by $\varphi$ (remark \ref{vpi}):
\[H^0(\overbar{U_{s'_0}},\Omega^1_{\overbar{U_{s'_0}}})^{\chi'}_{\tau}\simeq H^1(\overbar{U_{s'_0}},\cO_{\overbar{U_{s'_0}}})^{(\chi')^p}_{\bar{\tau}}\]
and use remark \ref{rsym} to identify $H^0(\overbar{U_{s'_0}},\Omega^1_{\overbar{U_{s'_0}}})^{\chi'}_{\tau}$ with $(\Sym^{p-2}(O_E/p)^2)\otimes \det {}^{j+1}$, then $\bar{\theta}_b$ is given by:
\begin{eqnarray*}
\bar{\theta}_b:\sigma_{p-2}(j+1)&\to& \sigma_{p-2}(j+1)\\
X&\mapsto& -bX+(-1)^{j+1}\tau(w_1^p)T(X)-bT^2(X).
\end{eqnarray*}
\item Assume $i=1$. If we use remark \ref{rsym} to make the following identification
\[H^0(\overbar{U_{s'_0}},\Omega^1_{\overbar{U_{s'_0}}})^{(\chi')^p}_{\bar{\tau}}\simeq(\Sym^{p-2}(O_E/p)^2)\otimes \det {}^{j+1},\] 
then $\bar{\theta}_b$ is given by:
\begin{eqnarray*}
\bar{\theta}_b:\sigma_{p-2}(j+1)&\to& \sigma_{p-2}(j+1)\\
X&\mapsto& X+(-1)^{j+1}b\tau(w_1^{-1})T(X)+T^2(X).
\end{eqnarray*}
\end{enumerate}
\end{thm}

Just like the previous section, we list some corollaries first.
\begin{cor}
$\bar{\theta}_b$ is surjective.
\end{cor}

\begin{cor} \label{by2}
$\theta_b:(H^0(\wtso,\omega^1)\otimes O_E)^{\chi,\Gal(F/F_0)}\to J_2$ is surjective and we have the following exact sequence:
\begin{eqnarray}
0\to M(\chi,[1,b]) \to (H^0(\wtso,\omega^1)\otimes O_E)^{\chi,\Gal(F/F_0)}\to J_2\to 0.
\end{eqnarray}
Applying the functor $M\mapsto M^d=\Hom^{\mathrm{cont}}_{O_E}(M,E)$, we get:
\begin{eqnarray}
0\to J_2^d\to ((H^0(\wtso,\omega^1)\otimes O_E)^{\chi,\Gal(F/F_0)})^d\to B(\chi,[1,b])\to 0.
\end{eqnarray}
The kernel and the middle term of this exact sequence do not depend on $b$. The kernel $J_2^d$ is the completion of $c-\indkg\rho_{\chi^{-1}}$ with respect to the lattice $c-\indkg \rho_{\chi^{-1}}^o$ ,where $\rho_{\chi^{-1}}^o\subset \rho_{\chi^{-1}}$ is an $O_E$-lattice. It is the universal unitary completion of $c-\indkg\rho_{\chi^{-1}}$.
\end{cor}

\begin{cor} \label{mc2}
Assume $i=p$, as a representation of $\GL_2(\Q_p)$,
\[M(\chi,[1,b])/p\simeq \{X\in \sigma_{p-2}(j+1), -bX+(-1)^{j+1}\tau(w_1^p)T(X)-bT^2(X)=0\}.\]
When $i=1$,
\[M(\chi,[1,b])/p\simeq \{X\in  \sigma_{p-2}(j+1),  X+(-1)^{j+1}b\tau(w_1^{-1})T(X)+T^2(X)=0\}.\]
Thus in any case, $B(\chi,[1,b])$ is non-zero and admissible.
\end{cor}

\begin{proof}[Proof of Theorem \ref{mt1p}]
We only deal with the case $i=1$. The case where $i=p$ can be treated in almost the same way.

Consider $[\Id,x^ky^{p-2-k}]$ as an element in
\[ \indkg (\Sym^{p-2}(O_E/p)^2)\otimes \det {}^{j+1}\simeq \indkg H^0(\overbar{U_{s'_0}},\Omega^1_{\overbar{U_{s'_0}}})^{(\chi')^p}_{\bar{\tau}}.\]
Let $\omega\in(H^0(\wtso,\omega^1)\otimes O_E)^{\chi,\Gal(F/F_0)}$ be a lift of $[\Id,x^ky^{p-2-k}]$. As before, we may assume $\omega_{\tau}=0$. It is clear from our construction that for any even vertex $s'\neq s'_0$,
\begin{eqnarray}\label{i1ae}
\omega_{\bar{\tau},s',\xi}\in pH^0(\wtxp,\omega^1)_{\bar{\tau}}.
\end{eqnarray}
Hence for any odd vertex $s\notin A(s'_0)$, 
\begin{eqnarray}\label{i1ao}
\omega_{\bar{\tau},s,\xi}\in pH^0(\wtxs,\omega^1)_{\bar{\tau}}.
\end{eqnarray}
This follows from remark \ref{oeoe} and the fact $H^0(\overbar{U_s},\Omega^1_{\overbar{U_s}})^{(\chi')^p}=0$. Thus using lemma \ref{rm4} and remark \ref{vorm4}, we know that $\bar{\theta}_b([\Id,x^ky^{p-2-k}])$ must be of the following form:
\begin{lem} \label{u_s'0s0s''0}
\[\bar{\theta}_b([\Id,x^ky^{p-2-k}])=[\Id,u_{s'_0}]+\sum_{s\in A(s'_0)}[g_s^{-1},u_s]+\sum_{s'\in A^2(s'_0)}[g_{s'}^{-1},u_{s'}],\]
where $A^2(s'_0)=\{s'\in A(s),s\in A(s'_0),s'\neq s'_0\}$. 
\end{lem}

First we compute $[\Id,u_{s'_0}]$. It suffices to compute the image of $\varpi^{-[-pm]}\omega_{\bar{\tau}}$ in 
\[H^1_{\dR}(\overbar{U_{s'_0}})_{\bar{\tau}}^{(\chi')^p}=H^0(\overbar{U_{s'_0}},\Omega^1_{\overbar{U_{s'_0}}})_{\bar{\tau}}^{(\chi')^p}. \]
As before, on $\widetilde{\Sigma_{1,O_F,[s'_0,s_0],\xi}}$, we can write (use a variant of lemma \ref{ld0} and notice that $\omega_{\bar{\tau}}$ is in the $(\chi')^p$-isotypic component):
\[\omega_{\bar{\tau}}|_{\widetilde{\Sigma_{1,O_F,[s'_0,s_0],\xi}}}=\varpi^{[-mp]}f(\eta)\e\frac{d\e}{\e}+\varpi^{p^2-1-m}g(\zeta)\e'^p\frac{d\e'}{\e'},\]
where $f(\eta)\in O_{F_0}[\eta,1/(\eta^{p-1}-1)]~\widehat{},~g(\zeta)\in O_{F_0}[\zeta,1/(\zeta^{p-1}-1)]~\widehat{}$. As usual, we identify $\widetilde{\Sigma_{1,O_F,[s'_0,s_0],\xi}}$ with
\[\Spf \frac{O_F[\eta,\zeta,\frac{1}{\eta^{p-1}-1},\frac{1}{\zeta^{p-1}-1},\e,\e']~\widehat{}}
{(\e^{p+1}+v_1w_1^{-1}\xi\frac{\eta^p-\eta}{\zeta^{p-1}-1},\e'^{p+1}+v_1^{-1}w_1\xi\frac{\zeta^p-\zeta}{\eta^{p-1}-1},\e\e'-\varpi^{p-1}\xi)}.\]
Our choice of $\omega$ implies that 
\begin{lem} \label{lemf}
\[f(\eta)\equiv \eta^k~\modd~pO_{F_0}[\eta,\frac{1}{\eta^{p-1}-1}]~\widehat{}.\] 
\end{lem}
Now restricted on $\widetilde{\Sigma_{1,O_F,s_0,\xi}}$, 
\[\omega_{\tau,s_0,\xi}=\omega_{\tau}|_{\widetilde{\Sigma_{1,O_F,s_0,\xi}}}= \varpi^{p^2-1-m}(-\xi\e'^{-2} f(\frac{p}{\zeta})d\e'+g(\zeta)\e'^{p-1}d\e').\]
By \eqref{i1ae}, we know that for any $s'\in A(s_0)$ that is not $s'_0$,
\[\omega_{\bar{\tau},s',\xi}\in pH^0(\wtxp,\omega^1)_{\bar{\tau}}.\]
Then remark \ref{oeoe} implies that the reduction of $\varpi^{-(p^2-1-m)}\omega_{\bar{\tau},s_0,\xi}$ modulo $p$, as a meromorphic differential form on $\overbar{U_{s_0,\xi}}$, can only have poles at $\zeta=\e'=0$. Here we identify $\overbar{U_{s_0,\xi}}$ with the projective curve in $\PP^2_{\F_{p^2}}$ defined by 
\[\e'^{p+1}=v^{-1}w_1\xi(\zeta^p-\zeta)\]
Therefore the only possible pole must come from $-\xi f(\frac{p}{\zeta})\frac{d\e'}{\e'^2}$. Notice that by lemma \ref{lemf}, this term is non-zero modulo $p$ if and only if $k=0$. Thus when $k\neq 0$, the reduction of $\omega_{\bar{\tau},s_0,\xi}$ is a holomorphic differential form on $\overbar{U_{s_0,\xi}}$. But $H^0(\overbar{U_{s_0}},\Omega^1_{\overbar{U_{s_0}}})^{(\chi')^p}=0$, hence $g(\zeta)$ has to be zero modulo $p$ in this case. Therefore we have proved the following
\begin{lem}\label{kn0}
If $k\neq 0$, then $g(\zeta)\in pO_{F_0}[\zeta,1/(\zeta^{p-1}-1)]~\widehat{}$, and
\[\omega_{\bar{\tau},s_0,\xi}\in pH^0(\widetilde{\Sigma_{1,O_F,s_0,\xi}},\omega^1)_{\bar{\tau}}\]
\end{lem}

When $k=0$. Rewrite 
\begin{eqnarray*}
\e'^{-2} f(\frac{p}{\zeta})\equiv\frac{1}{\e'^2}=\frac{\e'^{p-1}}{\e'^{p+1}}\equiv \frac{\e'^{p-1}}{v_1^{-1}w_1\xi(\zeta^p-\zeta)}\equiv -\frac{\e'^{p-1}}{v_1^{-1}w_1\xi \zeta}+\frac{\e'^{p-1}\zeta^{p-2}}{v_1^{-1}w_1\xi (\zeta^{p-1}-1)}\\
(\modd~pO_{F_0}[\e',\zeta,\frac{1}{\zeta^p-\zeta}]~\widehat{}/(\e'^{p+1}+v_1^{-1}w_1\xi\frac{\zeta^p-\zeta}{(p/\zeta)^{p-1}-1}))
\end{eqnarray*}
Thus 
\[\varpi^{-(p^2-1-m)}\omega_{\bar{\tau},s_0,\xi}\equiv \frac{\e'^{p-1}d\e'}{v_1^{-1}w_1\zeta}+(-\frac{\e'^{p-1}\zeta^{p-2}}{v_1^{-1}w_1 (\zeta^{p-1}-1)}+g(\zeta)\e'^{p-1})d\e'\mod pH^0(\widetilde{\Sigma_{1,O_F,s_0,\xi}},\omega^1)_{\bar{\tau}}.\]
Notice that the first term $\frac{\e'^{p-1}d\e'}{v_1^{-1}w_1\zeta}$ only has a pole at $\e'=\zeta=0$ and the second term is holomorphic at this point. Therefore the second term (modulo $p$) is a holomorphic differential form on $\overbar{U_{s_0}}$, which has to be zero since it is in $H^0(\overbar{U_{s_0}},\Omega^1_{\overbar{U_{s_0}}})^{(\chi')^p}=0$. Hence, 
\begin{lem} \label{ki0}
When $k=0$,
\begin{eqnarray*}
\omega_{\bar{\tau},s_0,\xi}&\equiv&\varpi^{p^2-1-m}\frac{\e'^{p-1}d\e'}{v_1^{-1}w_1\zeta}~\modd~pH^0(\widetilde{\Sigma_{1,O_F,s_0,\xi}},\omega^1)_{\bar{\tau}}\\
g(\zeta)&\equiv& \frac{\zeta^{p-2}}{v_1^{-1}w_1(\zeta^{p-1}-1)}~\modd~pO_{F_0}[\zeta,\frac{1}{\zeta^{p-1}-1}]~\widehat{}.
\end{eqnarray*}
\end{lem}

A direct corollary of lemma \ref{kn0} and lemma \ref{ki0} is that
\begin{lem} \label{i1g0}
For any $k$, we always have $g(0)\in pO_{F_0}$.
\end{lem} 

Now we try to compute the image of $\omega$ inside $\varpi^{p^2-1-m}H^1_{\crys}(\overbar{U_{s'_0}}/O_{F_0})_{\bar{\tau}}$ (modulo $p$). As we did in the previous section, we can use a $1$-hypercocycle $(\{\omega_s\},\{f_{s_1,s_2}\})$ to represent this cohomology class. Moreover, there exists $\{f_s\}_{s\in A(s'_0)}$, where $f_s\in \varpi^{p^2-1-m}H^0(V_{c,\xi},\cO_{V_{c,\xi}})$ such that $f_{s_1,s_2}=f_{s_2}-f_{s_1}$ and $\omega_s=\omega-df_s$. See the proof of proposition \ref{ims} for the notations here.

From lemma \ref{i1g0}, we know that $g(0)$ is divisible by $p$. Therefore lemma \ref{b_0} tells us that 
\[\varpi^{-(p^2-1-m)}f_{s_0}\in pH^0(V_{c,\xi},\cO_{V_{c,\xi}}).\]
Using the action of $\GL_2(\Z_p)$, it is easy to see that the above inclusion is also true for other vertex $s\in A(s'_0)$. Hence all $f_{s_1,s_2}$ are divisible by $p$ and all $\omega_s$ is congruent to $\omega$ modulo $p$. This certainly implies that the image of $\varpi^{-(p^2-1-m)}\omega$ in $H^1_{\dR}(\overbar{U_{s'_0}})_{\bar{\tau}}^{(\chi')^p}=H^0(\overbar{U_{s'_0}},\Omega^1_{\overbar{U_{s'_0}}})_{\bar{\tau}}^{(\chi')^p}$ is
\[\varpi^{-(p^2-1-m)}\omega\equiv\eta^kd\e\]
considered as a differential form by lemma \ref{pss}. In other words, 
\begin{lem}\label{i1s'0}
$u_{s'_0}=x^ky^{p-2-k}$.
\end{lem}

Next we compute $u_{s_0}$. As we did in the previous section, we define 
\begin{eqnarray} \label{hs02}
h'_{s_0}=(w_1^{-1})^*(\omega_{\bar{\tau}})\in H^0(\wtxc,\omega^1)_{\tau}^{\chi'}.
\end{eqnarray}
Hence lemma \ref{hs0} tells us that 
\[h'_{s_0}|_{\wtxc}=\varpi^{p^2-1-m}\tilde{g}(-\eta)\e^{p}(-v_1^{-1})^{p}(-1)^{j+1}\frac{d\e}{\e}-\varpi^{[-mp]}\tilde{f}(-\zeta)\e'v_1 (-1)^{j+1}\frac{d\e'}{\e'},\]
where $\tilde{f}(-\zeta)=\Fr(f(-\zeta))$, and $\tilde{g}(-\eta)$ is defined similarly. 

We need to compute the image of $\varpi^{-(p^2-1-m)}h'_{s_0}$ in $H^1_{\dR}(\overbar{U_{s'_0}})_{\tau}^{\chi'}=H^1(\overbar{U_{s'_0}},\cO_{\overbar{U_{s'_0}}})^{\chi'}_{\tau}$. Now the argument becomes exactly the same as the proof of theorem \ref{mt1}: By abuse of notations, we use a $1$-hypercocycle $(\{\omega_s\},\{f_{s_1,s_2}\})$ to represent the cohomology class of $h'_{s_0}\in \varpi^{p^2-1-m}H^1_{\crys}(\overbar{U_{s'_0}}/O_{F_0})_{\tau}^{\chi'}$. Also there exists $\{f_s\}$ such that $f_{s_2,s_1}=f_{s_2}-f_{s_1}$. By \eqref{i1ae} and the lemma \ref{rm4}, we know that all $f_s$ is divisible by $p$ for $s\neq {s_0}$. As for $f_{s_0}$, we can compute it using lemma \ref{b_0} and  lemma \ref{lemf}. We omit all the details here but just refer to the arguments from lemma \ref{hs0} to lemma \ref{end} in the proof of theorem \ref{mt1}.
\begin{lem}\label{i1s0}
\[u_{s_0}=u_{[w^{-1}]}=\left\{\begin{array}{ll}
(-1)^{j+1}b\tau(w_1^{-1})x^{p-2} & k=0\\
0 & k\neq0.
\end{array}\right.\]
\end{lem}

Finally we come to the case $s'\in A^2(s'_0)$, which does not exist when $i\in\{2,\cdots,p-1\}$. 
\begin{definition} \label{s''0}
We define $s''_0\in A(s_0)$ as the vertex that corresponds to the coset 
\[\GL_2(\Z_p)\Q_p^{\times}\ws^{-1}\in\GL_2(\Z_p)\Q_p^{\times}\setminus \GL_2(\Q_p).\]
\end{definition}

When $k\neq 0$, lemma \ref{kn0} tells us that $\omega_{\bar{\tau},s_0,\xi}\in pH^0(\widetilde{\Sigma_{1,O_F,s_0,\xi}},\omega^1)_{\bar{\tau}}$. Therefore by lemma \ref{rm4}, the cohomology class of $\varpi^{-[-mp]}\omega_{\bar{\tau}}$ in $H^1_{\crys}(\overbar{U_{s''_0}}/O_{F_0})_{\bar{\tau}}$ is inside $pH^1_{\crys}(\overbar{U_{s''_0}}/O_{F_0})_{\bar{\tau}}$.

\begin{lem} \label{kn02}
When $k\neq 0$, $u_{s''_0}=0$.
\end{lem}

So we assume $k=0$ from now on.

Notice that
\[\w^{-1}\ws=\begin{pmatrix} 1&0\\-1&1\end{pmatrix}\w^{-1}.\]
Hence the (right) action of $\ws$ fixes the vertex $s_0$ and sends $s''_0$ to $s'_0$. This clearly implies that $\ws$ sends the edge $[s''_0,s_0]$ to $[s'_0,s_0]$. In other words, we get an isomorphism:
\[\Psi_{s'_0,s''_0}:\widetilde{\Sigma_{1,O_F,[s''_0,s_0]}}\stackrel{\sim}{\longrightarrow}\widetilde{\Sigma_{1,O_F,[s'_0,s_0]}},
\]
which we denote by $\Psi_{s'_0,s''_0}$. Restrict $\Psi_{s'_0,s''_0}$ on $\widetilde{\Sigma_{1,O_F,s_0}}$, we thus get an automorphism of $\widetilde{\Sigma_{1,O_F,s_0}}$. As usual, we identify $\widetilde{\Sigma_{1,O_F,s_0,\xi}}$ with 
\[\Spf O_F[\zeta,\e',\frac{1}{\zeta^p-\zeta}]/(\e'^{p+1}+v_1^{-1}w_1\xi\frac{\zeta^p-\zeta}{(p/\zeta)^{p-1}-1})~\widehat{}.\]
To see $\Psi_{s'_0,s''_0}$ explicitly on it, we use $\w$ to send $\widetilde{\Sigma_{1,O_F,s_0,-\xi}}$ to $\wtxc$ and then apply the results in section \ref{act}. An easy computation shows that 
\begin{lem} \label{psi'''}
$\Psi_{s'_0,s''_0}|_{\widetilde{\Sigma_{1,O_F,s_0,\xi}}}$ is:
\[\zeta\mapsto \zeta+1,~\e'\mapsto \e'~\modd~pO_F[\zeta,\e',\frac{1}{\zeta^p-\zeta}]/(\e'^{p+1}+v_1^{-1}w_1\xi\frac{\zeta^p-\zeta}{(p/\zeta)^{p-1}-1})~\widehat{}.\]
\end{lem}

Now consider 
\begin{eqnarray} \label{hs''0}
h_{s''_0}\defeq(\ws^{-1})^*(\omega_{\bar{\tau}})\in H^0(\wtso,\omega^1)^{(\chi')^p}_{\bar{\tau}}.
\end{eqnarray}
On $\widetilde{\Sigma_{1,O_F,[s'_0,s_0],\xi}}$, it can be written as:
\[h_{s''_0}|_{\widetilde{\Sigma_{1,O_F,[s'_0,s_0],\xi}}}=\varpi^{[-mp]}f_1(\eta)\e\frac{d\e}{\e}+\varpi^{p^2-1-m}g_1(\zeta)\e'^p\frac{d\e'}{\e'}.\]

By our construction (see \eqref{i1ae}), $\omega_{\bar{\tau},s''_0,\xi}\in pH^0(\widetilde{\Sigma_{1,O_F,s''_0,\xi}},\omega^1)_{\bar{\tau}}$. Hence,
\[h_{s''_0}\big|_{\wtxc}\in pH^0(\wtxc,\omega^1)_{\bar{\tau}}.\]

This means that
\begin{lem}
$f_1(\eta)\in pO_{F_0}[\eta,\frac{1}{\eta^{p-1}-1}]~\widehat{}$.
\end{lem}

Restrict $h_{s''_0}$ on $\wtwx$. Then we have
\[h_{s''_0}\big|_{\wtwx}\equiv \varpi^{p^2-1-m}g_1(\zeta)\e'^p\frac{d\e'}{\e'}\mod pH^0(\wtwx,\omega^1)_{\bar{\tau}}.\]

By definition, $h_{s''_0}=(\ws^{-1})^*(\omega_{\bar{\tau}})$. Hence $\Psi_{s'_0,s''_0}$ maps $\omega_{\bar{\tau}}\big|_{\wtwx}$ to $h_{s''_0}\big|_{\wtwx}$. Thanks to lemma \ref{psi'''}, we can write down this map explicitly (after reducing modulo $p$). Recall that an explicit expression of $\omega_{\bar{\tau}}\big|_{\wtwx}$ is given in lemma \ref{ki0}. Thus a simple computation gives us the following
\begin{lem}When $k=0$,
\[g_1(\zeta)\equiv \frac{1}{v_1^{-1}w_1(\zeta-1)}~\modd~O_{F_0}[\zeta,\frac{1}{\zeta^{p-1}-1}]~\widehat{}.\] 
\end{lem}

With this lemma in hand, we can compute the image of $\varpi^{-(p^2-1-m)}h_{s''_0}$ in $H^1_{\dR}(\overbar{U_{s'_0}})^{(\chi')^p}_{\bar{\tau}}=H^0(\overbar{U_{s'_0}},\cO_{\overbar{U_{s'_0}}})^{(\chi')^p}_{\bar{\tau}}$. We note that $h_{s''_0}\big|_{\wts}\in H^0(\wts,\omega^1)_{\bar{\tau}}$ for any $s\in A(s'_0)$ that is not $s_0$. So the computation is exactly the same as the case when we compute $u_{s_0}$. I omit the details here. The final result is that
\begin{lem}\label{ki02}
When $k=0$,~$u_{s''_0}=-x^{p-2}$.
\end{lem}

We need to compute the same term of $T^2([\Id,x^ky^{p-2-k}])$. Assume $k=0$. When $k\neq 0$, it's easy to see this term is zero. We already computed that
\[T([\Id,y^{p-2}])=[\w,x^{p-2}]+\mbox{other terms}.\]
Since
\[\ws=\w \begin{pmatrix}0&p^{-1}\\-1&-p^{-1} \end{pmatrix},\]
by definition we have,
\[T([\w,x^{p-2}])=[\ws,\varphi_{p-2}(\begin{pmatrix}0&p^{-1}\\-1&p^{-1}\end{pmatrix}^{-1})(x^{p-2})].\]
Write 
\[\begin{pmatrix}0&p^{-1}\\-1&-p^{-1}\end{pmatrix}^{-1}=p\begin{pmatrix}0&-1\\1&0\end{pmatrix}\wo \begin{pmatrix}1&0\\1&1\end{pmatrix}.\]
Then
\begin{eqnarray*}
\varphi_{p-2}(\begin{pmatrix}0&p^{-1}\\-1&p^{-1}\end{pmatrix}^{-1})(x^{p-2})&=&\varphi_{p-2}(\begin{pmatrix}0&-1\\1&0\end{pmatrix}\wo)((x+y)^{p-2})\\
&=&\varphi_{p-2}(\begin{pmatrix}0&-1\\1&0\end{pmatrix})(y^{p-2})\\
&=&-x^{p-2}.
\end{eqnarray*}
Hence,
\begin{lem} \label{TT}
\[T^2([\Id,y^{p-2}])=\left\{\begin{array}{ll}[\ws,-x^{p-2}]+\mbox{other terms}~&k=0\\
{[}\ws,0]+\mbox{other terms}~&k\neq0. \end{array}\right.\]
\end{lem}
Combining the results of lemma \ref{i1s'0}, \ref{i1s0}, \ref{kn02}, \ref{ki02} and lemma \ref{T}, \ref{TT} together:
\[\bar{\theta}_b(X)=X+(-1)^{j+1}b\tau(w_1^{-1})T(X)+T^2(X).\]
\end{proof}

\section{A conjecture on \texorpdfstring{$B(\chi,[1,b])$}{}}
In the previous two sections, we have proved the admissibility of $B(\chi,[1,b])$ and explicitly compute its residue representation (see corollary \ref{mc1}, corollary \ref{mc2} and remark \ref{mc3}). Recall that for each data $(\chi,[1,b])$, we associate a two dimensional Galois representation $V_{\chi,[1,b]}$ (proposition \ref{sav}) and prove that $B(\chi,[1,b])$ is a completion of the smooth representation   $c-\indkg\rho_{\chi^{-1}}$ with respect to the lattice $\Theta(\chi,[1,b])$ (proposition \ref{completion}). Up to some twist, this smooth representation, via the classical Local Langlands correspondence for $\GL_2$, corresponds to the Weil-Deligne representation associated to $V_{\chi,[1,b]}^{\vee}$ in \cite{Fon}. It is natural to make the following

\begin{conj}
Up to a twist of some character, $B(\chi,[1,b])$ is isomorphic to $\Pi(V_{\chi,[1,b]}^\vee)$ as a Banach space representation of $\GL_2(\Q_p)$, where $\Pi(V_{\chi,[1,b]}^\vee)$ is defined via the $p$-adic local Langlands correspondence for $\GL_2(\Q_p)$ (see \cite{Colm} and \cite{CDP}). 
\end{conj}

The evidence for this conjecture is that we can verify this conjecture when we modulo  $\varpi_E$,  the uniformizer of $E$, namely:

\begin{thm} \label{modpc}
Up to a twist by some character and semi-simplification, $\Theta(\chi,[1,b])/\varpi_E$, via the semi-simple modulo $p$ Langlands correspondence defined by Breuil (see \cite{Bre3} or \cite{Bre2}), corresponds to the residue representation of $V_{\chi,[1,b]}^\vee$ with respect to some lattice inside.
\end{thm}

\begin{proof}
The residue representation of $V_{\chi,[1,b]}$ is computed in Theorem 6.12. of \cite{Sav}. I almost follow his notations except that his $w$ there is my $u_x$ here. $\Theta(\chi,[1,b])/\varpi_E \Theta(\chi,[1,b])$ is computed in corollary \ref{mc1}, corollary \ref{mc2} and remark \ref{mc3}. A direct computation shows that they indeed match via Breuil's dictionary. I omit the details here.
\end{proof}

\begin{rem}
There is some duality involved in the conjecture. The reason is that we are using de Rham cohomology rather than its dual, de Rham cohomology with compact support.
\end{rem}

\begin{rem}
It seems that this conjecture follows from the work \cite{DLB} of Dospinescu and Le Bras by taking the universal unitary completion in their construction. The interested reader is referred to their paper.
\end{rem}

\subsection*{Index of notation}
\begin{list}{}{}
\item[Section \ref{bfd}:] $\Omega,\widehat{\Omega},\widehat{\Omega_e},\widehat{\Omega_s},s'_0,X_n,\X_n,\Sigma_n,T_0,T_1$
\item[Section \ref{rayppp}:] $\chi_1,\chi_2,\cL_i,c_i,d_i$
\item[Section \ref{soxfd}:]$\lambda_1,\lambda_{\cL_1},\widetilde{\lambda_{\cL_1}}\widehat{\Sigma_1^{nr}}$ (\ref{defsn})
\item[Section \ref{secact}:] $\widehat{\Sigma_{1,e}^{nr}},\widehat{\Sigma_{1,s}^{nr}},w_1,v_1,\widehat{\Sigma_1}$ (\ref{descent})
\item[Section \ref{ssm}:] $\widehat{\Sigma_1'},F_0,F,\varpi$ (\ref{defoF}), $\widehat{\Sigma_{1,O_F}},\widehat{\Sigma_{1,O_F,s}},\wtso,\wts,\wtxs,\widehat{\Sigma_{1,O_F}^{(0)}},\cdots,g_{\varphi}$ (\ref{defogv})
\item[Section \ref{act}:]$\tomega_2$ (\ref{defoto})
\item[Section \ref{adoc}:] $\Sigma_{1,F},\Sigma_{1,F}^{(0)},U_n$
\item[Section \ref{dR}:] $\Omega^i_{\Sigma_{1,F}},H^i_{\dR}(\Sigma_{1,F}),U_e,U_s$ (\ref{tnbhd}), $(s,\xi)$ (\ref{defosx}), $\overbar{U_s},\overbar{U_{s,\xi}},\overbar{U_{s,\xi}^0},\overbar{U_s^0},U_{s,\xi}$ (\ref{u_s})
\item[Section \ref{F_0s}:]$F_{1,\xi},D_{1,\xi},\psi_{s,\xi},\widehat{D_{1,O_{F_0},\xi}},\rho_{\chi}$ (\ref{deforho}), $D_{\crys,\chi},m,c_x$ (\ref{dcrys})
\item[Section \ref{Gal}:] $i,j,D_{\chi,[a,b]},V_{\chi,[1,b]}$ (\ref{sav})
\item[Section \ref{cb}:] $\omega^1,M(\chi,[1,b]),B(\chi,[1,b])$
\item[Section \ref{cmodp}:] $H^{(0),\chi,\Q_p},H^{\chi,F_0}_{*},H^{\chi',F_0}_{*},H^{\chi',F_0}_{*,?}$ (\ref{defoH}), $A(s)$ (\ref{defoAs})
\item[Section \ref{M1}:]$F_{0,\xi},D_{0,\xi},\psi_{s',\xi},\widehat{D_{0,O_{F_0},\xi}},V_{s,\xi},V_{c,\xi}$ (\ref{vsx}), $W_{s,\xi},Z_{s,\xi}$ (\ref{defoWZ}), $w_s,f_{s_1,s_2},f_s$ (\ref{fs12}), $A(s'_0)$ (\ref{defoAs'0}), $V_0,V_{\infty}$ (\ref{V0i}), $F_1,F_2$ (\ref{f1f2}), $\theta_b$ (\ref{thb}), $J_1,J_2$ (\ref{J_1}), $\bar{\theta}_b,\bar{\theta}_{b,1},\bar{\theta}_{b,2}$ (\ref{cd6}), $v_s$ (\ref{fma}), $h_{s_0}$ (\ref{h_s0}), $\omega_{\tau},\omega_{\bar{\tau}},\omega_{\tau,s,\xi},\omega_{1},\omega_2,\omega_{1,s},\omega_{1,s,\xi},\cdots$ (\ref{lastld})
\item[Section \ref{M2}:] $u_{s'_0},u_{s_0},u_{s''_0}$ (\ref{u_s'0s0s''0}), $s''_0$ (\ref{s''0}), $h'_{s_0}$ (\ref{hs02}), $h_{s''_0}$ (\ref{hs''0})
\end{list}

\end{document}